%% file: main.tex
\documentclass[12pt]{amsart}

\input{packages}

\input{commands}

\UseRawInputEncoding

\title{Categorification of local relative Langlands duality}

\author{Yuta Takaya}
\address{Graduate School of Mathematical Sciences, The University of Tokyo, 3-8-1 Komaba, Meguro-ku, Tokyo 153-8914, Japan}
\email{takaya@ms.u-tokyo.ac.jp}

\author{Milton Lin}
\address{Department of Mathematics, Johns Hopkins University, Baltimore, MD, U.S.A.}
\email{clin130@jh.edu}

\begin{document}
\begin{abstract}
    We formulate the normalized period conjecture proposed by Ben-Zvi, Sakellaridis and Venkatesh in the framework of the categorical local Langlands correspondence and study its relation to distinction problems. Motivated by the work of Feng and Wang in the geometric setting, we verify the conjecture for the Iwasawa-Tate and Hecke periods, assuming the existence of the categorical local Langlands correspondence for $\GL_2$ with the Eisenstein compatibility. 
\end{abstract}

\maketitle
\setcounter{tocdepth}{2}

\begingroup
%\small 
\tableofcontents
\endgroup

\section*{Introduction}

\subsection{Background}

In the relative Langlands program, the relation between periods and $L$-functions is studied in several contexts: the numerical relation between automorphic period integrals and special $L$-values over global fields, and distinction problems, the characterization of distinguished representations in terms of $L$-parameters, over local fields. 

In \cite{BZSV}, Ben-Zvi, Sakellaridis and Venkatesh proposed \textit{relative Langlands duality}, aiming at unifying this relation in the categorical framework. 
%The whole story is motivated in the framework of $4$-dimensional arithmetic quantum field theories. 
For a connected reductive group $G$ and its dual $\widehat{G}$, they speculated that there should be an equivalence
\[
    \text{automorphic theory } \cl{A}_{G} \cong \text{spectral theory } \cl{B}_{\widehat{G}}
\]
of $4$-dimensional arithmetic quantum field theories, and the relation between periods and $L$-functions should be an instance of an equivalence of boundary theories
\[
    \Theta_{M} \in \cl{A}_G \leftrightarrow \cl{L}_{\widehat{M}} \in \cl{B}_{\widehat{G}}
\]
associated to dual hyperspherical varieties $(G, M) \leftrightarrow (\widehat{G}, \widehat{M})$. They applied this philosophy to global fields and derived general numerical conjectures on period integrals and $L$-values. 
%For each $2$-dimensional object $\Sigma$ in \cite[Table 1.3.1]{BZSV}, the evaluation at $\Sigma$ should correspond to the (categorical) Langlands correspondence $\cl{A}_G(\Sigma) \cong \cl{B}_{\widehat{G}}(\Sigma)$. 

In the present paper, we carry out analogous constructions for non-archimedean local fields %and formulate the \textit{normalized period conjecture} 
in the geometric framework of \cite{FS24}, which can be seen as the boundary theory attached to polarized dual hyperspherical varieties\footnote{It is often \textit{not} the case that dual hyperspherical varieties are polarized.} $(G,T^*X) \leftrightarrow (\widehat{G},T^*\widehat{X})$. We find that the resulting framework provides a natural geometric setting for studying distinction problems. 

\subsection{Formulation}

Let $F$ be a non-archimedean local field with residual characteristic $p$. Let $G$ be a connected quasi-split reductive group over $F$. Take another prime $\ell \neq p$. Throughout the introduction, we fix a coefficient ring to $\Qla$ and choose $\sqrt{q} \in \Qla$. Let $\LG$ be the $L$-group over $\Qla$. In \cite{FS24}, the categorical local Langlands correspondence is formulated as a categorical equivalence
\[
    \bb{L}_G \colon \cl{D}_\lis(\Bun_G, \Qla) \cong \IndCoh(\Par_{\LG}). 
\]
Here, $\Bun_G$ is a $v$-stack classifying $G$-torsors on the Fargues-Fontaine curve, and $\Par_{\LG}$ is the stack of $L$-parameters with values in $\LG$. The existence of such an equivalence remains widely open; the only fully established case in the literature is that of tori (see \cite{Zou24}). 
%the choice of $\sqrt{q} \in \Lambda$ is an analogue of a square root of the canonical sheaf in the geometric setting \cite{BZSV}. 

\subsubsection{The $\cl{A}$-side}

First, we explain the story on the $\cl{A}$-side. Let $k$ be an algebraic closure of the residue field of $F$. Let $\Perf$ be the $v$-site of perfectoid spaces over $k$ and let $\Perfs \subset \Perf$ be the sub-site consisting of strictly totally disconnected spaces. For each $S \in \Perfs$, the algebraic Fargues-Fontaine curve $X_S^\alg$ is defined in \cite[p.67]{FS24}. Then, we have a description of $\Bun_G$ as a mapping stack
\[
    \Bun_G(S) = \Map(X_S^\alg, [\ast/G]), \quad S \in \Perfs. 
\]
Now, let $X$ be a normal quasi-projective $G$-variety. % (see \Cref{conv:Gvar} for our terminology on $G$-varieties). 
Following the original definition \cite[(10.6)]{BZSV}, we define the prestack $\Bun_G^X$ of $G$-bundles with an $X$-section as a mapping stack
\[
    \Bun_G^X(S) = \Map(X_S^\alg, [X/G]), \quad S \in \Perf^\std. 
\]
This coincides with what is denoted by $\cl{M}_{[X/G]}$ in \cite{Ham22}, but we change the notation to match with \cite{BZSV}. Let $\Bun_G^1 \subset \Bun_G$ be the trivial locus, which is an open substack isomorphic to $[\ast/\und{G(F)}]$. Then, we have the following basic properties of $\Bun_G^X$. 

\begin{prop2}\textup{(\Cref{prop:BunGXArtin} and \Cref{prop:trivfib})} \label{prop2:BunGX}
    The prestack $\Bun_G^X$ is a $v$-stack on $\Perfs$ and can be extended to an Artin $v$-stack on $\Perf$. There is an open and closed immersion
    \[
        i_X^1\colon [\underline{X(F)}/\underline{G(F)}] \hookrightarrow \Bun^X_G\times_{\Bun_G} \Bun_G^1
    \]
    over $\Bun_G^1 \cong [\ast/\underline{G(F)}]$ and it is an isomorphism if $X$ is quasi-affine. 
\end{prop2}

The first part was proved in \cite{Ham22} when $X$ is smooth, but we present a different proof based on Sumihiro's theorem \cite{Sum75}. %In the original definition \cite[(10.6)]{BZSV}, $\Bun_G^X$ is twisted depending on a character $\eta_X \colon G \to \bb{G}_m$ coming from a nowhere vanishing eigen-volume form on $X$. We adopt a non-twisted version here since the original twisting does not make sense over the Fargues-Fontaine curve. 

Let $\pi_X \colon \Bun_G^X \to \Bun_G$ be the natural map. To introduce the period sheaf, we need a functor $\pi_{X!}$. For this, we need to replace the sheaf theory because lower shriek functors are not available for lisse-\'{e}tale sheaves. This defect was noticed by experts for long. It was expected to be resolved by the motivic six functor formalism \cite{Sch24}, and a construction of an $\ell$-adic six functor formalism based on \cite{Sch24} was given in \cite[Notation]{HHS24}. In \Cref{sec:sixfunc}, we provide a detailed account of this $\ell$-adic six functor formalism. Following \cite{Sch25}, the sheaf theory is denoted by $\cl{D}^\oc(S, \Qla)$ for a small $v$-stack $S$. 

In \Cref{prop:rel!able}, we show that $\pi_X$ is $!$-able in the $\cl{D}^\oc$-formalism and define the unnormalized period sheaf by $\cl{P}_X = \pi_{X!} \Qla$. By \Cref{prop2:BunGX}, $\cl{P}_X\vert_{\Bun_G^1}$ contains $C_c^\infty(X(F), \Qla)$ as a direct summand, so it can be regarded as a geometrization of the Schwartz space of $X$. 

To introduce the normalized period sheaf, we first introduce the degree sheaf 
\[
    \underline{\deg} \in \cl{D}^\oc(\Bun_{\bb{G}_m}, \Qla)
\]
as in the geometric setting \cite[Remark 10.4.1]{BZSV} (see \Cref{defi:degshf}). Here, we need the choice of $\sqrt{q} \in \Qla$, which is an analogue of a square root of the canonical sheaf in the geometric setting. Now, suppose that $X$ is smooth. The canonical bundle of $X$ provides
\[
    \Omega_X^\tp \colon [X / G] \to [\ast / \bb{G}_m]. 
\]

\begin{defi2}\textup{(\Cref{defi:normP})}
    Let $\omega_X \colon \Bun_G^X \to \Bun_{\bb{G}_m}$ be the composition map with $\Omega_X^\tp$. We define the normalized period sheaf associated to $X$ by
    \[
        \cl{P}_X^\norm = \pi_{X!} \omega_X^* \und{\deg}. 
    \]
\end{defi2}

When $X$ admits a nowhere vanishing eigen-volume form with an eigencharacter $\eta_X \colon G \to \bb{G}_m$, we have $\cl{P}_X^\norm = \cl{P}_X \otimes \eta_X^* \und{\deg}$ (see \Cref{prop:normPetaX}). 
%In particular, $\cl{P}_X^\norm = \cl{P}_X$ if $X$ is unimodular. 
Our normalization works even if $X$ is proper. Let $P \subset G$ be a parabolic subgroup with a Levi subgroup $M$. In \Cref{prop:EisPnorm}, we show the compatibility with geometric Eisenstein series 
\[
    \cl{P}^\norm_{G\times^P X} = \Eis_{P !}(\cl{P}_X^\norm)
\]
for every smooth quasi-projective $M$-variety $X$. We may regard $\cl{P}_X^\norm$ as the geometrization of the space of half-densities. 

%In contrast to \cite[Remark 10.4.2]{BZSV}, our normalization of $\cl{P}_X$ concerns only with the degree sheaf. We are not sure if our normalization is enough, but we expect that simpler normalization would do due to the simplicity of our situation such as $\dim \Bun_G = 0$.  

\subsubsection{The $\cl{B}$-side}

Next, we explain the story on the $\cl{B}$-side. Since we work in characteristic $0$, we use the stack of Weil-Deligne $L$-parameters as in \cite[Section 3.1]{Zhu21}. Let $Z^1(\WD_F, \LG)$ be the moduli space of Weil-Deligne $L$-parameters with values in $\LG$. The quotient stack $[Z^1(\WD_F, \LG) / \widehat{G}]$ is called the stack of $L$-parameters. 

Though this moduli is usually considered only over classical $\Qla$-algebras, it was observed in \cite[Remark 3.1.4]{Zhu21} and \cite[Proposition VIII.2.1]{FS24} that the definition can be extended to animated $\Qla$-algebras. This extension is important for our purposes since relative spectral stacks have nontrivial derived structure in general. 

Let $\WD_{F} = \bb{G}_a \rtimes \und{W_F}$ be the Weil-Deligne group scheme over $\Qla$. Let $Q$ be a finite quotient of $W_F$ such that the action of $W_F$ on $\widehat{G}$ factors through $Q$. Let $\LG = \widehat{G} \rtimes Q$ be the $L$-group of $G$. In \Cref{sssec:ParLGmap}, we verify that the derived mapping stack
\[
    \Par_{\LG} = \und{\Map}_{[\ast/Q]}([\ast/\WD_{F}], [\ast/\LG])
\]
defined over derived $\Qla$-schemes recovers the classical stack $[Z^1(\WD_F, \LG) / \widehat{G}]$. This is a derived enhancement of \cite[Definition 4.3]{Sch25}. Then, it is natural to define the relative spectral stack as the mapping stack
\[
    \Par_{\LG}^{\widehat{X}} = \und{\Map}_{[\ast/Q]}([\ast/\WD_{F}], [\widehat{X}/\LG])
\]
for a derived $\LG$-scheme $\widehat{X}$. We can show that it behaves well under suitable conditions on $\widehat{X}$. % (see \Cref{defi:relspec}), which are satisfied when $\widehat{X}$ is a quasi-projective smooth $\LG$-variety. 
Let $\clas \widehat{X}$ denote the classical truncation of $\widehat{X}$. 

\begin{prop2}\textup{(\Cref{prop:ParLGXschematic})}
    Let $\widehat{X}$ be a derived $\LG$-scheme locally of finite presentation over $\Qla$ such that $\clas\widehat{X}$ is $\widehat{G}$-quasi-projective. Then, $\Par_{\LG}^{\widehat{X}}$ is a $1$-Artin stack locally of finite presentation over $\Qla$ and $\pi_{\widehat{X}} \colon \Par_{\LG}^{\widehat{X}} \to \Par_{\LG}$ is schematic. %We define the unnormalized $L$-sheaf associated to $\widehat{X}$ by 
    %\[
    %    \cl{L}_{\widehat{X}} = \pi_{\widehat{X}*} \omega_{\Par_{\LG}^{\widehat{X}}} \in \IndCoh(\Par_{\LG}).  
    %\]
\end{prop2}

By Sumihiro's theorem, the condition on $\widehat{X}$ is satisfied when $\widehat{X}$ is a smooth quasi-projective $\LG$-variety. The key input of the proof is the Artin-Lurie representability theorem. 

Let $\omega_{\widehat{X}}$ be the dualizing complex of $\Par_{\LG}^{\widehat{X}}$. We employ the theory of ind-coherent sheaves developed in \cite{GR17I} and define the unnormalized $L$-sheaf by $\cl{L}_{\widehat{X}} = \pi_{\widehat{X}*} \omega_{\widehat{X}}$. 

%define the unnormlized $L$-sheaf by $\cl{L}_{\widehat{X}} = \pi_{\widehat{X}*} \omega_{\widehat{X}}$. 
%Here, $\omega_{\Par_{\LG}^{\widehat{X}}}$ is the dualizing complex of $\Par_{\LG}^{\widehat{X}}$ and we employ the theory of ind-coherent sheaves developed in \cite{GR17I}. 

Now, we will introduce the normalized $L$-sheaf. We first need to normalize $\Par_{\LG}^{\widehat{X}}$ itself, similarly to the normalization on the $\cl{A}$-side in \cite[Section 10.2]{BZSV}. Let $\bb{G}_{gr}$ denote a copy of $\bb{G}_m$ and endow $\widehat{X}$ with a grading, that is, a left $\bb{G}_{gr}$-action commuting with the $\LG$-action. From the choice of $\sqrt{q}$, we have an unramified $L$-parameter $\sqrt{\cyc} \colon W_F \to \bb{G}_{gr}$ sending $\Fr$ to $q^{-1/2}$. Then, we define the normalization $\Par_{\LG}^{\widehat{X}, \norm}$ by the following Cartesian diagram: 
\begin{center}
    \begin{tikzcd}[column sep = large]
        \Par_{\LG}^{\widehat{X}, \norm} \ar[r] \ar[d, "\pi_{\widehat{X}}^{\norm}"] & \Par_{\LG \times \bb{G}_{gr}}^{\widehat{X}} \ar[d, "\pi_{\widehat{X}}"] \\
        \Par_{\LG} \ar[r, "\id \times \sqrt{\cyc}"] & \Par_{\LG\times \bb{G}_{gr}}. 
    \end{tikzcd}
\end{center}
Let $\omega^\norm_{\widehat{X}}$ be the dualizing complex of $\Par_{\LG}^{\widehat{X}, \norm}$. We can see that $\pi_{\widehat{X}*}^\norm \omega_{\widehat{X}}^\norm$ is naturally equipped with a grading, so we have the weight decomposition and the shearing
\[
    \pi_{\widehat{X}*}^\norm \omega_{\widehat{X}}^\norm = \bigoplus_{n \in \bb{Z}} (\pi_{\widehat{X}*}^\norm \omega_{\widehat{X}}^\norm )_{\std^n}, \quad
    (\pi_{\widehat{X}*}^\norm \omega_{\widehat{X}}^\norm)^\shear = \bigoplus_{n \in \bb{Z}} (\pi_{\widehat{X}*}^\norm \omega_{\widehat{X}}^\norm )_{\std^n}[n]. 
\]
\begin{defi2}\textup{(\Cref{defi:normalized_Lsheaf})}
    We define the normalized $L$-sheaf associated to $\widehat{X}$ by 
    \[
        \cl{L}_{\widehat{X}}^\norm = (\pi_{\widehat{X}*}^\norm \omega_{\widehat{X}}^\norm)^\shear. 
    \]
\end{defi2}

Let $Z = Z(\widehat{G})^Q$. Then, $\IndCoh(\Par_{\LG})$ admits a weight decomposition with respect to $Z$. When the grading on $\widehat{X}$ is given by a central cocharacter $z_{\widehat{X}} \colon \bb{G}_{gr} \to \LG$, $\cl{L}_{\widehat{X}}^\norm$ can be computed from $\cl{L}_{\widehat{X}}$ in terms of this $Z$-weights (see \Cref{cor:normLzX}). 
%so that
%\[
%    \cl{L}_{\widehat{X}}^\norm = \tau_{\sqrt{\cyc}}^* \cl{L}_{\widehat{X}}^\shear, \quad
%    \cl{L}_{\widehat{X}}^\shear = \bigoplus_{\chi \in X^*(Z)} \cl{L}_{\widehat{X}, \chi}[\langle \chi, z_{\widehat{X}} \rangle]. 
%\]
%(see \Cref{cor:normLzX}). Here, 
%\[
%    \tau_{\sqrt{\cyc}} \colon \Par_{\LG} \to \Par_{\LG}, \quad \varphi \mapsto  (z_{\widehat{X}}\circ \sqrt{\cyc}) \cdot \varphi. 
%\]
%is a translation map. 

The normalization by the half-cyclotomic character $\sqrt{\cyc}$ is reminiscent of local Tate duality. In \Cref{cor:FEstd}, we verify that functional equations holds in the vectorial case 
\[
    \cl{L}_V^\norm \cong \cl{L}_{V^\vee}^\norm, \quad V \in \Rep(\LG). 
\]
This essentially follows from the functional equation developed in \cite[Section 11.10.2]{BZSV}. In \Cref{prop:EisLnorm}, we also prove the compatibility with spectral Eisenstein series
\[
    \cl{L}^\norm_{\LG \times^{\LP} \widehat{X}} = \Eis_P^\spec(\cl{L}_{\widehat{X}}^\norm)
\]
for every pair $(\LM, \widehat{X})$. 

%When the grading on $\widehat{X}$ is given by a central cocharacter $z_{\widehat{X}} \colon \bb{G}_{gr} \to \LG$, we have $\cl{L}_{\widehat{X}}^\norm = \tau_{\sqrt{\cyc}}^*\cl{L}_{\widehat{X}}^\shear$ (see \Cref{cor:normLzX}). Here, 
%\[
%    \tau_{\sqrt{\cyc}} \colon \Par_{\LG} \to \Par_{\LG}, \quad \varphi \mapsto  (z_{\widehat{X}}\circ \sqrt{\cyc}) \cdot \varphi. 
%\]
%is a translation map. 

%This description is compatible with the $\cl{A}$-side via the local class field theory: twisting by a character corresponds to translating $L$-parameters. It is a bit hard to see a similarity with the geometric setting \cite[Remark 11.5.2]{BZSV}: % aside the shearing operation: 
%the half-cyclotomic translation $\tau_{\sqrt{\cyc}}^*$ seems to be an analogue of the half-epsilon factor $\varepsilon^{1/2}$, but the former is modelled on $z_{\widehat{X}}$, while the latter is modelled on an eigencharacter $\eta_{\widehat{X}}$ of $\widehat{X}$. 

%Inspite of this incompatibility, our normalization works for the Iwasawa-Tate and Hecke periods as we will see below. We can also prove functional equations of $L$-sheaves in the vectorical case with this normalization, by adapting \cite[Section 11.10.2]{BZSV} to our setting (see \Cref{sssec:FEL}). 
%We are not sure if our normalization works in general due to this incompatibility, but it works at least for the Iwasawa-Tate and Hecke periods as we will see below. Another justification is that the functional equation in \cite[Section 11.10.2]{BZSV} can be adapted to our situation with this normalization (see \Cref{sssec:FEL}). 

\subsection{Main computations}

Now, we can translate the \textit{normalized period conjecture} \cite[Conjecture 12.1.1]{BZSV} into the setting of the categorical local Langlands correspondence. 

\begin{conj2}\textup{(\Cref{conj:normP=L})} \label{conj2:NPC}
    Let $G$ be a connected quasi-split reductive group over $F$ and fix a Whittaker datum of $G$. Suppose that there exists a categorical equivalence 
    \[
        \bb{L}_G \colon \cl{D}^\oc(\Bun_G, \Qla) \cong \IndCoh(\Par_{\LG})
    \]
    as formulated in \textup{\cite[Conjecture X.1.4]{FS24}}. For every (conjecturally defined) dual pair $(G, X) \leftrightarrow (\LG, \widehat{X})$, we have 
    \[
        \bb{L}_G(\cl{P}_X^\norm) \cong \cl{L}_{\widehat{X}}^\norm. 
    \]
\end{conj2}

Although a satisfactory general definition of dual pairs $(G, X) \leftrightarrow (\widehat{G}, \widehat{X})$ is not available yet, a partial list of such pairs can be found in \cite[Table 1.5.1]{BZSV}.

Inspired by the computation \cite{FW25} in the geometric setting, we verify this conjecture for the Iwasawa-Tate and Hecke periods under certain conjectural properties on $\bb{L}_G$. The method in loc. cit. relies heavily on the Ran version of Hecke operators and the theory of chiral algebras, which are not currently available in our setting. Instead, we carry out a more direct and detailed computation based on weight decompositions. 

\subsubsection{The Iwasawa-Tate period}
The dual pair and the normalization factors are as follows. 
\begin{center}
    \begin{tabular}{cccc|cc}
        $G$ & $X$ & $\LG$ & $\widehat{X}$ & $\eta_X$ & $z_{\widehat{X}}$ \\ \hline
        $\bb{G}_m$ & $\std$ & $\bb{G}_m$ & $\std$ & $\std$ & $\std$ 
    \end{tabular}
\end{center} 
Here, $\std$ is the standard character of $\bb{G}_m$, $\eta_X$ is an eigencharacter of $X$ and the grading on $\widehat{X}$ is given by a central cocharacter $z_{\widehat{X}}$. Unwinding the definitions, we can see that
\[
    \Bun_{\bb{G}_m}^{\bb{A}^1} \to \Bun_{\bb{G}_m}
\]
is the Banach-Colmez space of the universal line bundle parametrized by $\Bun_{\bb{G}_m}$, and 
\[
    \Par_{\bb{G}_m}^{\bb{A}^1} \to \Par_{\bb{G}_m}
\]
is the derived vector bundle associated to the Weil-Deligne cohomology of the universal Weil-Deligne representation over $\Par_{\bb{G}_m}$. On the $\cl{A}$-side, $\cl{P}_X$ is a compactly supported cohomology of Banach-Colmez spaces, whose computation is well-known (e.g. \cite[Lemma 2.6]{HI24}). On the other hand, our computation of $\cl{L}_{\widehat{X}}$ is based on a detailed study of derived vector bundles. As a result, we obtain the following compatibility. 

\begin{thm2} \textup{(\Cref{thm:IwTatecomp} and \Cref{prop:FEIwTate})}
    For the Iwasawa-Tate period, we have
    \[
        \bb{L}_{G}(\cl{P}_X^\norm) \cong \cl{L}_{\widehat{X}}^\norm. 
    \]
    Moreover, we have a functional equation $\cl{P}_{\std}^\norm \cong \cl{P}_{\std^\vee}^\norm$. 
\end{thm2}

Note that the existence of $\bb{L}_G$ is proved by \cite{Zou24} in this case. The functional equation $\cl{P}_{X}^\norm \cong \cl{P}_{X'}^{\norm}$ is expected to hold when the associated hyperspherical $G$-varieties are isomorphic: $T^*X \cong T^*X'$. For the Iwasawa-Tate period, we can verify this from the explicit description and the Fourier transform on the local field $F$. 

\subsubsection{The Hecke period}
The dual pair and the normalization factors are as follows. 
\begin{center}
    \begin{tabular}{cccc|cc}
        $G$ & $X$ & $\LG$ & $\widehat{X}$ & $\eta_X$ & $z_{\widehat{X}}$ \\ \hline
       $\GL_2$ & $\GL_2/A$ & $\GL_2$ & $\std$ & $\triv$ & $\textrm{scaling}$
    \end{tabular}
\end{center}
Here, $\std$ is the standard two-dimensional representation of $\GL_2$, $A = {\scriptsize \left\{ \begin{pmatrix} 1 & 0 \\ 0 & \ast \end{pmatrix} \right\} } \cong \bb{G}_m$, $X$ is unimodular and the grading on $\widehat{X}$ is given by a central cocharacter $z_{\widehat{X}}(t) = {\scriptsize \left\{ \begin{pmatrix} t & 0 \\ 0 & t \end{pmatrix} \right\} }$. %, and $\std$ denotes the standard (two-dimensional) representation of $\GL_2$. 
Unwinding the definitions, we can see that
\[
    \Bun_{\GL_2}^{X} \cong \Bun_A \to \Bun_{\GL_2}, 
\]
and 
\[
    \Par_{\GL_2}^{\std} \to \Par_{\GL_2}
\]
is the derived vector bundle associated to the Weil-Deligne cohomology of the universal Weil-Deligne representation over $\Par_{\GL_2}$. Let $\Bun_{A} = \coprod_{n \in \bb{Z}} \Bun_{A, n}$ be the decomposition in terms of the degree of line bundles. Then, we have a decomposition
\[
    \cl{P}_X = \bigoplus_{n \in \bb{Z}} \cl{P}_{X, n}, \quad \cl{P}_{X, n} = \pi_{X!} \und{\Lambda}_{\Bun_{A, n}}. 
\]
On the $\cl{B}$-side, we have a weight decomposition
\[
    \cl{L}_{\widehat{X}} = \bigoplus_{n \in \bb{Z}} \cl{L}_{\widehat{X}, n}
\]
with respect to the center of $\GL_2$. Due to the Hecke compatibility, these decompositions should match under $\bb{L}_G$. We verify this compatibility in the following form by providing a description of $\cl{P}_{X, n}$ and $\cl{L}_{\widehat{X}, n}$ in terms of geometric and spectral Eisenstein series (see \Cref{prop:HeckepEis} and \Cref{prop:HeckeL}). 

\begin{thm2}\textup{(\Cref{thm:Heckecomp})} 
    Suppose that $\bb{L}_{\GL_2}$ exists and satisfies the Eisenstein compatibility. For $n \neq 0$, we have
    \[
        \bb{L}_{\GL_2}(\cl{P}_{X,n}^\norm) \cong \cl{L}_{\widehat{X}, n}^\norm. 
    \]
    For $n = 0$, the first (resp.\ last) term of the fiber sequence
    \[
        \cl{W}_\psi \hookrightarrow \cl{P}^\norm_{X, 0} \to \Eis_{B!}(i^1_! \cInd_{A(F)}^{T(F)} \Qla_\norm^{-1/2})
    \]
    maps under $\bb{L}_{\GL_2}$ to the first (resp.\ last) term of the fiber sequence
    \[
        \cl{O}_{\Par_{\GL_2}} \to \cl{L}^\norm_{\widehat{X}, 0} \to \Eis_{\ov{B}}^\spec(\cl{O}_{\Par_{\bb{G}_m}} \boxtimes  (i_{\cyc^{-1/2}*} \Qla)_{\triv}). 
    \]
\end{thm2}

Here, $\Eis_{B!}$ (resp.\ $\Eis_{\ov{B}}^\spec$) denotes the geometric (resp.\ spectral) Eisenstein series with respect to $B$ (resp.\ $\ov{B}$). On the $\cl{A}$-side, $\cl{P}_{X, n}$ is computed directly for $n \neq 0$. For $n = 0$, a fiber sequence representing $\cl{P}_{X, 0} = \cInd_{A(F)}^{G(F)} \Qla$ is constructed via the unfolding technique.

On the $\cl{B}$-side, the computation of $\cl{L}_{\widehat{X}, n}$ is somewhat indirect. As in the geometric setting, we have a closed-open decomposition
\begin{center}
    \begin{tikzcd}
        \Par_{\GL_2} \ar[r, hook] & \Par_{\GL_2} ^\std & \Par_{\ov{\Mir}_2}. \ar[l, hook']
    \end{tikzcd}
\end{center}
Here, $\ov{\Mir}_2 = \left\{ \scriptsize \begin{pmatrix} \ast & 0 \\ \ast & 1 \end{pmatrix} \right\} \subset \GL_2$. Following the strategy of \cite{FW25}, we get a fiber sequence representing $\cl{L}_{\widehat{X}, n}$ from this decomposition. It is enough for the claim when $n \leq 0$. For $n > 0$, the fiber sequence is rather complicated, so we first use the functional equation
%However, the fiber sequence is too complicated for $n > 0$. In this case, we first use the functional equation %developed in \cite[Section 11.10.2]{BZSV} to deduce that 
\[
    \cl{L}_{\std, n}^{\norm} \cong \cl{L}_{\std^\vee, -n}^\norm, 
\]
and then apply the strategy for $n < 0$. % to $\cl{L}_{\std^\vee, -n}^\norm$. 
%By applying the strategy for $n < 0$ to $\cl{L}_{\std^\vee, -n}^\norm$, we obtain a description of $\cl{L}_{\std, n}^{\norm}$ matching with $\cl{P}_{X, n}^\norm$ under the Eisenstein compatibility. 

\subsubsection{Relation to distinction problems}

In \Cref{ssec:NPC}, we derive the following consequence of the normalized period conjecture from a mild expectation on the form of Hecke eigensheaves (see \Cref{conj:Heckeeigen}). 
%For example, the latter has a representation-theoretic consequence \Cref{conj:Xdistrep} at semisimple $L$-parameters of Langlands-Shahidi type, under some expectations on categorical local Langlands correspondence.  

\begin{prop2} \textup{(\Cref{conj:Xdistrep})}
    Suppose that \Cref{conj2:NPC} holds for a dual pair $(G, X) \leftrightarrow (\LG, \widehat{X})$ such that $X$ is unimodular. If a smooth irreducible $G(F)$-representation $\pi$ is $X$-distinguished and its Fargues-Scholze parameter $\varphi_\pi^{\FS}$ is of Langlands-Shahidi type and satisfies \Cref{conj:Heckeeigen}, $\varphi_\pi^{\FS}$ lies in the image of $\pi_{\widehat{X}}^\norm$. 
\end{prop2}

Under suitable conditions, we expect that $X$ admits an $L$-group $\LGX \subset \LG$, and as in \cite[Section 4.1.1]{BZSV}, $\widehat{X}$ takes the form 
\[
    \widehat{X} = \LG \times^{\LGX} V
\]
for some $V \in \Rep(\LGX)$, with a grading induced from $V$. Then, $\pi_{\widehat{X}}^\norm$ factors through $\Par_{\LGX}$, so $\varphi_\pi^\FS$ should factor through a conjugate of $\LGX$ in the above setting.

In \cite[Conjecture 2]{Pra15}, a particular case of this expectation was proposed when $X$ is a Galois symmetric variety, where $G = \Res_{E/F}(H_E)$ and $X = G / H$ for a quasi-split connected reductive group $H$ over $F$ and a separable quadratic extension $E / F$. In this situation, we have an $L$-group $\LGX \subset \LG$ and a dual pair
\[
    (G, X = G / H) \leftrightarrow (\LG, \widehat{X} = \LG / \LGX). 
\]
It turns out that the normalized period conjecture can also recover the multiplicity formula of Prasad's conjecture (see \cite[Conjecture 3.16]{BP25}) under some conditions. 

\begin{prop2}\textup{(\Cref{prop:multpi})}
    Let $\varphi \colon W_E \to \LH(\Qla)$ be a supercuspidal $L$-parameter with a centralizer $S_\varphi \subset \widehat{H}$ and let $z_\varphi \colon [\ast / S_\varphi] \hookrightarrow \Par_{\LG}$. Suppose that \Cref{conj2:NPC} holds for 
    \[
        (G, X = G / H) \leftrightarrow (\LG, \widehat{X} = \LG / \LGX)
    \]
    and $\rho_\pi \in \Irr(\pi_0(S_\varphi) / \pi_0(Z))$ satisfies $\bb{L}_G(i^1_! \pi) \cong z_{\varphi*} \cl{O}_{\rho_\pi}$. When the center of $H$ is anisotropic, 
    \[
        \dim \Hom(C_c^\infty(X(F), \Qla), \pi) = \sum_{\psi / \varphi} \dim \rho_\pi^{\pi_0(S_\psi) = \id}. 
    \]
    Here, $\psi\colon W_F \to \LGX(\Qla)$ runs through all lifts of $\varphi$ (up to conjugate) and $S_\psi \subset \widehat{H}$ denotes the centralizer of $\psi$. Moreover, for every $i > 0$, we have
    \[
        \Ext^i(C_c^\infty(X(F), \Qla), \pi) = 0. 
    \]
\end{prop2}

We expect that a similar analysis would yield to multiplicity formulas in several other settings. For general $X$, it would be interesting to seek for the relation to the local $L^2$-conjecture of Sakellaridis-Venkatesh \cite[Conjecture 16.2.2]{SV17}. 

%Let $H$ be a connected quasi-split reductive group over $F$ and let $E$ be a separable quadratic extension of $F$. When $G = \Res_{E/F}(H_E)$ and $X = G / H$, this recovers part of Prasad's conjecture; \cite[Conjecture 1]{BP18} for the trivial character. 
%Suppose that $X$ is a spherical $G$-variety and let $\LGX \subset \LG$ be the $L$-group of $X$ (see \cite[Section 2.2.10]{BP25}). Under suitable conditions, we expect $\widehat{X}$ to take the form 
%\[
%    \widehat{X} = \LG \times^{\LGX} V
%\]
%for some $V \in \Rep(\LGX)$ as in \cite[Section 4.1.1]{BZSV}. Then, $\pi_{\widehat{X}}^\norm$ factors through $\Par_{\LGX}$, so $\varphi_\pi^\FS$ should factor through a conjugate of $\LGX$ in the above setting. 

%While $\cl{P}_X$ geometrizes the Schwartz space $C_c^\infty(X(F), \Qla)$, the local conjecture of Sakellaridis-Venkatesh \cite[Conjecture 16.2.2]{SV17} concerns the Plancherel decomposition of the $L^2$-spectrum $L^2(X(F))$. As it intrinsically depends on the norm on $\bb{C}$, it seems hard to geometrize it with $\ell$-adic coefficients. It would be interesting to seek for the relation between Sakellaridis-Venkatesh's local conjecture and the normalized period conjecture. 

\addtocontents{toc}{\protect\setcounter{tocdepth}{1}}

\subsection*{Authorship}

The project originated during the Summer School on Relative Langlands Duality at the proposal of the second author (M.L.). The main proofs and technical developments were carried out by the first author (Y.T.). M.L. contributed through essential discussions and revisions of the manuscript; part of this work was completed during M.L.'s military service in Taiwan. In view of Y.T.'s primary role in proving the results and drafting the manuscript, the authors are listed in that order. We recognize that alphabetical ordering is common in parts of mathematics and provide this transparent contribution statement for clarity. Both authors approved the final version and share responsibility for the results.

\subsection*{The structure of the paper}

In \Cref{sec:sixfunc}, we provide a detailed account of an $\ell$-adic six functor formalism on small $v$-stacks. In \Cref{sec:Aside}, we study the basic properties of the stack $\Bun_G^X$ and the unnormalized period sheaf $\cl{P}_X$. In \Cref{sec:Acomputation}, we compute $\cl{P}_X$ for the Iwasawa-Tate and Hecke periods. In \Cref{sec:Bside}, we study the basic properties of the stack $\Par_{\LG}^{\widehat{X}}$ and the unnormalized $L$-sheaf $\cl{L}_{\widehat{X}}$. In \Cref{sec:Bcomputation}, we compute $\cl{L}_{\widehat{X}}$ for the Iwasawa-Tate and Hecke periods. In \Cref{sec:normalization}, we state the normalized period conjecture and study its relation to distinction problems. In \Cref{sec:comparison}, we verify the normalized period conjecture for the Iwasawa-Tate and Hecke periods.

\subsection*{Acknowledgements}

The authors are grateful to Rok Gregoric, Linus Hamann, David Hansen, Naoki Imai and Yiannis Sakellaridis for helpful discussions and valuable comments on an earlier version of this paper. The first author would like to thank his advisor Yoichi Mieda for his constant support and encouragement. He would also like to thank Peter Scholze for kindly answering questions on the motivic six functor formalism \cite{Sch24}. 

The first author was supported during this work by the WINGS-FMSP
program at the Graduate School of Mathematical Sciences, the University of Tokyo and JSPS KAKENHI Grant number JP24KJ0865.

\subsection*{Notation} \label{ssec:Notation}

All rings are assumed to be commutative. We fix a prime $p$ and another prime $\ell \neq p$. Let $F$ be a non-archimedean local field with residual characteristic $p$ and let $W_F$ be the Weil group of $F$. Let $O_F$ be the ring of integers of $F$ with a uniformizer $\pi$ and let $q$ be the number of elements of the residue field of $F$. %For a perfect $\bb{F}_q$-algebra $A$, let $W_{O_F}(A)=W(A)\otimes_{W(\bb{F}_q)} O_F$. 
%Let $k$ be an algebraic closure of $\bb{F}_q$. Let $\Perf$ be the category of perfectoid spaces over $k$. 

Let $S$ be a perfectoid space over $\bb{F}_q$. The Fargues-Fontaine curve over $F$ associated to $S$ is denoted by $\cl{X}_S$. It is defined as a Frobenius quotient
$\cl{X}_S = \cl{Y}_S / \varphi^\bb{Z}$
of an analytic adic space $\cl{Y}_S$. When $S = \Spa(R,R^+)$, $\cl{X}_S$ is also denoted by $\cl{X}_{(R, R^+)}$. 

The (large) $\infty$-category of $\infty$-categories is denoted by $\Cat_\infty$. The (large) $\infty$-category of presentable (resp.\ stable) $\infty$-categories with colimit preserving functors as morphisms is denoted by $\Pr^L$ (resp.\ $\Pr^L_{\st}$). The (large) $\infty$-category of $\infty$-groupoids is denoted by $\Ani$. The (enriched) mapping space in an (enriched) $\infty$-category $\cl{C}$ is denoted by $\Map_{\cl{C}}(-, -)$ and the (enriched) set of homomorphisms in the homotopy category $h\cl{C}$ is denoted by $\Hom_{h\cl{C}}(-,-)$. In particular, $\pi_0 \Map_{\cl{C}}(-,-)\cong \Hom_{h\cl{C}}(-,-)$. When $\cl{C}$ is clear from the context, they are just denoted by $\Map(-, -)$ and $\Hom(-, -)$. In some contexts, $\cl{C}$ is enriched in the derived category $D(\Qla)$.  

The abelian category of finite dimensional algebraic representations of an algebraic group $G$ is denoted by $\Rep(G)$. 
We refer to \Cref{app:DAG} for our convention on derived algebraic geometry. For an animated ring $R$, the derived category of $R$-modules is denoted by $D(R)$. 

\addtocontents{toc}{\protect\setcounter{tocdepth}{2}}

\section{An $\ell$-adic six functor formalism on small $v$-stacks} \label{sec:sixfunc}

%Fix a prime $\ell \neq p$. 
In this section, we provide a detailed account of an $\ell$-adic six functor formalism on small $v$-stacks, following the construction in \cite[Notation]{HHS24}. It is essentially obtained as the base change of the motivic six functor formalism \cite{Sch24}. In contract to the \'{e}tale six functor formalism \cite{Sch17}, a coefficient ring can be taken to be an arbitrary $\bb{Z}_\ell$-algebra. 
Also, in contrast to the lisse-\'{e}tale five functors \cite{FS24}, excision holds in many situations and the class of $!$-able morphisms is large enough for application. 

%In \cite{Sch17}, small $v$-stacks on $\Perf$ are proved to admit an \'etale $\ell$-adic 6 functor formalism in torsion coefficients. In contrast, Scholze's arc-theoretic motivic framework \cite{Sch24} provides a six functor formalism in $\bb{Z}[\tfrac{1}{p}]$-coefficients, which induces one on $v$-stacks, referred to as the \textit{overconvergent motivic sheaf theory} $\cl{D}^\oc_\mot(-)$. Its base change along the $\ell$-adic realization provides an $\ell$-adic six functor formalism, which we will use in application. 

\subsection{Preliminaries on six functor formalisms}

In this section, we recall the extension procedure of a six functor formalism developed in \cite{HM24}. First, let us recall the definition. 

\begin{defi} \label{defi:six_functor_formalism}
    Let $(\cl{C},E)$ be a geometric setup. A $3$-functor formalism is a lax symmetric monoidal functor 
    \[ 
        \cl{D}\colon\Corr(\cl{C},E) \to \Cat_\infty;
    \]
    \[ 
        \smbr{ X \xleftarrow{g} W \xrightarrow{f} Y}  \mapsto  f_!g^*\colon \cl{D}(X) \to \cl{D}(Y). 
    \] 
    This is a six functor formalism if the right adjoints $\underline{\Hom}, f_*, f^!$ exist. We say that $\cl{D}$ is presentable if $\cl{D}$ factors through $\Pres^L$. When $\cl{C}$ is a site, we say that $\cl{D}$ is sheafy if $\cl{D}$ satisfies the descent along every \v{C}ech nerve. 
\end{defi}

A typical construction of a six functor formalism is packaged in the following form. 

\begin{prop}\label{prop:3_functor_formalism_from_IP} \textup{(\cite[Proposition 1.2.5]{HM24}, \cite[Proposition A.5.10]{Man22})} 
    Let $(\cl{C},E)$ be a geometric set up such that $\cl{C}$ admits pullbacks and let $I,P \subset E$ be a suitable decomposition. Let $\cl{D}\colon\cl{C}^\mathrm{op} \to \CAlg(\Cat_\infty)$ be a functor satisfying the following conditions. 
    \begin{enumerate}
        \item For all $j \in I$, $j^*$ admits a left adjoint $j_!$ satisfying the base change and the projection formula. 
        \item For all $g \in P$, $g^*$ admits a right adjoint $g_*$ satisfying the base change and the projection formula. 
        \item For every Cartesian square 
        \begin{center}
            \begin{tikzcd}
                W \rar["{j'}"] \ar[d,"{g'}"] & Y \ar[d,"g"]\\
                U\rar["j"] & X
            \end{tikzcd}
        \end{center}
        in $\cl{C}$ with $j \in I$ and $g\in P$, the natural transformation $j_!g'_* \to g_* j'_!$ is an isomorphism.  
    \end{enumerate}
    Then, $\cl{D}$ extends to a $3$-functor formalism
    \[
        \cl{D}\colon \Corr(\cl{C},E) \to \Cat_\infty
    \]
    such that $f_! \cong g_*j_!$ for all $f=gj$ with $j\in I$ and $g \in P$. Moreover, $\cl{D}$ extends to a six functor formalism if the following additional conditions are satisfied. 
    \begin{enumerate}
        \item For all $X \in \cl{C}$, $\cl{D}(X)$ is a closed symmetric monoidal $\infty$-category. 
        \item For all $f\colon Y \to X$, $f^*$ admits a right adjoint $f_*$. 
        \item For all $f\colon Y\to X$ in $P$, $f_*$ admits a right adjoint $f^!$. 
    \end{enumerate}
\end{prop}

This construction can be extended further in the following way. 

\begin{prop} \label{prop:extension_6ff_frome_site} \textup{(\cite[Theorem 3.4.11]{HM24})}
    Let $\cl{D}$ be a sheafy presentable six functor formalism on a geometric setup $(\cl{C},E)$ where $\cl{C}$ is a subcanonical site. Then, there is a collection of edges $E'$ in $\Shv(\cl{C})$  such that 
    \begin{enumerate}
        \item $\cl{D}$ extends uniquely to a sheafy presentable six functor formalism $\cl{D}'$ on $(\Shv(\cl{C}),E')$. 
        \item $E'$ is $*$-local on target and $!$-local on source and target. 
        \item $E'$ is tame: every map in $E'$ with target in $\cl{C}$ is $!$-locally on the source in $E$. 
    \end{enumerate}
\end{prop} 

\subsection{Motivic six functor formalism on small $v$-stacks} \label{ssec:sixfunc}
%In this section, we provide a detailed account of an $\ell$-adic six functor formalism on small $v$-stacks that is sketched in \cite[Notation]{HHS24}. In \cite{Sch17}, small $v$-stacks on $\Perf$ are proved to admit an \'etale $\ell$-adic 6 functor formalism in torsion coefficients. In contrast, Scholze's arc-theoretic motivic framework \cite{Sch24} provides a six functor formalism in $\bb{Z}[\tfrac{1}{p}]$-coefficients, which induces one on $v$-stacks, referred to as the \textit{overconvergent motivic sheaf theory} $\cl{D}^\oc_\mot(-)$. Its base change along the $\ell$-adic realization provides an $\ell$-adic six functor formalism, which we will use in application. 

In this section, we review the motivic six functor formalism developed in \cite{Sch24} and prove basic properties which will be needed later on. For our use, we take a different set of generating objects from \cite{Sch24}. 

Let $k$ be an algebraic closure of $\bb{F}_p$. Let $\Perf$ be the category of perfectoid spaces over $k$. It is equipped with a subcanonical Grothendieck topology called the \textit{$v$-topology}, which is finer than the pro-\'etale topology. 

\begin{defi}
    A collection $\crbr{f_i: Y_i \to X} _{i \in I}$ of morphisms in $\Perf$ is a $v$-cover if for any quasicompact open subset $U \subset X$, there exist a finite subset $J \subset I$ and a collection of quasicompact open subsets $\crbr{ V_j \subset Y_j}_{j \in J}$ such that $U= \bigcup _{j \in J} f_j(V_j)$. 
\end{defi}

A sheaf (resp.\ stack) over $\Perf$ with the $v$-topology is called a $v$-sheaf (resp.\ $v$-stack). 

\begin{defi} \label{def:small_v_stack} \textup{(\cite[Definition 12.4]{Sch17})}
    \begin{enumerate}
        \item A $v$-sheaf $Y$ is \textit{small} if it has a surjection $X \thra Y$ from a perfectoid space $X$. 
        \item A $v$-stack $Y$ is \textit{small} if it has a surjection $X \thra Y$ from a perfectoid space $X$ such that  $X \times_Y X$ is a small $v$-sheaf.
    \end{enumerate}
   
\end{defi} 
%One feature of small $v$-stacks is that they can be progressively accessed from perfectoid spaces, to small $v$-sheaves, and then to small $v$-stacks via quotients. 
% the following steps: from perfectoid spaces to diamonds via quotients under pro-\'etale equivalence, from diamonds to small v-sheaves by quotients under small diamond equivalence relation; from small v-sheaves to small v-stacks by quotients under small v-sheaf equivalence.\\ 

Let $\vStack$ denote the category of small $v$-stacks on $\Perf$. We have a stronger topology on $\Perf$ called the \textit{arc-topology}. 

\begin{defi}
    \label{defi:arc_topology} A collection $\crbr{f_i: Y_i \to X} _{i \in I}$ of morphisms in $\Perf$ is an arc-cover if for any quasicompact open subset $U \subset X$, there exist a finite subset $J \subset I$ and a collection of quasicompact open subsets $\crbr{ V_j \subset Y_j}_{j \in J}$ such that every rank $1$ point of $U$ has a lift to $V_j$ for some $j \in J$.
\end{defi}

In other words, an arc-cover ensures the existence of lifts only for rank $1$ valuations. As in the $v$-topology, we have a notion of (small) arc-stacks on $\Perf$. For a perfectoid Huber pair $(R,R^+)$ over $k$, $\Spa(R,R^\circ) \to \Spa(R,R^+)$ is an arc-cover. Here, $R^\circ \subset R$ denotes the set of topologically bounded elements. Then, arc-stacks are naturally regarded as stacks on the arc-site $\Perft$ of perfectoid Tate $k$-algebras. As in \cite[Section 12]{Sch24}, let $\arcStack$ denote the category of small arc-stacks on $\Perft$. 

\begin{defi}
    Let $\cl{M}_\arc(R)$ denote the small arc-sheaf represented by a perfectoid Tate $k$-algebra $R$. It is called an affine arc-sheaf.
\end{defi}

There is a natural pullback functor
\[
    a'^* \colon \vStack \to \arcStack
\]
induced from the morphism of sites
\[ 
    \Perf \ra \Perft, \quad (R, R^+) \mapsto R. 
\]
It sends a small $v$-stack $X$ to the arc-sheafification $a'^*X$ of the functor $R\mapsto X(R,R^\circ)$ for a perfectoid Tate $k$-algebra $R$. In particular, $\Spa(R,R^+)$ is sent to $\cl{M}_\arc(R)$.

The \'{e}tale motivic sheaf theory on $\arcStack$ is developed in \cite{Sch24}. Let us denote the six functor formalism as
\[
    \Corr(\arcStack, E') \to \Pres^L, \quad X \mapsto \cl{D}_\mot(X).
\]
It is first constructed from \Cref{prop:3_functor_formalism_from_IP} by taking qcqs maps of finite cohomological dimension between affine arc-sheaves as proper maps, and then extended (and restricted) to $\arcStack$ via \Cref{prop:extension_6ff_frome_site}. 

One way to get a six functor formalism on $\vStack$ is to take the pullback of the above six functor formalism along $a'^*$. For our use, we will adopt a little modification. Let us first begin by reviewing some geometric properties on small arc-stacks.

Any perfectoid Tate $k$-algebra $R$ can be equipped with a seminorm by choosing a pseudo-uniformizer $\varpi \in R$. Then, it can be equipped with the Berkovich spectrum $\cl{M}(R)$, which is independent of the choice of $\varpi \in R$. For the reader's convenience, let us recall the definition. 

\begin{defi}\label{def:seminormed_ring_berkovich_spectrum}
    The \textit{Berkovich spectrum} $\cl{M}(A)$ of a seminormed ring $(A,\lvert \cdot \rvert_A)$ is the closed subspace of $\prod_{a \in A} [0, |a|_A]$ consisting of bounded multiplicative seminorms. 
\end{defi}

\begin{prop} \label{prop:indeptop}
    Let $X$ be a small arc-stack on $\Perft$. Take surjections
    \[
        Y = \bigsqcup_{i\in I} \cl{M}_\arc(R_i) \to X, \quad
        Z = \bigsqcup_{j\in J} \cl{M}_\arc(S_j) \to Y \times_X Y
    \] 
    from disjoint unions of affine arc-sheaves. Then, the topological quotient space \[ \bigsqcup_{i \in I}  \cl{M}(R_i)/\bigsqcup_{j \in J} \cl{M}(S_j) \]  is independent of the choice of surjections and bijective to the equivalence class of points \[ \crbr{ \cl{M}_\arc(C) \to X } /\sim \] from perfectoid fields $C$ over $k$. Here, two maps $\cl{M}_\arc(C_i) \to X$ $(i=1,2)$ are equivalent if there is a point $\cl{M}_\arc(C_3) \to X$ over each point such that we have a commutative diagram
    \[ 
        \begin{tikzcd}
            & \cl{M}_\arc(C_1) \ar[dr] & \\ 
            \cl{M}_\arc(C_3)\ar[ur] \ar[rr] \ar[dr] & & X. \\ 
            & \cl{M}_\arc(C_2) \ar[ur] 
        \end{tikzcd}
    \]
    %Here, the equivalence relation between points of $X$ is taken as in \cite[Proposition 11.13]{Sch17}. 
\end{prop}
\begin{proof}
    The same argument as in \cite[Proposition 11.13]{Sch17} works in our setting. Here, we briefly recall it for the reader's convenience. 
    \begin{enumerate}
        \item First, we need to prove the bijectivity on the set of points. For surjectivity, we can always lift a map $\cl{M}_\arc(C) \to X$ arc-locally to $Y$. The injectivity follows from the fact that any two points $\cl{M}_\arc(C_i) \to \cl{M}_\arc(C)$ $(i=1,2)$ have common refinements $\cl{M}_\arc(C') \to \cl{M}_\arc(C_i)$ $(i=1,2)$ by \cite[Proposition 3.2(i)]{Sch24}. 
        \item Next, we need to prove the independence of the quotient topology. This follows from the fact that an arc-cover $R \to S$ induces a quotient map $\cl{M}(S) \to \cl{M}(R)$ since any surjection between compact Hausdorff spaces is a quotient map. 
    \end{enumerate}
\end{proof}

\begin{defi} \label{defi:topsmarc}
    For a small arc-stack $X$, take surjections as in \Cref{prop:indeptop}. Let $\cl{M}(X)$ denote the quotient space $\bigsqcup_{ i \in I} \cl{M}(R_i)/\bigsqcup_{j \in J} \cl{M}(S_j)$, which is independent of the choice we made. For a small $v$-stack $X$, we abusively denote $\cl{M}(a'^*X)$ by $\cl{M}(X)$. 
\end{defi}

\begin{defi}
    Let $f\colon Y \to X$ be a map of small arc-stacks. We say that $f$ is an open (resp.\ closed) immersion if there is an open (resp.\ closed) subset $W \subset \cl{M}(X)$ such that for every perfectoid Tate $k$-algebra $R$, $f(R)$ identifies $Y(R)$ with the full subgroupoid of $X(R)$ consisting of maps $\cl{M}_\arc(R) \to X$ such that $\cl{M}(R) \to \cl{M}(X)$ factors through $W$. 
\end{defi}

\begin{rmk}
    This definition is compatible with \cite[Definition 4.21]{Sch24}. To recover $W \subset \cl{M}(X)$ from an open immersion in the sense of loc. cit., we take covers as in \Cref{prop:indeptop} and descend an open subset of $\cl{M}(Y)$ to $\cl{M}(X)$. In the same way, we can show that open (resp.\ closed) immersions can be checked arc-locally. 
\end{rmk}

\begin{lem} \label{lem:M(X)}
        For any small $v$-stack $X$, there is a continuous surjection $\lvert X \rvert \to \cl{M}(X)$. If $X$ is qcqs, $\cl{M}(X)$ is identified with the maximal Hausdorff quotient of $\lvert X \rvert$. 
\end{lem}
\begin{proof}
    For the first claim, take $v$-covers $Y\to X$ and $Z \to Y\times_X Y$ from disjoint unions of affinoid perfectoid spaces $Y$ and $Z$. Then, $\lvert X \rvert$ is a quotient space $\lvert Y\rvert/\lvert Z\rvert$. Since the sheafification preserves effective epimorphisms and finite limits, $\cl{M}(X)$ is identified with the quotient space $\cl{M}(Y)/\cl{M}(Z)$. The claim follows since there is a continuous surjection $\lvert Y \rvert \to \cl{M}(Y)$ for every affinoid perfectoid space $Y$ (see \cite[Proposition 13.9]{Sch17}). 
    
    We turn to the case where $X$ is qcqs. Then, $Y$ and $Z$ can be taken as affinoid perfectoid spaces. Then, the equivalence relation of $\cl{M}(Y)$ given by $\cl{M}(Z)$ is closed, so $\cl{M}(Y)/\cl{M}(Z)$ is compact and Hausdorff. Thus, $\cl{M}(X)$ is compact and Hausdorff, so $\lvert X\rvert \to \cl{M}(X)$ factors through the maximal Hausdorff quotient $\lvert X \rvert^B$. By \cite[Proposition 13.9]{Sch17}, $\lvert Y\rvert \to \lvert X \rvert^B$ (resp.\ $\lvert Z\rvert \to \lvert X \rvert^B$) uniquely factors through $\cl{M}(Y)$ (resp.\ $\cl{M}(Z)$). By the uniqueness, the sequence $\cl{M}(Z) \rightrightarrows \cl{M}(Y) \to \lvert X \rvert^B$ induces $\cl{M}(X) \to \lvert X\rvert ^B$, so $\cl{M}(X)$ is identified with $\lvert X \rvert^B$.  \\ 
\end{proof}

\begin{lem} \label{lem:qcqs_preserved} \noindent
    \begin{enumerate}
        \item If a small $v$-stack $X$ is quasicompact (resp.\ quasiseparated), then $a'^*X$ is quasicompact (resp.\ quasiseparated). 
        \item If a map $f\colon Y \to X$ of small $v$-stacks is quasicompact (resp.\ quasiseparated), then $a'^*f$ is quasicompact (resp.\ quasiseparated). 
    \end{enumerate}
\end{lem}
% \begin{proof}
%     If $X$ is quasicompact, there is a $v$-cover $Y \to X$ from an affinoid perfectoid space $Y$ over $k$. Then, $a'^*X$ is quasicompact since the sheafification preserves surjections and $a'^*Y$ is quasicompact. Now, all claims follow using the fact that the sheafification preserves fiber products. 
% \end{proof}

\begin{proof}
     If $X$ is quasicompact, there is a $v$-cover $Y \to X$ from an affinoid perfectoid space $Y$ over $k$. Then, $a'^*X$ is quasicompact since sheafification preserves effective epimorphisms and $a'^*Y = \cl{M}_\arc(R)$ is quasicompact. Now, all other claims follow formally from the fact that $a'^*$ preserves effective epimorphisms and finite limits.  
\end{proof}

We will take qcqs $v$-sheaves as generating objects of our six functor formalism. To introduce $!$-able morphisms between them, let us recall several classes of morphisms. 
% First, recall that a map of small arc-stacks is proper if it is qcqs.
\begin{defi}
    A $0$-truncated map $f:X \ra Y$ of small arc-stacks is \textit{proper} if it is qcqs. 
\end{defi}

\begin{defi} (\cite[Definition 4.10, Proposition 4.12]{Sch24})
    Let
    \[
        \cl{D}_{\fin}(X,\mathbb Z) \subset \cl{D}(X_\arc, \bb{Z})
    \]
    denote the stable $\infty$-category of finitary arc-sheaves on a small arc-stack $X$, which is equivalent to the $\infty$-category of functors from strictly totally disconnected objects in $\Perft$ over $X$ to $\cl{D}(\mathbb Z)$ which commute with finite products and filtered colimits. It is equipped with a natural $t$-structure induced from $\cl{D}(\bb{Z})$.  
\end{defi}

\begin{conv} \label{conv:finsheaf}
    A finitary sheaf on $X$ refers to an object in the heart of $\cl{D}_\fin(X, \bb{Z})$.
\end{conv}
%In our terminology, a finitary sheaf on $X$ will refer to an object in the heart of $\cl{D}_\fin(X, \bb{Z})$.

%Following \cite[Definition 4.10]{Sch24}, let $\cl{D}_{\fin}(X,\mathbb Z)\subset\cl{D}(X_{\arc},\mathbb Z)$ denote the stable $\infty$-category of finitary arc-sheaves on a small arc-stack $X$. It is equivalent to the $\infty$-category of functors from strictly totally disconnected Banach algebras over $X$ to $\cl{D}(\mathbb Z)$ which commute with finite products and filtered colimits \cite[Proposition 4.12, 4.13]{Sch24}. 

%In the following, a finitary sheaf on a small arc-stack $X$ refers to a $0$-truncated object in $\cl{D}_\fin(X, \bb{Z})$. 

\begin{defi}(\cite[Definition 4.17]{Sch24}) \label{def:cohomological_dimension}
    For a small arc-stack $X$, let $d(X)$ denote the cohomological dimension of the global sections functor 
    \[ 
        \cl{D}_\fin(X, \mathbb Z) \ra \cl{D}(\mathbb Z), \quad M \mapsto M(X). 
    \]
\end{defi}
\begin{defi}(\cite[Definition 4.19]{Sch24}) \label{def:cohomological_dimension}
    Let $f\colon Y \to X$ be a $0$-truncated proper map of small arc-stacks. 
    \begin{enumerate}
        \item We say that $f$ has finite cohomological dimension if the cohomological dimension of $Y\times_X \cl{M}_\arc(C)$ is bounded by a constant for all geometric points $\cl{M}_\arc(C) \to X$. Here, $C$ is an algebraically closed perfectoid field over $k$. The maximum of the cohomological dimension as above is denoted by $d(f)$. 
        \item We say that $f$ has locally finite cohomological dimension if the base change of $f$ to every qcqs arc-sheaf $Z$ over $X$ has finite cohomological dimension. 
    \end{enumerate}
\end{defi}

Note that (2) is a bit different from \cite[Definition 4.19]{Sch24}, but we think that it is natural to make the above definition to show the stability in \Cref{lem:propfincoh}. 

\begin{lem} \label{lem:propfincoh}
    Let $f\colon Y\to X$ and $g\colon Z \to Y$ be maps of small arc-stacks. Suppose that $f$ is $0$-truncated, proper and has locally finite cohomological dimension. 
    \begin{enumerate}
        \item The base change of $f$ along any map $W \to X$ of small arc-stacks is $0$-truncated, proper and has locally finite cohomological dimension. 
        \item If $f$ has finite cohomological dimension, the cohomological dimension of $Y \times_X \cl{M}_\arc(R)$ is bounded by $d(f)$ for all strictly totally disconnected spaces $\cl{M}_\arc(R) \to X$. 
        \item If $g$ is $0$-truncated, proper and $f$ and $g$ have (resp.\ locally) finite cohomological dimensions, $f \circ g$ is $0$-truncated, proper and has (resp.\ locally) finite cohomological dimension up to $d(f) + d(g)$. 
        \item If $f\circ g$ is $0$-truncated, proper and has locally finite cohomological dimension, the same holds for $g$. 
    \end{enumerate}
\end{lem}
\begin{proof}
    First, (1) immediately follows from the definition, and (2) is essentially proved in \cite[Theorem 4.20]{Sch24}. Let us explain more about (2) for the reader's convenience. 

    For (2), we may assume that $X = \cl{M}_\arc(R)$ is strictly totally disconnected. We show that $\tau^{\geq d(f) + 1} f_*\cl{F} = 0$ for every finitary sheaf $\cl{F}$ on $Y$. It is enough to show that the $i$-th cohomology sheaf $\cl{H}^i(f_*\cl{F})$ vanishes for every $i > d(f)$. By \cite[Theorem 4.20]{Sch24}, $f_*\cl{F}$ is finitary and commutes with the base change in $X$. Thus, for every section $s\in \cl{H}^i(f_*\cl{F})(R)$, its stalk $s_x$ at a closed point $x\in X$ vanishes, and it implies that $s_U = 0$ for some neighborhood $U \subset \cl{M}(X)$ of $x$. Thus, we have $s = 0$. 

    For (3), we may assume that $f$ and $g$ have finite cohomological dimensions. Let $\cl{F}$ be a finitary sheaf on $Z$. By (1), $g_*\cl{F}$ has cohomological degree up to $d(g)$. Then, $g_*\cl{F}$ is written as an extension of $\cl{H}^i(g_*\cl{F})[-i]$ for $0 \leq i \leq d(f)$. The claim follows since each $f_*\cl{H}^i(g_*\cl{F})$ has cohomological degree up to $d(f)$ by (1). 

    For (4), let $\Gamma_g\colon Z \to Y\times_X Z$ be the graph of $g$ over $X$. By (1), $\pr_1 \colon Y\times_X Z \to Y$ is $0$-truncated, proper and has locally finite cohomological dimension, so it is enough to show that the same holds for $\Gamma_g$ by (3). It is easy to see that $\Gamma_g$ is $0$-truncated and quasicompact. Moreover, $\Gamma_g$ is an injection since $f$ is $0$-truncated. Thus, $\Gamma_g$ is quasiseparated, and every fiber of a geometric point is empty or isomorphic to the point itself. Thus, $\Gamma_g$ has locally finite cohomological dimension. 
\end{proof}

For our later use, we present a criterion for finite cohomological dimension. 

\begin{lem} \label{lem:qcqscolimpres}
    For any filtered diagram $(\cl{F}_\lambda)_{\lambda  \in \Lambda}$ of finitary sheaves (in the sense of \Cref{conv:finsheaf}) on a qcqs arc-stack $X$, we have $(\colim_{\lambda \in \Lambda} \cl{F}_\lambda)(X) \cong \colim_{\lambda \in \Lambda} \cl{F}_\lambda(X)$. 
\end{lem}
\begin{proof}
    Since $X$ is qcqs, we can take a hypercovering $\cl{M}_\arc(A^\bullet) \to X$ with $A^\bullet$ strictly totally disconnected. By \cite[Theorem 4.1]{Sch24}, $\cl{F}$ is hypercomplete, so $\cl{F}(X) \cong \lim \cl{F}(A^\bullet)$. By loc. cit., each $\cl{F}(A^\bullet)$ commutes with small colimits in $\cl{F}$. Thus, we have
    \[
        (\colim_{\lambda \in \Lambda} \cl{F}_\lambda)(X) \cong \lim_{\Delta} \colim_{\lambda \in \Lambda} \cl{F}_\lambda(A^\bullet) \cong \colim_{\lambda \in \Lambda} \lim_\Delta \cl{F}_\lambda(A^\bullet) \cong \colim_{\lambda \in \Lambda} \cl{F}_\lambda(X). 
    \]
    Here, the second isomorphism holds since each $\cl{F}_\lambda(A^\bullet)$ is static. 
\end{proof}

\begin{prop} \label{prop:fincohexc}
    Let $X$ be a qcqs arc-stack. Let $j \colon U \hookrightarrow X$ be an open immersion and let $i \colon Z \hookrightarrow X$ be its closed complement. If $Z$ and every closed substack of $X$ contained in $U$ have finite cohomological dimensions up to a constant, $X$ has finite cohomological dimension. 
\end{prop}

\begin{proof}

    Since $X$ is qcqs, $\cl{M}(X)$ is compact and Hausdorff by \Cref{lem:M(X)}. Let $\{V_\lambda\}_{\lambda\in \Lambda}$ be the filtered system consisting of open neighborhoods of $\cl{M}(Z) \subset \cl{M}(X)$. Let $\ov{V}_\lambda$ be the closure of $V_\lambda$. Then, it satisfies $\bigcap_{\lambda \in \Lambda} \ov{V}_\lambda = \cl{M}(Z)$. Let $W_\lambda \subset X$ be the closed substack corresponding to $\cl{M}(X)-V_\lambda$ and let $W_\lambda^\circ \subset X$ be the open substack corresponding to $\cl{M}(X) - \ov{V}_\lambda$. Let $j_\lambda \colon W_\lambda \hookrightarrow X$ be the closed immersion and let $j^\circ_\lambda \colon W_\lambda^\circ \hookrightarrow X$ be the open immersion. 

    Since $\bigcap_{\lambda \in \Lambda} \ov{V}_\lambda = \cl{M}(Z)$, $\colim_{\lambda \in \Lambda} j^\circ_{\lambda !} j_{\lambda}^{\circ*} \to j_!j^*$ is an isomorphism. For a finitary sheaf $\cl{F}$ on $X$, we have a short exact sequence of finitary sheaves by \cite[Proposition 4.25]{Sch24}
    \[
        \colim_{\lambda \in \Lambda} j^\circ_{\lambda !} j_{\lambda}^{\circ*}\cl{F} \to \cl{F} \to i_* i^* \cl{F}. 
    \]
    By \Cref{lem:qcqscolimpres}, we get a short exact sequence
    \[
        \colim_{\lambda \in \Lambda} (j^\circ_{\lambda !} j_{\lambda}^{\circ*}\cl{F})(X) \to \cl{F}(X) \to(i^*\cl{F})(Z). 
    \]
    Since $Z$ has finite cohomological dimension, $(i^*\cl{F})(Z)$ has finite cohomological dimension. On the other hand, since $\cl{M}(W^\circ_\lambda) \cap V_\lambda = \phi$, each $(j^\circ_{\lambda !} j_{\lambda}^{\circ*}\cl{F})(X)$ equals $(j_\lambda^*j^\circ_{\lambda !} j_{\lambda}^{\circ*}\cl{F})(W_\lambda)$. Thus, it has finite cohomological dimension bounded by the cohomological dimension of $W_\lambda$. 
\end{proof}

\begin{const} \label{const:6ff_qcqs}
    Let $\qcqsvShf$ denote the site of qcqs $v$-sheaves and let $E^{\qcqs}$  be the collection of morphisms $f$ in $\qcqsvShf$ such that $a'^*f$ has finite cohomological dimension. Thanks to \cite[Corollary 10.2]{Sch24}, we may apply \Cref{prop:3_functor_formalism_from_IP} for $I = \{\id\}$ and $P = E^\qcqs$ to get a sheafy presentable six functor formalism
    \[
        \Corr(\qcqsvShf, E^\qcqs) \to \Pres^L, \quad X \mapsto \cl{D}_\mot^\oc(X) = \cl{D}_\mot(a'^*X). 
    \]
    We extend it by applying \Cref{prop:extension_6ff_frome_site} and restrict the output to the full subcategory $\vStack \subset \Shv(\qcqsvShf)$. Then, we get a six functor formalism
    \[
        \Corr(\vStack, E) \to \Pres^L, \quad X \mapsto \cl{D}_\mot^\oc(X). 
    \]
    We say that a morphism in $\vStack$ is $!$-able if it lies in $E$. 
\end{const}

\begin{rmk} \label{rmk:diff!}
    We need to be cautious about the fact that the class of $!$-able morphisms are different from the torsion case studied in \cite{Sch17}. In loc. cit., a map of small $v$-stacks is $!$-able if it is compactifiable, representable in locally spatial diamonds and locally $\text{dim.trg} f <\infty$. 
\end{rmk}

\subsection{Base change to an $\ell$-adic coefficient}

In this section, we take the base change of the motivic six functor formalism constructed as \Cref{const:6ff_qcqs} to an $\ell$-adic coefficient following \cite[Notation]{HHS24}. 

Let $\ast = \Spd(k)$ be the final object of $\vStack$. For each small $v$-stack $X$, $\cl{D}^\oc_\mot(X)$ is naturally endowed with a module structure over $\cl{D}_\mot(\ast)$ (see \cite[Lemma 3.2.5]{HM24}). %, i.e. it factors through $\text{LMod}_{\cl{D}(*)}(\Pr^L)$ (see \cite[Lemma 3.2.5]{HM24}). 
Take a $\bb{Z}_\ell$-algebra $\Lambda$ as a coefficient ring and consider the $\ell$-adic realization
\[
\cl{D}_\mot(\ast) \to \cl{D}(\bb{Z}_\ell) \to \cl{D}(\Lambda).
\]
It is a symmetric monoidal functor. Here, the first functor is constructed as follows. First, by taking the limit of the \'{e}tale realization \cite[Proposition 12.4]{Sch24} 
\[ 
%\rho_{\ell^n} : 
    \cl D_\mot(*) = \cl{D}_\mot(\ast, \bb{Z}) \ra \cl{D}_\mot(\ast, \bb{Z}/\ell^n) \cong \cl D(\mathbb Z/\ell^n ) 
\] 
with $\bb{Z}/\ell^n$-coefficients over $n\geq 1$, we get a functor $\cl{D}_\mot(\ast) \to \cl{D}(\Spf(\bb{Z}_\ell))$.  By \cite[Theorem 11.1]{Sch24}, $\cl{D}_\mot(\ast)$ is compactly generated and compact dualizable generators map to $\Perf(\Spf(\bb{Z}_\ell))$ via the $\ell$-adic realization. The colimit-preserving extension of $\cl{D}_\mot^\omega(\ast) \to \Perf(\Spf(\bb{Z}_\ell)) \cong \Perf(\bb{Z}_\ell)$ provides $\cl{D}_\mot(\ast) \to \cl{D}(\bb{Z}_\ell)$. 

By \cite[Lemma A.0.3]{ALM24}, the base change along $\cl{D}_\mot(\ast) \to \cl{D}(\Lambda)$ provides an $\ell$-adic six functor formalism
\[
    \Corr(\vStack, E) \to \Pres^L_{D_\mot(\ast)} \to \Pres^L_{\Lambda}. 
\]
Here, $\Pres^L_{\Lambda}$ (resp.\ $\Pres^L_{D_\mot(\ast)}$) denotes the category of $\Lambda$-linear (resp.\ $\cl{D}_\mot(\ast)$-linear) presentable stable $\infty$-categories.

\begin{prop} \label{prop:linearity_statement}
    Let $\cl{D}^\oc(X,\Lambda) = \cl{D}^\oc_\mot(X) \otimes_{\cl{D}_\mot(\ast)} \cl{D}(\Lambda)$ for a small $v$-stack $X$. Then, 
    \[
        \Corr(\vStack, E) \to \Pres^L_{\Lambda}, \quad X \mapsto \cl{D}^\oc(X,\Lambda)
    \]
    is a sheafy $\Lambda$-linear presentable six functor formalism. 
\end{prop}
\begin{proof}
    It is enough to verify the sheafyness. By construction, the functor $X \mapsto \cl{D}_\mot^\oc(X)$ is sheafy. By \cite[Theorem 11.1]{Sch24}, $\cl{D}_\mot(\ast)$ is rigid. Since $\cl{D}(\Lambda)$ is compactly generated, it is dualizable as a $\cl{D}_\mot(\ast)$-module by \cite[Proposition 9.4.4]{GR17I}. Thus, the base change $\Pres^L_{\cl{D}_\mot(\ast)}\to \Pres^L_{\Lambda}$ commutes with small limits, so the claim follows. 
\end{proof}

For each small $v$-stack $X$, let $\underline{\Lambda}_X \in \cl{D}^\oc(X, \Lambda)$ denote the pullback of $\Lambda \in \cl{D}(\Lambda) = \cl{D}^\oc(\ast, \Lambda)$ along $X \to \ast$. When $X$ is clear from the context, $\underline{\Lambda}_X$ is abbreviated to $\underline{\Lambda}$ by abuse of notation. %Note that $\Lambda$ is endowed with a discrete topology when we regard it as a topological ring. 

\begin{lem} \label{lem:locprof!}
    Let $S$ be a locally profinite space and let $\pi_{S} \colon \underline{S} \to \ast$ be the projection map. Then, $\pi_{S!}\underline{\Lambda} = C_c^\infty(S, \Lambda)$ in $\cl{D}^\oc(\ast, \Lambda) = \cl{D}(\Lambda)$. 
\end{lem}

Here, a topological space is locally profinite if it is Hausdorff, locally compact and totally disconnected. In particular, locally profinite spaces are unions of profinite spaces. 

\begin{proof}
    Since $\pi_{S!}$ commutes with small colimits, we may assume that $S$ is profinite. Let us write $S$ as a cofiltered limit $\lim_{i \in I} S_i$ of finite sets. Since $\pi_S$ is proper, we have
    \[
        \pi_{S!} \underline{\bb{Z}} \cong \pi_{S*} \underline{\bb{Z}} \cong \colim_{i\in I} \underline{C(S_i, \bb{Z})}
    \]
    by \cite[Lemma 10.4]{Sch24}. Note that the second isomorphism follows from the commutativity of finitary sheaves with filtered colimits. Then, the claim follows by taking the realization $\cl{D}_\mot(\ast) \to \cl{D}(\Lambda)$. 
\end{proof}

\subsection{Examples of $!$-able maps}

In this section, we verify $!$-ability for several maps of small $v$-stacks. First, we treat closed immersions and partially open immersions. 

\begin{defi}(\cite[Definition 18.4]{Sch17})
    \label{def:partially_proper_v_stacks}
    A separated morphism $f \colon X \ra Y$ of $v$-stacks is \textit{partially proper} if for every perfectoid Huber pair $(R,R^+)$ over $k$ and every diagram
    \begin{center}
        \begin{tikzcd}
            \Spa(R,R^\circ) \ar[d]  \rar & X \ar[d,"f"] \\ 
            \Spa(R,R^+) \rar \ar[ur, dashed] & Y, 
        \end{tikzcd}
    \end{center}
    there exists a (necessarily unique) dotted arrow making the diagram commute. 
\end{defi}
\begin{lem}
    \label{lem:immersion_vstacks_arcstacks} Let $f\colon X \to Y$ be a map of small $v$-stacks. 
    \begin{enumerate}
        \item If $f$ is a closed immersion, then so is ${a'}^{*}f$. 
        \item If $f$ is a partially proper open immersion, then ${a'}^{*}f$ is an open immersion. 
    \end{enumerate}
\end{lem}

\begin{proof}
    Since the assertion is local on target, we may assume that $Y$ is qcqs. 

    First, we prove (1). In this case, $X$ is also qcqs and $\lvert X \rvert \subset \lvert Y \rvert$ is a closed and generalizing subset. In particular, $\lvert X \rvert = \pi^{-1} \pi(\lvert X \rvert)$ for the canonical quotient map $\pi \colon \lvert Y \rvert \to \cl{M}(Y)$. Then, for any perfectoid Tate $k$-algebra $R$, $\cl{M}_{\arc}(R) \to a'^*Y$ factors through $a'^*X$ if and only if $\cl{M}(R) \to \cl{M}(Y)$ factors through $\pi(\lvert X \rvert)$, so $a'^*f$ is closed. 

    Next, we prove (2). Since $f$ is partially proper, we have $\lvert X \rvert = \pi^{-1} \pi(\lvert X \rvert)$. Moreover, $\pi(\lvert X \rvert) \subset \cl{M}(Y)$ is an open subset since $\pi$ is a quotient map. Then, it is easy to see that $a'^*f$ is an open immersion associated with $\pi(\lvert X \rvert)$.
\end{proof}

\begin{prop} \label{prop:clop!-able}
    Closed immersions and partially proper open immersions of small $v$-stacks are $!$-able.
\end{prop}
\begin{proof}
    First, let $i \colon Z \hookrightarrow X$ be a closed immersion. By \Cref{prop:extension_6ff_frome_site} (2), we may assume that $X$ is qcqs. By \Cref{lem:immersion_vstacks_arcstacks}, $a'^*i$ is a closed immersion, so it is qcqs and has cohomological dimension $0$, so it follows that $i$ is $!$-able. 
    
    Let $j \colon U \to X$ be a partially proper open immersion. Again, we may assume that $X$ is qcqs. By \Cref{lem:immersion_vstacks_arcstacks}, $a'^*j$ is an open immersion associated with $\cl{M}(U) \subset \cl{M}(X)$. Let $Z = \cl{M}(X) - \cl{M}(U)$. By \Cref{lem:M(X)}, $\cl{M}(X)$ is compact and Hausdorff, so $Z$ is compact. Then, the filtered system $\{V_\lambda\}_{\lambda\in \Lambda}$ consisting of open neighborhoods of $Z$ satisfies $\bigcap_{\lambda \in \Lambda} V_\lambda = Z$. Let $W_\lambda \subset X$ be the closed subsheaf corresponding to the generalizing closed subset $\pi_X^{-1}(\cl{M}(X)-V_\lambda)$. Since closed immersions are $!$-able, each $W_\lambda \to X$ is $!$-able. Moreover, $\{W_\lambda \to U\}_{\lambda\in \Lambda}$ is a !-cover since each $x\in \cl{M}(U)$ has a neighborhood $U_x$ such that $U_x \cap V_\lambda = \phi$ for some $\lambda \in \Lambda$. Thus, the claim follows from  \Cref{prop:extension_6ff_frome_site} (3). 
\end{proof}

Then, excision makes sense for a partially proper open immersion and its complement. 

\begin{prop} \label{prop:excmot}
    Let $X$ be a small $v$-stack. Let $j \colon U \hookrightarrow X$ be a partially proper open immersion and let $i \colon Z \hookrightarrow X$ be its closed complement. For every $M \in \cl{D}^\oc(X, \Lambda)$, 
    \[
        j_!j^* M \to M \to i_*i^*M
    \]
    is a fiber sequence in $\cl{D}^\oc(X,\Lambda)$. 
\end{prop}

\begin{proof}
    The claim of \cite[Proposition 4.25]{Sch24} naturally extends from $\cl{D}_\fin$ to $\cl{D}_\mot$ by the definition of $\cl{D}_\mot(-)$ \cite[Definition 9.1]{Sch24}. Then, the claim follows from the fact that the base change functor
    \[
        (-)\otimes_{\cl{D}_\mot(*)} \cl{D}(\Lambda) \colon \Pres^{L}_{\st, \cl{D}_\mot(\ast)} \ra \Pres^{L}_{\st, \Lambda}
    \]
    preserves localization sequences (compare with \cite[Lemma 2.24]{haine2022nonabelianbasechangebasechangecoefficients}).  
\end{proof}

Next, we study Banach-Colmez spaces. 

\begin{defi} \label{defi:BC_space}
    Let $S$ be a perfectoid space and let $\cl{E}$ be a vector bundle over $\cl{X}_S$. The Banach-Colmez space $\BC(\cl{E})$ is the small $v$-sheaf sending a perfectoid space $T$ over $S$ to 
    \[
        \BC(\cl{E})(T) = \Gamma(\cl{X}_T, \cl{E}\vert_{\cl{X}_T}). 
    \]
\end{defi}

%For a vector bundle $\cl{E}$ over the Fargues-Fontaine curve, let $\BC(\cl{E})$ denote the Banach-Colmez space associated to $\cl{E}$. %Then, we show the following finiteness. 

\begin{prop} \label{lem:projBCfin}
    Let $C$ be an algebraically closed perfectoid field over $k$ and let $\cl{E}$ be a vector bundle on $\cl{X}_{(C,O_C)}$. Then, we have the following. 
    \begin{enumerate}
        \item Every quasicompact closed subsheaf $Z \subset a'^*\BC(\cl{E})$ has finite cohomological dimension. 
        \item The qcqs arc-sheaf $a'^*(\BC(\cl{E})\backslash\{0\})/\underline{F^\times}$ has finite cohomological dimension. 
    \end{enumerate}
    Moreover, this cohomological dimension depends only on the Newton slope of $\cl{E}$, not on the choice of $C$. 
\end{prop}

Recall that $\cl{E}$ is written as a direct sum of $\cl{O}(\lambda)$'s for $\lambda \in \bb{Q}$ (see \cite[Theorem II.2.14]{FS24}). The Newton slope of $\cl{E}$ refers to the multiset of $\lambda$'s appearing in the summand. 

\begin{proof}
    First, we prove (1) by the induction on the rank of $\cl{E}$. First, suppose that $\cl{E}$ has a decomposition $\cl{E} = \cl{E}_1 \oplus \cl{E}_2$. Since $\cl{M}(Z)$ is compact and Hausdorff, $Z \to a'^*\BC(\cl{E}_1)$ factors through a quasicompact closed subsheaf $Z' \subset a'^*\BC(\cl{E}_1)$. Then, $Z \to Z'$ is $0$-truncated and proper. Now, consider the following diagram. 
    \begin{center}
        \begin{tikzcd}
            Z \ar[r, hook] \ar[d] & a'^*\BC(\cl{E}) \ar[d] \ar[r] & a'^*\BC(\cl{E}_2) \ar[d] \\
            Z' \ar[r, hook] & a'^*\BC(\cl{E}_1) \ar[r] & \cl{M}_\arc(C)
        \end{tikzcd}
    \end{center}
    Here, the right square is Cartesian, so every fiber of $Z \to Z'$ is a quasicompact closed subset of (the base change of) $\BC(\cl{E}_2)$. By the induction hypothesis, both $Z \to Z'$ and $Z'$ have finite cohomological dimension, thus so does $Z$. 

    Now, we are reduced to the case where $\cl{E}$ is indecomposable. Let us write $\cl{E} = \cl{O}(\tfrac{r}{d})$ with $r$ and $d$ coprime. Let $F_d$ be the unramified extension of $F$ of degree $d$ and let $\pi_d \colon \cl{X}_{d, (C,O_C)} \to \cl{X}_{(C,O_C)}$ be the finite Galois cover corresponding to $F_d$. Then, $\cl{E} = \pi_{d*} \cl{O}_{\cl{X}_d(C,O_C)}(r)$, so we have $\BC(\cl{E}) = \BC(\cl{O}_{\cl{X}_d(C,O_C)}(r))$. By passing from $F$ to $F_d$, we may assume that $\cl{E}=\cl{O}(n)$ for some $n \geq 0$.

    For $n = 0$, $\BC(\cl{O}) = \underline{F} \times \Spa(C,O_C)$, so it is an increasing union of profinite covers of $\cl{M}_\arc(C)$. Since $Z$ is quasicompact, it is a closed subspace of a profinite cover of $\cl{M}_\arc(C)$. As in \cite[Corollary 4.18]{Sch24}, every profinite cover of $\cl{M}_\arc(C)$ has cohomological dimension $0$, so $Z$ has cohomological dimension $0$. 
    
    For $n \geq 1$, we have an exact sequence
    \[
        \BC(\cl{O}(n-1)) \to \BC(\cl{O}(n)) \to \bb{A}^{1,\diamond}_{C^\sharp}
    \]
    for an untilt $C^\sharp$ of $C$ in characteristic $0$. Then, $Z \to a'^* \bb{A}^{1,\diamond}_{C^\sharp}$ factors through a quasicompact closed subsheaf $Z' \subset a'^* \bb{A}^{1,\diamond}_{C^\sharp}$. By the same argument as above, the induction hypothesis for $\cl{O}(n-1)$ implies that $Z \to Z'$ is $0$-truncated, proper and has finite cohomological dimension. It is enough to show that $Z'$ has finite cohomological dimension.

    By construction, $a'^*\bb{A}^{1,\diamond}_{C^\sharp} \to \cl{M}_\arc(C)$ is the tilting of $\bb{A}^{1}_{C^\sharp} \to \cl{M}_\arc(C^\sharp)$ (see \cite[Proposition 9.7]{Sch24} for the tilting equivalence). Then, it is enough to prove that every quasicompact closed subsheaf $Z'^\sharp \subset \bb{A}^1_{C^\sharp}$ has finite cohomological dimension. Now, we have 
    \[
        \bb{A}^1_{C^\sharp} = \colim_{T \mapsto \pi T} \cl{M}_\arc(C^\sharp\langle T \rangle_1). 
    \]
    Here, the natural closed embedding $\cl{M}_\arc(C^\sharp\langle T \rangle_1) \subset \bb{A}^1_{C^\sharp}$ is defined by the condition $\lvert T \rvert \leq 1$. Since $Z'^{\sharp}$ is quasicompact, $Z'^\sharp$ factors through some $\cl{M}_\arc(C^\sharp \langle T \rangle_1)$. Since closed immersions have cohomological dimension $0$, the claim follows from the finiteness of cohomological dimension of $\cl{M}_\arc(C^\sharp \langle T \rangle_1)$ (see \cite[Lemma 5.3]{Sch24}).  

    Next, we prove (2) again by the induction on the rank of $\cl{E}$. First, suppose that $\cl{E}$ has a decomposition $\cl{E} = \cl{E}_1 \oplus \cl{E}_2$. Then, we have the following diagram. 
    \begin{center}
        \begin{tikzcd}
            (\BC(\cl{E}_1)\backslash\{0\})/\underline{F^\times} \ar[r, hook] & (\BC(\cl{E})\backslash\{0\})/\underline{F^\times} & U \ar[l, hook'] \ar[d] \\
            & & (\BC(\cl{E}_2)\backslash\{0\})/\underline{F^\times}. 
        \end{tikzcd}
    \end{center}
    Here, $U$ is the complement of $(\BC(\cl{E}_1)\backslash\{0\})/\underline{F^\times}$, and it is easy to see that each fiber of the vertical arrow is isomorphic to (the base change of) $\BC(\cl{E}_1)$. Then, for every quasicompact closed subsheaf $Z \subset a'^*U$, it follows from (1) that $Z \to (\BC(\cl{E}_2)\backslash\{0\})/\underline{F^\times}$ has finite cohomological dimension, so $Z$ has finite cohomological dimension by the induction hypothesis. Thus, $a'^*(\BC(\cl{E})\backslash\{0\})/\underline{F^\times}$ has finite cohomological dimension by \Cref{prop:fincohexc}.  

    Now, we are reduced to the case where $\cl{E}$ is indecomposable. Let us write $\cl{E} = \cl{O}(\tfrac{r}{d})$ with $r$ and $d$ coprime.  Then, $\BC(\cl{E}) \cong \BC(\cl{O}_{\cl{X}_d(C,O_C)}(r))$, and
    \[
        (\BC(\cl{E})\backslash\{0\})/\underline{F^\times} \twoheadrightarrow (\BC(\cl{O}_{\cl{X}_d(C,O_C)}(r))\backslash\{0\})/\underline{F_d^\times}
    \]
    is quasi-pro-\'{e}tale, so it has finite cohomological dimension. By passing to $F_d$, we may assume that $\cl{E} = \cl{O}(n)$ for some $n \geq 0$. 

    We prove by the induction on $n$. The case $n=0$ is trivial. For $n \geq 1$, we have the following diagram. 
    \begin{center}
        \begin{tikzcd}
            (\BC(\cl{O}(n-1))\backslash\{0\})/\underline{F^\times} \ar[r, hook] & (\BC(\cl{O}(n))\backslash\{0\})/\underline{F^\times} & U \ar[l, hook'] \ar[d] \\
            & & (\bb{A}^{1, \diamond}_{C^\sharp}\backslash \{0\})/\underline{F^\times}. 
        \end{tikzcd}
    \end{center}
    Here, $U$ is the complement of $(\BC(\cl{O}(n-1))\backslash\{0\})/\underline{F^\times}$, and it is easy to see that each fiber of the vertical arrow is isomorphic to (the base change of) $\BC(\cl{O}(n-1))$. By repeating the previous argument (and the tilting equivalence), it is enough to show that $(\bb{A}^{1}_{C^\sharp}\backslash \{0\})/\underline{F^\times}$ has finite cohomological dimension. Since $F^\times \cong O_F^\times \times \pi^\bb{Z}$ and $B\bb{Z}$ has finite cohomological dimension, it is enough to show that $(\bb{A}^{1}_{C^\sharp}\backslash \{0\})/\underline{O_F^\times}$ has finite cohomological dimension. 
    
    First, $\bb{A}^{1}_{C^\sharp}\backslash \{0\}$ is an increasing union of $O_F^\times$-stable closed subspaces $W_r = \{\lvert \pi \rvert^{r} \leq \lvert T \rvert \leq \lvert \pi \rvert^{-r}\}$ for $r \geq 1$. It is enough to show that $W_r/\underline{O_F^\times}$ has finite cohomological dimension. Since $W_r/\underline{O_F^\times}$ is qcqs, $\cl{M}(W_r/\underline{O_F^\times})$ is compact and Hausdorff. Moreover, $O_F^\times$ acts freely on $W_r$, so $\cl{M}(W_r/\underline{O_F^\times}) \cong \cl{M}(W_r)/O_F^\times$. Thus, we have the following Cartesian diagram. 
    \begin{center}
        \begin{tikzcd}
            W_r \ar[r] \ar[d] & W_r/\underline{O_F^\times} \ar[d] \\
            \underline{\cl{M}(W_r)} \ar[r] & \und{\cl{M}(W_r)/O_F^\times}
        \end{tikzcd}
    \end{center}
    Then, each fiber of $W_r/\underline{O_F^\times} \to \und{\cl{M}(W_r)/O_F^\times}$ is represented by a residue field of $W_r$, so the map has finite cohomological dimension by the proof of \cite[Lemma 5.3]{Sch24}. Moreover, by the description in \cite[Section 2.4]{Sch24}, $\cl{M}(W_r)/O_F^\times$ is a cofiltered limit of finite trees. Since it has finite cohomological dimension as in the proof of \cite[Lemma 5.3]{Sch24}, $\und{\cl{M}(W_r)/O_F^\times}$ has finite cohomological dimension. Thus, the claim follows as in \cite[Corollary 4.18]{Sch24}. 
\end{proof}

\begin{cor} \label{cor:BC!able}
    Let $S$ be a perfectoid space over $k$ and let $\cl{E}$ be a vector bundle over $\cl{X}_S$. Then, both morphisms
    \[
        \BC(\cl{E}) \to S, \quad (\BC(\cl{E})\backslash \{0\})/\underline{F^\times} \to S
    \]
    are $!$-able. 
\end{cor}
\begin{proof}
    We may assume that $S$ is affinoid by \Cref{prop:extension_6ff_frome_site} (2). We may also assume that $\cl{E}$ has constant rank $n$. Then, $S \to \Bun_{\GL_n}$ has quasicompact image, so the Newton slope of $\cl{E}$ at each geometric point of $S$ is bounded. Then, $a'^*(\BC(\cl{E})\backslash \{0\})/\underline{F^\times} \to a'^*S$ has finite cohomological dimension by \Cref{lem:projBCfin} (2), so $(\BC(\cl{E})\backslash \{0\})/\underline{F^\times} \to S$ is $!$-able. Then, the $!$-ability of $\BC(\cl{E}) \to S$ follows since we have an open immersion
    \[
        \BC(\cl{E}) \hookrightarrow (\BC(\cl{E} \oplus \cl{O})\backslash \{0\})/\underline{F^\times}. 
    \]
\end{proof}

\subsection{Cohomology of Banach-Colmez spaces}

As explained in \cite[Section 2]{Sch25}, a well-known procedure of computing the compactly supported cohomology of Banach-Colmez spaces (see e.g. \cite[Lemma 2.6]{HI24}) works in the motivic six functor formalism. In this section, we present an argument of this computation for completeness. 

\begin{lem} \label{lem:computA1}
    Let $S$ be a perfectoid space and let $S^\sharp$ be an untilt of $S$ over $\bb{Z}_p$. Let $\pi \colon \bb{A}^{1,\diamond}_{S^\sharp} \to S$ be a map of small $v$-sheaves. Then, $\pi$ is cohomologically smooth, $\pi^! \bb{Z} \cong \bb{Z}(1)[2]$ and $\pi_! \pi^! \bb{Z} \cong \bb{Z}$ in the $\cl{D}^\oc_\mot$-formalism. Here, $\bb{Z}(1)$ denotes the Tate twist. 
\end{lem}
\begin{proof}
    We may assume that $S$ is affine. Let $S^\sharp = \Spa(R^\sharp,R^{\sharp +})$. Then, $a'^*\bb{A}^{1, \diamond}_{S^\sharp} = \bb{A}^1_{R^\sharp}$ and $a'^*\pi$ is the tilt of $\pi_{R^\sharp} \colon \bb{A}^1_{R^\sharp} \to \cl{M}_\arc(R^\sharp)$. Since $\bb{A}^1_{R^\sharp}$ is the complement of $\bb{P}^1_{R^\sharp}$ along the section $\cl{M}_\arc(R^\sharp) \to \bb{P}^1_{R^\sharp}$ at $\infty$, $\pi_{R^\sharp}^! \cong \pi_{R^\sharp}^* \otimes \bb{Z}(1)[2]$ by \cite[Proposition 9.3]{Sch24} and the computation of the reduced cohomology of $\pi_{R^\sharp}^* \otimes \bb{Z}(1)[2]$ on $\bb{P}^1_{R^\sharp}$ in loc. cit. implies $\pi_{R^\sharp !} \pi_{R^\sharp}^! \bb{Z} \cong \bb{Z}$ 
\end{proof}

\begin{prop} \label{lem:computBC}
    Let $S$ be a small $v$-stack. Let $\pi \colon \cl{P} \to S$ be a torsor under $\BC(\cl{O}(n))$ for some $n \geq 1$ in the $v$-topology. Then, $\pi$ is cohomologically smooth, $\pi^! \bb{Z} \cong \bb{Z}(n)[2n]$ and $\pi_! \pi^! \bb{Z} \cong \bb{Z}$ in the $\cl{D}^\oc_\mot$-formalism. In particular, $\pi^* \colon \cl{D}_\mot^\oc(S) \to \cl{D}_\mot^\oc(\cl{P})$ is fully faithful. 
\end{prop}
\begin{proof}
    We prove by the induction on $n$. We would like to prove that we have an adjoint pair
    \[
        \pi_! \dashv \pi^* \otimes \bb{Z}(n)[2n]
    \]
    and the unit $\pi_!\bb{Z}(n)[2n] \to \bb{Z}$ is an isomorphism. Since these two functors commute with pullbacks and $\pi$ is surjective, it is enough to prove this after pullbacks along the \v{C}ech nerve of $\pi$. In particular, we may assume that $\cl{P} \cong \BC(\cl{O}(n)) \times S$.
    
    It is enough to treat the absolute case $S = \ast$. For $n=1$, $\BC(\cl{O}(1)) \cong \Spd(k\llbracket t \rrbracket)$, so it is identified with the open unit disk $U = \{ \lvert T \rvert < 1\} \subset \bb{A}^1_k$. Its complement $\bb{P}^1_k - U$ is a closed unit disk, so by the ball invariance, $\pi_! \pi^! \bb{Z}$ computes the reduced cohomology of $\bb{Z}(1)[2]$ on $\bb{P}^1_k$. It follows from the proof of \cite[Proposition 9.3]{Sch24} that $\pi_! \pi^! \bb{Z} \cong \bb{Z}$ and $\pi^! \cong \pi^* \otimes \bb{Z}(1)[2]$. 

    For $n > 1$, we may pass to the \v{C}ech nerve of the cohomologically smooth cover
    \[
        \Spd(k((t)), k\llbracket t \rrbracket) \to \ast. 
    \]
    We may assume that $S = \Spa(R, R^+)$ is an affinoid perfectoid space in characteristic $p$.  
    
    Take an untilt $S^\sharp = \Spa(R^\sharp,R^{\sharp +})$ in characteristic $0$. We have an exact sequence 
    \[
        0 \to \BC(\cl{O}(n-1))\vert_S \to \BC(\cl{O}(n))\vert_S  \xrightarrow{q} \bb{A}^{1, \diamond}_{R^\sharp} \to 0. 
    \]
    Here, $q$ is a torsor under $\BC(\cl{O}(n-1))$. Thus, the induction hypothesis can be applied to $q$, so $\pi_! \pi^! \bb{Z} \cong \bb{Z}$ and $\pi^! \cong \pi^* \otimes \bb{Z}(n)[2n]$ by \Cref{lem:computA1}. Thus, we get the claim. 
\end{proof}

\section{Formulation of the $\cl{A}$-side} \label{sec:Aside}

In this section, we introduce the stack $\Bun_G^X$ of $G$-torsors with an $X$-section over the Fargues-Fontaine curve and the associated unnormalized period sheaf $\cl{P}_{X}$. Throughout this section, we will use the sheaf theory explained in \Cref{sec:sixfunc}. %One of the technical difficulties is in the verification of the $!$-ability of the projection $\Bun_G^X \to \Bun_G$. 
%In \Cref{ssec:IwTateP} and \Cref{ssec:HeckeP}, we compute the unnormalized period sheaves for the Iwasawa-Tate and Hecke periods. 

\subsection{Recollections of $\Bun_G$}
\label{ssec:recollections_BunG}
Let $G$ be a connected reductive group over $F$. In this section, we recall the basic properties of the small $v$-stack $\Bun_G$ introduced in \cite{FS24}. 

\begin{defi} \textup{(\cite[Chapter III]{FS24})}
    The stack sending $S\in \Perf$ to the groupoid of $G$-torsors on $\cl{X}_S$ is a small $v$-stack and denoted by $\Bun_G$. 
\end{defi}

When $S = \Spa(R,R^+)$ is an affinoid perfectoid space over $\bb{F}_q$, there is an algebraic version $X_S^\alg$ of the Fargues-Fontaine curve (see \cite[p.67]{FS24}). There is a morphism $\cl{X}_S \to X_S^\alg$ of locally ringed spaces and the pullback induces the equivalence of the categories of vector bundles (see \cite[Proposition II.2.7]{FS24}). The same applies to the groupoids of $G$-torsors by the Tannakian interpretation. Let $\Perfa$ be the $v$-site of affinoid perfectoid spaces over $k$. As a stack on $\Perfa$, we have a description
\begin{equation}
    \Bun_G(S) = \Map(X_S^\alg, [\ast/G]). \label{eq:BunGalg}
\end{equation}
Here, the mapping anima is taken in the category of stacks over $F$. Let $\breve{F}$ be the completion of the maximal unramified extension of $F$. 

\begin{defi}
    Let $B(G)$ be the set of $\sigma$-conjugacy classes of $G(\breve{F})$. Here, two elements $g, g' \in G(\breve{F})$ are $\sigma$-conjugate if $g' = hg\sigma(h)^{-1}$ for some $h \in G(\breve{F})$.  
\end{defi}

Let $\bar{F}$ be an algebraic closure of $F$ and let $\Gamma=\Gal(\bar{F}/F)$. Fix a Borel pair $T \subset B \subset G_{\bar{F}}$ and let $X_*(T)^+_{\bb{Q}} \subset X_*(T)_{\bb{Q}}$ be the set of dominant rational cocharacters. Then, Newton points and Kottwitz values are described as the following maps
\[
    \nu \colon B(G) \to X_*(T)^{+, \Gamma}_{\bb{Q}} ,\quad 
    \kappa \colon B(G) \to \pi_1(G)_{\Gamma}
\]
and the product $\nu \times \kappa \colon B(G) \to X_*(T)^{+, \Gamma}_{\bb{Q}} \times \pi_1(G)_{\Gamma}$ is injective. For each $b \in B(G)$, $\nu(b)$ is denoted by $\nu_b$, and $b$ is called basic when $\nu_b$ is central. The set of basic $\sigma$-conjugacy classes is denoted by $B(G)_\bas$, and $\kappa$ induces a bijection $B(G)_\bas \cong \pi_1(G)_\Gamma$. 

\begin{prop}\textup{(\cite{Vie24})}
    There is a homeomorphism $\lvert \Bun_G \rvert \simeq B(G)$. Here, $B(G)$ is equipped with the order topology concerning Newton points and Kottwitz values. 
\end{prop}

For each $b \in B(G)$, let $\Bun_G^b = \Bun_G\times_{|\Bun_G|} \crbr{b}  \subset \Bun_G$ be the locally closed substack associated to $b$. Let $i^b \colon \Bun_G^b \hookrightarrow \Bun_G$ denote the locally closed immersion. The fibers of $\kappa$ are exactly the connected components of $\lvert \Bun_G \rvert$. For each $b\in B(G)_\bas$, the associated connected component is denoted by $\Bun_G^{(b)}$. 

\begin{exa}(\cite[Section IX.6.4]{FS24}) \label{exa:BunGm}
    When $G = \bb{G}_m$, $\Bun_{\bb{G}_m}$ is the classifying stack of line bundles on $\cl{X}_S$. By \cite[Theorem III.4.5]{FS24}, we have a concrete description
    \[
        \Bun_{\bb{G}_m} = \bigsqcup_{n \in \bb{Z}} \Bun_{\bb{G}_m, n} \cong \bigsqcup_{n \in \bb{Z}} [\ast/\underline{F^\times}]. 
    \]
    Here, $\Bun_{\bb{G}_m, n}$ classifies lines bundles of degree $n$ (see \cite[Section II.2.4]{FS24} for the notion of degrees). In other words, $\Bun_{\bb{G}_m, n}$ classifies line bundles which are pro-\'{e}tale locally isomorphic to $\cl{O}(n)$ and its underlying point corresponds to the $\sigma$-conjugacy class $[\pi^{-n}] \in B(\bb{G}_m)$. As in \cite[Theorem 3.1 (ii)]{Sch25}, $\cl{D}^\oc(\Bun_{\bb{G}_m, n}, \Lambda) \cong \cl{D}(F^\times, \Lambda)$ for a $\bb{Z}_\ell$-algebra $\Lambda$. 
\end{exa}

\subsection{Relative stacks $\Bun_{G}^{X}$} 

In this section, we introduce the stack $\Bun_G^X$ of $G$-torsors with an $X$-section over the Fargues-Fontaine curve and the unnormalized period sheaf $\cl{P}_{X}$. We will show that $\Bun_G^X$ is an Artin $v$-stack and the projection $\pi_X \colon \Bun_G^X \to \Bun_G$ is $!$-able. 

Throughout this section, $X$ is a normal quasi-projective $G$-variety over $F$. Our convention of a $G$-variety is as follows. 

\begin{conv} \label{conv:Gvar}
    A $G$-variety $X$ over $F$ is an $F$-variety equipped with a left $G$-action. When we consider the quotient stack $[X/G]$, we take the right $G$-action on $X$ given by $x \cdot g = g^{-1} x$ for $x \in X$ and $g \in G$. For an $F$-scheme $S$, $[X/G](S)$ is the groupoid of pairs $(P, s)$ where $P$ is a $G$-bundle over $S$ and $s$ is a section of $P\times^{G} X \to S$. 
\end{conv}

\begin{rmk}
    As we will recall later, $X$ admits a $G$-equivariant ample line bundle under our assumptions by \cite[Theorem 2.5]{Sum75}. Thus, $P \times^G X$ is a scheme for every $G$-bundle $P$ over an $F$-scheme by the descent \cite[Theorem 7]{BLR90}. 
\end{rmk}

Following \eqref{eq:BunGalg}, we define $\Bun_G^X$ as follows. For convenience, we work with the subcategory $\Perfs  \subset \Perfa$ consisting of strictly totally disconnected spaces. It forms a base in the $v$-topology. 

\begin{defi}
    For a normal quasi-projective $G$-variety $X$ over $F$, let $\Bun_G^X$ be the prestack over $\Perfs$ sending $S$ to $\Map(X_S^\alg, [X/G])$. Concretely, $\Bun_G^X(S)$ is the groupoid of pairs $(P, s)$ where $P$ is a $G$-bundle on $X_S^\alg$ and $s$ is a section of $P \times^G X \to X_S^\alg$. 
\end{defi}

\begin{rmk}\label{rmk:partially_proper_}
    As in \cite[Proposition II.2.16]{FS24}, $X_S^\alg$ is independent of the choice of $R^+ \subset R$, so $\Bun^X_G \to \ast$ is partially proper (recall \Cref{def:partially_proper_v_stacks}). 
\end{rmk}

%\begin{rmk}
%    Note that $\Bun_G^X$ equals what is denoted by $\cl{M}_{[X/G]}$ in \cite{Ham22}, where $X$ is assumed to be smooth. % and its cohomologically smooth locus is studied via the Jacobian criterion. Our proof for showing that $\Bun_G^X$ is a $v$-stack is independent of \cite{Ham22}, but we will rely on it when we show that it is an Artin $v$-stack. 
%    We will prove that $\Bun_G^X$ is extended to an Artin $v$-stack with a bit different argument from \cite{Ham22}. %\yuta{I found that untimately we rely on some results in \cite{Ham22}}. 
%\end{rmk}

%\begin{rmk}
    %Precisely, this definition of $\Bun_G^X$ is the \textit{untwisted} analogue of the original definition \cite[(10.6)]{BZSV}. To define the twisting properly, we need a grading on $X$ (see \Cref{defi:grading}) and the choice of a square root of the dualizing complex on a curve. Nevertheless, $\Bun_G^X$ would be the correct one if the grading on $X$ is trivial. We do not \yuta{still!} explore the correct twisting since the dualizing complex on the Fargues-Fontaine curve is not clear. \yuta{it would be $\cl{O}(-2)$ if we believe the analogy between $\bb{P}^1$ and $X_{FF}$. If so, we can use $\cl{O}(-1)$ as a square root of the dualizing complex.}
    %In the original definition \cite[(10.6)]{BZSV}, $X$ is equipped with a grading, and $\Bun_G^X$ is twisted depending on this grading and the choice of a square root of the dualizing complex on a curve. Here, we adopt a simplified version due to the absence of a meaningful dualizing complex (e.g. with Grothendieck-Serre duality) on the Fargues-Fontaine curve. 
%\end{rmk}

\begin{exa} \label{exa:univvb}
    Here are some concrete examples of $\Bun_G^X$. 
    \begin{enumerate}
        \item When $X$ is a homogeneous $G$-variety, i.e. $X = G / H$ for some closed subgroup $H \subset G$, we have $[X/G] \cong [\ast/H]$ and $\Bun_G^X \cong \Bun_H$. 
        \item When $G = \GL_n$ and $X = \bb{A}^n$ is the standard representation of $\GL_n$, $[\bb{A}^n/\GL_n]$ classifies vector bundles with a global section. Let $\BC(\cl{E}^\univ)$ be the Banach-Colmez space of the universal rank $n$ vector bundle over ``the Fargues-Fontaine curve of $\Bun_{\GL_n}$'': 
        \[
            \BC(\cl{E}^\univ)(S) = \BC(\cl{E}_S)
        \]
        for a perfectoid space $S \to \Bun_{\GL_n}$ with $\cl{E}_S$ the associated vector bundle over $\cl{X}_S$. Then, we have an isomorphism 
        \[
            \Bun_{\GL_n}^{\bb{A}^n} \cong \BC(\cl{E}^\univ)
        \]
        over $\Bun_{\GL_n}$. 
    \end{enumerate}
\end{exa}

Each $V\in \Rep(G)$ is an instance of a $G$-variety. Thanks to the quasi-projectivity and the normality of $X$, we have the following structural result for $X$. 

\begin{lem}\textup{(\cite[Theorem 2.5]{Sum75})} \label{lem:sumihiro}
    There is a $G$-equivariant embedding $X \hookrightarrow \bb{P}(V)$ for some $V \in \Rep(G)$. 
\end{lem}

Fix such a $V \in \Rep(G)$ and let $\bar{X} \subset \bb{P}(V)$ be the schematic image of $X$. Since $X$ is Noetherian, $X \hookrightarrow \bar{X}$ is an open immersion and $\bar{X}$ is a projective $G$-variety. A priori, $\Bun_G^X$ is just a prestack on $\Perfs$, but we can define closed and open immersions between prestacks by imposing that every pullback to an object of $\Perfs$ is represented by a closed or open subspace. 

\begin{lem} \label{lem:algvsadic}
    Let $S \in \Perfs$ and let $\pi_S \colon \lvert \cl{X}_S \rvert \to \lvert X_S^\alg \rvert$ denote the natural map. For a closed subset $Z \subset \lvert X_S^\alg \rvert$, $\pi_S^{-1}(Z)$ is empty if and only if $Z$ is empty, and $\pi_S^{-1}(Z) = \lvert \cl{X}_S \rvert$ if and only if $Z = \lvert X_S^\alg \rvert$. 
\end{lem}
\begin{proof}
    For both claims, the converse direction is obvious. First, suppose that $\pi_S^{-1}(Z)= \lvert \cl{X}_S \rvert$. If $Z \neq \lvert X_S^\alg \rvert$, there is a nonzero element $t \in \Gamma(X_S^{\alg}, \cl{O}(n))$ for some $n \geq 0$ such that $t = 0$ on $Z$. As in the proof of \cite[Proposition 11.2.1]{SW20}, $\cl{X}_S$ is sousperfectoid. In particular, $\cl{X}_S$ is reduced, so  $t = 0$ on $\cl{X}_S$. However, $\Gamma(X_S^\alg, \cl{O}(n)) \to \Gamma(\cl{X}_S, \cl{O}(n))$ is an isomorphism (see \cite[Proposition II.2.7]{FS24}), so it contradicts $t \neq 0$. 

    Next, suppose that $\pi_S^{-1}(Z)$ is empty. Since $\pi_0(S)$ is profinite, we have the following commutative diagram of continuous maps
    \begin{center}
        \begin{tikzcd}
            \lvert \cl{X}_S \rvert \ar[r, "\pi_S"] \ar[d, "p_S"'] & \lvert X_S^\alg \rvert \ar[d] \\
            \lvert S \rvert \ar[r] & \pi_0(S). 
        \end{tikzcd}
    \end{center}
    Since Banach-Colmez spaces are extensions of $\bb{G}_a$ and $\underline{F}$, $\Gamma(X_S^\alg, \cl{O}(n))$ commutes with filtered colimits in $S$, so the fiber of $\lvert X_S^\alg \rvert \to \pi_0(S)$ at a point $s \in \pi_0(S)$ is equal to the image of $\lvert X_s^\alg \rvert \to \lvert X_S^\alg \rvert$. Suppose that $Z$ is nonempty. Then, $Z \cap \lvert X_s^\alg \rvert$ is nonempty for some $s \in \pi_0(S)$. Since $S$ is strictly totally disconnected, we reduce to the case $S = \Spa(C,O_C)$ for some perfectoid field $C$ over $k$. Then, $X_S^\alg$ is a curve and $Z$ contains $V_+(\xi)$ for some $\xi \in \Gamma(X_S^\alg, \cl{O}(1))$. However, $\pi_S^{-1}(V_+(\xi))$ is obviously nonempty, so we get a contradiction. 
\end{proof}

\begin{lem} \label{lem:redPV}
    The natural map $\Bun_G^X \to \Bun_G^{\bar{X}}$ is an open immersion and $\Bun_G^{\bar{X}} \to \Bun_G^{\bb{P}(V)}$ is a closed immersion. 
\end{lem}
\begin{proof}
    For the first claim, take a map $S \to \Bun_G^{\bar{X}}$ corresponding to a pair $(P, s \colon X_S^\alg \to P\times^G \bar{X})$. Then, $s^{-1}(P\times^G X)$ is an open subset $U \subset \lvert X_S^\alg \rvert$. For $T \in \Perfs_{/S}$, $T \to S \to \Bun_G^{\bar{X}}$ factors through $\Bun_G^X$ if and only if $x_{T/S}^{-1}(U) = \lvert X_T^\alg \rvert$ where $x_{T/S}\colon \lvert X_T^\alg \rvert \to \lvert X_S^\alg \rvert$. By \Cref{lem:algvsadic}, this is equivalent to $x'^{-1}_{T/S}(\pi_S^{-1}(U)) = \lvert \cl{X}_T \rvert$ where $x'_{T/S}\colon \lvert \cl{X}_T \rvert \to \lvert \cl{X}_S \rvert$. Now, $p_S \colon \lvert \cl{X}_S \rvert \to \lvert S \rvert$ is closed since $\Div^1 \to \ast$ is proper and $\cl{X}_S \cong \Div^1 \times S$, so $\lvert S \rvert - p_S(\lvert \cl{X}_S \rvert - \pi_S^{-1}(U))$ is an open subset of $\lvert S \rvert$. It represents the fiber product $S \times_{\Bun_G^{\bar{X}}} \Bun_G^X$. 

    The second claim follows by the same argument. Here, we additionally need the facts that $P$ is reduced and $p_S$ is open. The first claim follows from the reducedness of $\cl{Y}_S$ because it implies that $X_S^\alg$ is reduced. The second claim follows from the cohomological smoothness of $\Div^1 \to \ast$ (see \cite[Proposition 23.11]{Sch17} and \cite[Proposition II.3.7]{FS24}). 
\end{proof}

In this way, we are reduced to the case $X = \bb{P}(V)$. Then, we have a Cartesian diagram
\begin{equation*}\label{eq:reduction_square}
    \begin{tikzcd}
        \Bun_G^{\bb{P}(V)} \ar[r] \ar[d] & \Bun_{\GL(V)}^{\bb{P}(V)} \ar[d] \\
        \Bun_G \ar[r] & \Bun_{\GL(V)}. 
    \end{tikzcd}
\end{equation*}
\begin{prop} \label{prop:BunGXArtin}
    The prestack $\Bun_G^X$ is a $v$-stack on $\Perfs$. Moreover, it is extended to an Artin $v$-stack on $\Perf$. 
\end{prop}
\begin{proof}
    By \Cref{lem:redPV}, we may assume that $X = \bb{P}(V)$. By the above diagram, we are reduced to the case $G = \GL(V)$ and $X = \bb{P}(V)$ by \cite[Proposition IV.1.8 (i)]{FS24}. However, $[\bb{P}(V) / \GL(V)] \cong [\ast/ P]$ for a maximal parabolic subgroup $P \subset \GL(V)$, so $\Bun_{\GL(V)}^{\bb{P}(V)} \cong \Bun_P$, which is a $v$-stack from Tannakian interpretation. By \cite[Lemma 4.1]{ALB21} and \cite[Proposition IV.1.8 (iii)]{FS24}, $\Bun_P$ is an Artin $v$-stack, so the claim follows. 
\end{proof}

To define the unnormalized period sheaf $\cl{P}_X$ on $\Bun_G$, we need the following property. 

\begin{prop} \label{prop:rel!able}
    The natural map $\pi_X\colon \Bun_G^X \to \Bun_G$ is $!$-able. 
\end{prop}

\begin{proof}
    By definition, $\Bun_G^X \to \ast$ is partially proper, so $\Bun_G^X \to \Bun_G^{\bb{P}(V)}$ is $!$-able by \Cref{prop:clop!-able} and \Cref{lem:redPV}. By the above Cartesian diagram, we are reduced to the case $G = \GL(V)$ and $X = \bb{P}(V)$. In particular, it is enough to show the $!$-ability of $\Bun_P \to \Bun_G$ for a suitable parabolic subgroup $P \subset G$. This is basically follows from the existence of the Drinfeld compactification in \cite{HHS24}. Instead, we will directly argue in our case to ensure the finite cohomological dimension of a compactification to apply the six functor formalism in \Cref{sec:sixfunc}. 

    Let $d = \dim V$. Then, $\Bun_G$ classifies vector bundles $\cl{E}$ of rank $d$ on the Fargues-Fontaine curve and $\Bun_P \to \Bun_G$ classifies line subbundles $ \cl{L} \subset \cl{E}$. We get a decomposition $\Bun_P = \bigsqcup_{n \in \bb{Z}} \Bun_P^n$ with respect to the degree of $\cl{L}$. It is enough to show that $\Bun_P^n \to \Bun_G$ is !-able. 

    Let $Z \to \Bun_G$ be the $v$-stack classifying $(H^0(X_S, \cl{E}_S(-n))\backslash \{0\})/\underline{F^\times}$ over $\cl{E}_S \in \Bun_G(S)$. The natural inclusion $\Bun^n_P \hookrightarrow Z$ sending $ \cl{L} \subset \cl{E}$ to $\cl{O} \hookrightarrow \cl{E} \otimes \cl{L}^\vee$ is a partially proper open immersion by \cite[Lemma 4.1.4]{HHS24}. Moreover, $Z \to \Bun_G$ is proper and representable in spatial diamonds by \cite[Proposition II.2.16]{FS24}. Thus, it is enough to show that $Z \to \Bun_G$ has locally finite cohomological dimension. 

    Take a qcqs $v$-sheaf $Y \to \Bun_G$ and let $\Spa(C,O_C) \to Y$ be an arbitrary geometric point. Let $\cl{E}_C$ be the vector bundle corresponding to $\Spa(C,O_C) \to \Bun_G$. Since $Y$ is quasicompact, the Newton polygon of $\cl{E}_C$ is bounded, so the number of possible isomorphism classes of $\cl{E}_C$ is finite. Thus, the claim follows from \Cref{lem:projBCfin}. 
\end{proof}

%This verifies that geometric Eisenstein series are well-defined in our six functor formalism as claimed in \cite{HHS24}. 

%\begin{cor}
%    Let $P \subset G$ be a parabolic subgroup with a Levi subgroup $M$. Consider the diagram
%    \begin{center}
%        \begin{tikzcd}
%            \Bun_G & \Bun_P \ar[l, "\mfr{p}"'] \ar[r, "\mfr{q}"', yshift = - 2pt] & \Bun_M. \ar[l, "\mfr{i}"', yshift = 2pt]
%        \end{tikzcd}
%    \end{center}
%    Then, $\mfr{p}$, $\mfr{q}$ and $\mfr{i}$ are $!$-able. 
%\end{cor}
%\begin{proof}
%    By \Cref{prop:rel!able}, $\mfr{p}$ is $!$-able since $\mfr{p} = \pi_{G / P}$. Since $\mfr{p} \circ \mfr{i} = \id$ and $E$ is right cancellative by \cite[Lemma 2.1.5 (iii)]{HM24}, it is enough to show that $\mfr{q}$ is $!$-able. 
%\end{proof}

\begin{rmk}
    It follows that $\Bun_P \to \Bun_G$ is !-able for a parabolic subgroup $P \subset G$ since $\Bun_P = \Bun_G^{G/P}$. This verifies that geometric Eisenstein series are well-defined in our six functor formalism as claimed in \cite{HHS24}. 
    %In applications, \Cref{prop:rel!able} is useful in checking !-ability. For example, we can verify the following !-ability. 
    %\begin{enumerate}
    %    \item The map $\Bun_P \to \Bun_G$ is !-able for a parabolic subgroup $P \subset G$ since $\Bun_P = \Bun_G^{G/P}$. This verifies that geometric Eisenstein series are well-defined in our six functor formalism as claimed in \cite{HHS24}. 
    %    \item For a vector bundle $\cl{E}$ over the Fargues-Fontaine curve of a small $v$-stack $X$, $\BC(\cl{E}) \to X$ is !-able. We may assume by restricting to some component that $\cl{E}$ has a constant rank $n$. Then, $\BC(\cl{E}) \to X$ is the base change of $\Bun_{\GL_n}^{\bb{A}^n} \to \Bun_{\GL_n}$ along the map $X \to \Bun_{\GL_n}$ associated with $\cl{E}$, so $\BC(\cl{E}) \to X$ is !-able. 
    %\end{enumerate}
\end{rmk}

\begin{defi} \label{defi:periodsheaf}
    Let $\Lambda$ be a $\bb{Z}_\ell$-algebra. We define the unnormalized period sheaf by 
    \[\cl{P}_{X, \Lambda} = \pi_{X!} \underline{\Lambda}_X\]
    where $\underline{\Lambda}_X$ denotes the constant sheaf on $\Bun_G^X$ with values in $\Lambda$. 
\end{defi}

When $\Lambda$ is clear from context, we abuse the notation to write $\cl{P}_X$ for $\cl{P}_{X,\Lambda}$. 

\subsection{Basic locus of $\Bun_G^X$}

We take $X$ as in the previous section. In this section, we first study the fiber of $\Bun_G^X \to \Bun_G$ over $\Bun_G^1 \subset \Bun_G$. 

\begin{lem} \label{lem:totallydiscon}
    For every $S \in \Perf^\std$, every open cover of $\Spec(C(\lvert S \rvert, F))$ splits. 
\end{lem}
\begin{proof}
    First, $\lvert S \rvert \to \pi_0(S)$ is a maximal Hausdorff quotient, so $C(\lvert S \rvert, F) = C(\pi_0(S), F)$. We will freely use this identification. Then, we have a map
    \[
        i\colon \pi_0(S) \to \Spec(C(\lvert S \rvert, F))
    \]
    sending $s \in \pi_0(S)$ to the closed point corresponding to the maximal ideal $\{f(s) = 0 \lvert f \in C(\lvert S \rvert, F)\}$. For each $f \in C(\lvert S \rvert, F)$, the inverse image of $D(f)$ equals the open locus $\{ f(s) \neq 0\}$, so the map is continuous. 

    Since $\pi_0(S)$ is profinite, there is a natural continuous map 
    \[
        q \colon \Spec(C(\lvert S \rvert, F)) \to \pi_0(S). 
    \]
    Its fiber at each $s \in \pi_0(S)$ is the spectrum of a local ring $\colim_{s\in U \subset \pi_0(S)} C(U, F)$. In particular, the image of $i$ equals to the set of closed points. Then, for every open subset $U \subset \Spec(C(\lvert S \rvert, F))$, we have $q^{-1}(i^{-1}(U)) \subset U$. 

    Take an open covering $\{U_j \}_{j \in J}$ of $\Spec(C(\lvert S \rvert, F))$. Then, $\{q^{-1}(i^{-1}(U_j))\}_{j \in J}$ becomes its refinement. Since $\pi_0(S)$ is profinite, the open covering $\{i^{-1}(U_j)\}_{j \in J}$ of $\pi_0(S)$ splits. Thus, the open covering $\{q^{-1}(i^{-1}(U_j))\}_{j \in J}$ splits. 
\end{proof}

\begin{lem} \label{lem:imgpS}
    For every $S \in \Perfs$, the natural map
    \[
        p_S \colon X_S^\alg \to \Spec(C(\lvert S \rvert, F))
    \]
    does not factor through any proper locally closed subscheme of $\Spec(C(\lvert S \rvert, F))$. 
\end{lem}
\begin{proof}
    By \Cref{lem:totallydiscon}, the image of $p_S$ contains all closed point. Thus, $p_S$ does not factor through any proper open subscheme. Moreover, $p_S$ does not factor through any proper closed subscheme since $\Gamma(X_S^\alg, \cl{O}) \cong C(\lvert S \rvert, F)$. 
\end{proof}

\begin{lem} \label{lem:undsch}
    For every $S \in \Perf^\std$, $\underline{X(F)}(S) \cong X(C(\lvert S \rvert, F))$. 
\end{lem}
\begin{proof}
    By \Cref{lem:totallydiscon}, it is enough to show that $\underline{U(F)}(S) \cong U(C(\lvert S \rvert , F))$ for every affine scheme $U$ of finite type over $F$. Fix a closed immersion $U \hookrightarrow \bb{A}^n$. First, we have
    \[
        \underline{\bb{A}^n(F)}(S) \cong \underline{F}(S)^n = (C(\lvert S \rvert, F))^n \cong \bb{A}^n(C(\lvert S \rvert ,F)). 
    \]
    Then, we need to see that the closed subset $\underline{U(F)}(S) \subset \underline{\bb{A}^n(F)}(S)$ maps to $U(C(\lvert S \rvert, F))$. 
    
    Let $x \in \bb{A}^n(C(\lvert S \rvert ,F))$ and let $x_s \in \bb{A}^n(F)$ be the restriction to a point $s \in \pi_0(S)$. For a function $f$ on $\bb{A}^n$, $f(x) = 0$ if and only if $f(x_s) = 0$ for every $s \in \pi_0(S)$. Thus, $x \in U(C(\lvert S \rvert ,F))$ if and only if $x_s \in U(F)$ for every $s \in \pi_0(S)$. Thus, $\underline{U(F)}(S)$ maps exactly to $U(C(\lvert S \rvert, F))$. 
\end{proof}

\begin{prop}\label{prop:trivfib}
    Let $\Bun_G^{1, X} = \Bun_G^X \times_{\Bun_G} \Bun_G^1$. There is a closed and open immersion 
    \[
        i_X^1\colon [\underline{X(F)}/\underline{G(F)}] \hookrightarrow \Bun_G^{1,X}
    \]
    over $\Bun_G^1 \cong [\ast/\underline{G(F)}]$ and it is an isomorphism if $X$ is quasi-affine. 
\end{prop}
\begin{proof}
    First, we construct $i_X^1$. For each $S \in \Perf^\std$, the pullback along $p_S$ provides $\Bun_G^1 \cong [\ast/\underline{G(F)}]$ identifying $\Bun_G^1(S)$ with the groupoid of trivial $G$-torsors over $\Spec(C(\lvert S \rvert, F))$. 
    
    By \Cref{lem:undsch}, $[\und{X(F)}/\und{G(F)}](S)$ is the groupoid of trivial $G$-torsors $P$ on $\Spec(C(\lvert S \rvert , F))$ equipped with a section $s\in (P \times^G X)(C(\lvert S \rvert , F))$. Then, we get a desired map
    \[
        i_X^1\colon [\underline{X(F)}/\underline{G(F)}] \to \Bun_G^{1,X}, \quad (P, s) \mapsto (p_S^*P, p_S^*(s)). 
    \]

    Now, we show that $i_X^1$ is an isomorphism when $X$ is quasi-affine. Let $U = \Spec(\Gamma(X, \cl{O}))$. Then, $X \hookrightarrow U$ is a $G$-equivariant open immersion. Let $S \in \Perf^\std$ and let $P$ be a trivial $G$-torsor on $\Spec(C(\lvert S \rvert, F))$. The fiber of $\Bun_G^{1, U}(S) \to [\ast/\und{G(F)}](S)$ over $p_S^*P$ is the set of sections of $p_S^*P \times^{G} U \to X_S^\alg$. It is bijective to the set of maps $X_S^\alg \to P \times^G U$ over $C(\lvert S \rvert, F)$. Since $P \times^G U$ is affine, such a map uniquely factor through $p_S$. Thus, $i^1_U$ is an isomorphism.  
    %$G$-equivariant maps $p_S^* P \to U$. Since $P$ is flat over $\Spec(C(\lvert S \rvert, F))$,
    %\[
    %    \Gamma(p_S^*P, \cl{O}) \cong \Gamma(P, \cl{O}) \otimes_{C(\lvert S \rvert, F)} \Gamma(X_S^\alg, \cl{O}) \cong \Gamma(P, \cl{O}). 
    %\]
    %Then, every map $p_S^*P \to U$ factors uniquely as $p_S^*P \to P \to U$ since $U$ is affine. It follows that $i^1_U$ is an isomorphism. 

    Now, for every map $f\colon X_S^\alg \to P \times^G X$ over $C(\lvert S \rvert, F)$, there is a unique dotted arrow making the following diagram commute. 
    \begin{center}
        \begin{tikzcd}
            X_S^\alg \ar[r, "f"] \ar[d, "p_S"] & P \times^G X \ar[d, hook] \\
            \Spec(C(\lvert S \rvert, F)) \ar[r, dotted] & P \times^G U. 
        \end{tikzcd}
    \end{center}
    By \Cref{lem:imgpS}, the dotted arrow factors through $P \times^G X$. Thus, $i_X^1$ is an isomorphism. 

    We turn to the general case. Take a $G$-equivariant embedding $X\hookrightarrow \bb{P}(V)$ by \Cref{lem:sumihiro}. Since $[\bb{P}(V)/G] \cong [(V-\{0\})/G \times \bb{G}_m]$, $\Bun_G^{\bb{P}(V)} \cong \Bun_{G\times \bb{G}_m}^{V - \{0\}}$. Consider the following diagram. 
    \begin{center}
        \begin{tikzcd}
            \lbrack \und{V-\{0\}}/\und{G(F) \times F^\times} \rbrack \ar[r, "\cong"] \ar[d, "i^1_{V-\{0\}}"] & \lbrack \und{\bb{P}(V)}/\und{G(F)} \rbrack \ar[d, "i_{\bb{P}(V)}^1"] \\
            \Bun_{G \times \bb{G}_m}^{V-\{0\}} \ar[r, "\cong"] \ar[d] & \Bun_G^{\bb{P}(V)} \ar[d] \\
            \Bun_{G \times \bb{G}_m} \ar[r] & \Bun_G. 
        \end{tikzcd}
    \end{center}
    Since $\Bun_{G \times \bb{G}_m}^1 \hookrightarrow \Bun_G^1 \times \Bun_{\bb{G}_m}$ is a connected component, $\Bun_{G \times \bb{G}_m}^{1, V-\{0\}} \hookrightarrow \Bun_G^{1, \bb{P}(V)}$ is a closed and open immersion. Since $V-\{0\}$ is quasi-affine, $i_{V-\{0\}}^1$ is an isomorphism. Thus, $i_{\bb{P}(V)}^1$ is a closed and open immersion. It remains to see that
    \begin{center}
        \begin{tikzcd}
            \lbrack \und{X(F)}/\und{G(F)} \rbrack \ar[r, hook] \ar[d, "i_X^1"] & \lbrack \und{\bb{P}(V)(F)}/\und{G(F)} \rbrack \ar[d, "i_{\bb{P}(V)}^1"]\\
            \Bun_G^X \ar[r, hook] & \Bun_G^{\bb{P}(V)}. 
        \end{tikzcd}
    \end{center}
    is a fiber square. It is enough to show that for every commutative diagram
    \begin{center}
        \begin{tikzcd}
            X_S^\alg \ar[r, "f"] \ar[d, "p_S"] & P \times^G X \ar[d, hook] \\
            \Spec(C(\lvert S \rvert, F)) \ar[r, "s"] & P \times^G \bb{P}(V), 
        \end{tikzcd}
    \end{center}
    $s$ factors through $P \times^G X$. This follows from \Cref{lem:imgpS}. %Thus, $i_X^1$ is a closed and open immersion in general. 
\end{proof}

\begin{cor} \label{cor:computetriv}
    Let $X$ be a normal quasi-projective $G$-variety over $F$. Then, $C_c^\infty(X(F), \Lambda)$ is a direct summand of $i^{1*}\cl{P}_X$. Moreover, $i^{1*}\cl{P}_X \cong C_c^\infty(X(F), \Lambda)$ if $X$ is quasi-affine. Here, $G(F)$ acts on $C_c^\infty(X(F), \Lambda)$ by the left translation. 
\end{cor}
\begin{proof}
    It follows from \Cref{prop:trivfib} since the lower shriek of $\und{\Lambda}$ along $[\und{X(F)}/\und{G(F)}] \to [\ast/\und{G(F)}]$ is $C_c^\infty(X(F), \Lambda)$ by \Cref{lem:locprof!}. 
\end{proof}

\begin{cor} \label{cor:XFdecomposition}
    Let $H \subset G$ be a connected reductive subgroup and let $X = G / H$. Then, $\pi_X$ is identified with
    \[
       \pi_X: \Bun_G^X \cong  \Bun_H \to \Bun_G, 
    \]
    and we have $\pi_{X}^{-1}(B(G)_\bas) \subset B(H)_\bas$, where the induced map on the underlying spaces is also denoted by $\pi_X$. Moreover, we have a $G(F)$-equivariant isomorphism
    \[
        X(F) \cong \coprod_{b \in  \pi_X^{-1}(1)} G(F) / H_b(F). 
    \]
\end{cor}
\begin{proof}
    Since $[X / G] \cong  [\ast / H]$, we have $\Bun_G^X \cong \Bun_H$. For each $b \in B(H)$, its Newton point is central if that of $\pi_X(b)$ is central. Thus, we get $\pi_X^{-1}(B(G)_\bas) \subset B(H)_\bas$ and we have 
    \[
        \Bun^{1, X}_G\cong \coprod_{b \in \pi_X^{-1}(1)} [\ast / H_b(F)]. 
    \]
    Since $H$ is reductive, $X$ is affine, so the last claim follows from \Cref{prop:trivfib}. 
\end{proof}

Next, we study the fiber of $\Bun_G^X \to \Bun_G$ over $\Bun_G^b \subset \Bun_G$ for  $[b] \in H^1(F, G)$. Note that $H^1(F, G) \subset B(G)_\bas$ and it is identified with the torsion part of $\pi_1(G)_{\Gamma}$ via the Kottwitz value $\kappa$. In particular, $H^1(F, G)$ is finite. 

Fix a descent representative $b \in [b]$. It means that $b$ is defined over a finite unramified extension $F_r / F$ of degree $r$ and satisfies
\[
    b \sigma(b) \cdots \sigma^{r-1}(b) = 1. 
\]
The Frobenius action $b\sigma$ on $X \otimes F_r$ provides a descent datum of $X \otimes F_r$ over $F$. Since $X$ admits a $G$-equivariant ample line bundle, this descent datum is effective. 

\begin{defi}
    For each $[b] \in H^1(F, G)$, let $X_b$ be a normal quasi-projective $G_b$-variety given by descending $X \otimes F_r$ by the Frobenius action $b\sigma$. We say that $X_b$ is the inner form of $X$ associated to $b$. 
\end{defi}

Recall that we have an isomorphism
\[
    \iota^b \colon \Bun_G \cong \Bun_{G_b}, \quad \cl{P} \mapsto \cl{P}^b
\]
for each $b \in B(G)_\bas$. For each $\cl{P} \in \Bun_G(S)$, let $\varphi_\cl{P}$ denote the Frobenius action of $\cl{P}\vert_{\cl{Y}_S}$. Then, $\cl{P}^b \in \Bun_{G_b}(S)$ is given by the conditions
\[
    \cl{P}_b\vert_{\cl{Y}_S} = \cl{P}\vert_{\cl{Y}_S}, \quad
    \varphi_{\cl{P}_b}(p) = \varphi_{\cl{P}}(p)b^{-1}. 
\]

\begin{prop} \label{prop:innerformBunGX}
    For each $[b] \in H^1(F, G)$, there is an isomorphism
    $
    \iota_X^b \colon \Bun_G^X \cong \Bun_{G_b}^{X_b}
    $
    making the following diagram commute: 
    \begin{center}
        \begin{tikzcd}
            \Bun_G^X \ar[r, "\iota_X^b"] \ar[d, "\pi_X"] & \Bun_{G_b}^{X_b} \ar[d, "\pi_{X_b}"] \\
            \Bun_G \ar[r, "\iota^b"] & \Bun_{G_b}. 
        \end{tikzcd}
    \end{center}
\end{prop}
\begin{proof}
    Take $r \geq 1$ as previously and let $G_r = G \otimes F_r$ and $X_r = X \otimes F_r$. For each $\cl{P}\in \Bun_G(S)$ and $\cl{P}^b \in \Bun_{G_b}(S)$, there is a natural identification
    \[
        \cl{P}\vert_{\cl{Y}_S} \times^{G} X = \cl{P}\vert_{\cl{Y}_S} \times^{G_r} X_r = \cl{P}^b\vert_{\cl{Y}_S} \times^{G_b} X_b. 
    \]
    By the description of $\varphi_{\cl{P}_b}$, the Frobenius actions on both sides are equal, so it provides an identification of fibers of $\pi_X$ and $\pi_{X_b}$ over $S$. Thus, the claim follows. 
\end{proof}
\begin{cor} \label{cor:innerformBunGXexp}
    For each $[b] \in H^1(F, G)$, let $\Bun_G^{b, X} = \Bun_G^X \times_{\Bun_G} \Bun_G^b$. There is a closed and open immersion 
    \[
        i_X^b\colon [\underline{X_b(F)}/\underline{G_b(F)}] \hookrightarrow \Bun_G^{b,X}
    \]
    over $\Bun_G^b \cong [\ast/\underline{G_b(F)}]$ and it is an isomorphism if $X$ is quasi-affine. 
\end{cor}
\begin{proof}
    This follows from \Cref{prop:trivfib} and \Cref{prop:innerformBunGX}. Note that if $X$ is quasi-affine, so is $X_b$. 
\end{proof}

\begin{rmk}
    We cannot expect a similar description for basic $\sigma$-conjugacy classes outside $H^1(F, G)$. One can see this from the explicit description for the Iwasawa-Tate and Hecke periods. 
\end{rmk}

\section{Computations on the $\cl{A}$-side} \label{sec:Acomputation}

In this section, we carry out the explicit computation of the unnormalized period sheaf $\cl{P}_{X}$ for the dual pairs associated with the Iwasawa-Tate and Hecke periods.
%Building on the geometric properties of the relative stack $\Bun_G^{X}$ established in \Cref{sec:Aside}, we determine the structure of $\cl{P}_{X} = \pi_{X!} \underline{\Lambda}$ as an object in $\cl{D}^\oc(\Bun_G, \Lambda)$.
For the Iwasawa-Tate period, we calculate the compactly supported cohomology of the relevant Banach-Colmez spaces directly.
For the Hecke period, we employ the geometric Eisenstein series and the unfolding technique to describe the period sheaf.
%These computations provide the automorphic counterparts to the spectral $L$-sheaves computed in \Cref{sec:Bside}, necessary for the verification of the normalized period conjecture in \Cref{sec:cmparison}.

\subsection{The Iwasawa-Tate period} \label{ssec:IwTateP}

In this section, we compute the unnormalized Iwasawa-Tate period sheaf $\cl{P}_{X}$, where $G = \bb{G}_m$ and $X=\bb{A}^1$ with a standard $\bb{G}_m$-action. In this case, we have a simple description of $\Bun_{\bb{G}_m}^{\bb{A}^1}$. 

\begin{lem}
    %Let $\cl{L}^\univ$ be the universal line bundle over the Fargues-Fontaine curve of $\Bun_{\bb{G}_m}$. Then, 
    We have $\Bun_{\bb{G}_m}^{\bb{A}^1} \cong \BC(\cl{L}^\univ)$ (see \Cref{exa:univvb} (2) for the terminology). 
\end{lem}
\begin{proof}
    The stack $[\bb{A}^1/\bb{G}_m]$ sends a scheme $S$ to the groupoid of a line bundle $\cl{L}$ over $S$ with a section $s \in \Gamma(S, \cl{L})$ (see \Cref{conv:Gvar}). Thus, $\Bun_{\bb{G}_m}^{\bb{A}^1}(S) = \BC(\cl{L}^\univ)(S)$ for $S \in \Perfs$ by definition. The claim follows since $\Perfs \subset \Perf$ is a basis of the $v$-topology. 
\end{proof}

Recall from \Cref{exa:BunGm} that we have a decomposition $\Bun_{\bb{G}_m} = \bigsqcup_{n\in \bb{Z}} \Bun_{\bb{G}_m, n}$. Let $\cl{P}_{\bb{A}^1, n} \in \cl{D}^\oc(\Bun_{\bb{G}_m, n}, \Lambda) \cong \cl{D}(F^\times, \Lambda)$ denote the restriction of $\cl{P}_{\bb{A}^1}$. Then, we have the following description. 

\begin{prop} \label{prop:IwTateperiod}
    For each integer $n$, we have
    \[
        \cl{P}_{\bb{A}^1,n} = \Lambda_\triv (n < 0), \cl{P}_{\bb{A}^1,0} = C_c^\infty(F, \Lambda), \cl{P}_{\bb{A}^1,n} = \Lambda_\norm [-2n] (n > 0). 
    \]
    Here, $\Lambda_\triv$ denotes the trivial character of $F^\times$ and $\Lambda_\norm$ denotes the norm character sending $x \in F^\times$ to $\lvert x \rvert = q^{-\ord(x)}$. Moreover, $F^\times$ acts on $C_c^\infty(F,\Lambda)$ by the left translation. %That is, $(g \cdot f)(x) = f(g^{-1}x)$ for $g\in F^\times$ and $f \in C_c^\infty(F,\Lambda)$. 
\end{prop}
\begin{proof}
    Let $\Bun_{\bb{G}_m, n}^{\bb{A}^1} = \pi_{\bb{A}^1}^{-1}(\Bun_{\bb{G}_m, n})$. 
    For $n < 0$, $\cl{L}^\univ$ has negative degree on $\Bun_{\bb{G}_m, n}$, so $\Bun_{\bb{G}_m, n}^{\bb{A}^1} \to \Bun_{\bb{G}_m, n}$ is an isomorphism by \cite[Proposition II.3.4 (i)]{FS24}. Thus, $\cl{P}_{\bb{A}^1,n} = \Lambda_\triv$. 

    For $n=0$, the claim follows from \Cref{cor:computetriv}. 
    %$\Bun_{\bb{G}_m, 0}^{\bb{A}^1} \times_{\Bun_{\bb{G}_m, 0}} \ast = \BC(\cl{O}) = \underline{F}$, so the underlying $\Lambda$-module of $\cl{P}_{\bb{A}^1, 0}$ is $C_c^\infty(F, \Lambda)$ by \Cref{lem:locprof!}. 
    %the lower shriek of the constant sheaf $\underline{\Lambda}$ along $\underline{F} \to \ast$. First, for every profinite set $S$, the pushforward of $\underline{\bb{Z}}$ along $\underline{S} \to \ast$ is $\underline{C(S, \bb{Z})}$ in the $\cl{D}_\mot$-formalism. Then, the same holds with coefficient $\Lambda$ by base change, so we get $\cl{P}_{\bb{A}^1, 0} = \colim_{n \in \bb{Z}} C(\pi^{-n} O_F, \Lambda) = C_c^\infty(F, \Lambda)$.
    %The right action of $g\in F^\times$ on $\BC(\cl{O}) = \underline{F}$ is the multiplication by $g^{-1}$, so the action of $F^\times$ on $C_c^\infty(F, \Lambda)$ is given by the left translation.  

    For $n > 0$, $\Bun_{\bb{G}_m, n}^{\bb{A}^1} \times_{\Bun_{\bb{G}_m, n}} \ast = \BC(\cl{O}(n))$. Let $\pi_n \colon \BC(\cl{O}(n)) \to \ast$. Since the $\cl{D}^\oc$-formalism is the base change of the $\cl{D}^\oc_\mot$-formalism, $\pi_{n!} \Lambda \cong \Lambda[-2n]$ in $\cl{D}^\oc(\ast, \Lambda) \cong \cl{D}(\Lambda)$ by \Cref{lem:computBC}. In particular, $\cl{P}_{\bb{A}^1, n}$ is a shift of a character of $F^\times$. To identify this character, we may assume that $\Lambda = \bb{Z}_\ell$ by the projection formula. Then, we may reduce to the case $\Lambda = \bb{Z}/\ell^m \bb{Z}$ for $m \geq 0$. In this case, the claim is a special case of \cite[Lemma 2.4 (i)]{HI24}.
\end{proof}

%\begin{rmk}
%    In the last step of the proof, the sign of the norm character may not seem compatible with \cite{HI24}. This is due to the change of the $F^\times$-action when one passes from $\cl{L}^\univ$ to $(\cl{L}^\univ)^\vee$: for a vector bundle $\cl{E}$ on a locally ringed space, the natural isomorphism $\Aut(\cl{E}) \cong \Aut(\cl{E}^\vee)$ sends $\alpha$ to $(\alpha^\vee)^{-1}$. In our situation, the right action of $g\in F^\times$ on $\BC(\cl{O}(-n))$ is given by the multiplication by $g$. When $F = \bb{Q}_p$ and $n = -1$, the action of $p$ on $H_c^2(\BC(\cl{O}(1)), \bb{Z}/\ell^m)$ is given by the Frobenius pullback on $H_c^2(\Spa(C\langle T^{1/p^\infty} \rangle), \bb{Z}/\ell^m)$ for a perfectoid field $C$ over $k$, so it is the multiplication by $p$ (see \cite[Lemma 1.3]{Imai23}). 
%\end{rmk}

\subsection{The Hecke period} \label{ssec:HeckeP}

In this section, we compute the unnormalized Hecke period sheaf $\cl{P}_{X}$, where $G = \GL_2$ and $X = \GL_2 / A$. Here, $A = {\scriptsize \left\{ \begin{pmatrix} 1 & 0 \\ 0 & \ast \end{pmatrix} \right\} } \cong \bb{G}_m$. %In this case, $\Bun_G^X$ is isomorphic to $\Bun_{\bb{G}_m}$. 

%\begin{lem}
%    We have $\Bun_{G}^X \cong \Bun_{A}$ and $\Bun_{A} \cong \Bun_G^X \to \Bun_{G}$ sends a line bundle $\cl{L}$ on the Fargues-Fontaine curve to $\cl{O} \oplus \cl{L}$. Here, $\cl{O}$ denotes the structure sheaf. 
%\end{lem}
%\begin{proof}
%    Since $[X/G] \cong [\ast/A]$, the first claim follows from the definition of $\Bun_G^X$. The map $[\ast/A] \to [\ast/G]$ sends an $A$-torsor $P$ to a vector bundle $P\times^A (\std_{G}\vert_A)$ of rank $2$, where $\std_{G}$ denotes the standard representation of $\GL_2$. Since $\std_{G}\vert_A = \triv \oplus \std_A$ with $\std_A$ the standard representation of $A$, $P\times^A (\std_{G}\vert_A) \cong \cl{O} \oplus (P\times^A \std_A)$. 
%\end{proof}

Since $[X/G] \cong [\ast/A]$, $\Bun_G^X \cong \Bun_A$. % $\to \Bun_{G}$. This map is induced by the inclusion $A \hookrightarrow G$. 
Since $A \cong \bb{G}_m$, we have a decomposition  $\Bun_A = \coprod_{ n \in \bb{Z}} \Bun_{A, n}$ and let $\pi_n \colon \Bun_{A, n} \to \Bun_G$ be the restriction of the above map. Then, $\Bun_{A, n}$ maps to the point corresponding to the $\sigma$-conjugacy class $b_n = {\scriptsize \left[\begin{pmatrix} 1 & 0 \\ 0 & \pi^{-n} \end{pmatrix} \right]} \in B(G)$. Now, we have $\cl{P}_X = \bigoplus_{n \in \bb{Z}} \pi_{n!} \Lambda$. We will compute each factor $\cl{P}_{X, n} = \pi_{n!} \Lambda$. 

%For this, we need the description of the cohomology of negative Banach-Colmez spaces. The same argument as in \cite[Proposition II.2.5 (i)]{FS24} and \cite[Lemma 2.6 (3)]{HI24} works again here. 
%\begin{lem} \label{lem:negBC}
%    Let $S$ be a small $v$-stack. Let $\pi \colon \cl{P} \to S$ be a torsor under $\BC(\cl{O}(-n)[1])$ for some $n \geq 1$ in the $v$-topology. Then, $\pi$ is cohomologically smooth, $\pi^! \bb{Z} \cong \bb{Z}(n)[2n]$ and $\pi_! \pi^! \bb{Z} \cong \bb{Z}$ in the $\cl{D}^\oc_\mot$-formalism. In particular, $\pi^* \colon \cl{D}_\mot^\oc(S) \to \cl{D}_\mot^\oc(\cl{P})$ is fully faithful. 
%\end{lem}
%\begin{proof}
%    We prove by the induction on $n$. First, the assertion that $\pi$ is cohomologically smooth and $\pi_! \pi^! \bb{Z} \cong \bb{Z}$ is local on target, so we may assume that $S = \Spa(R,R^+)$ is affine and $\cl{P} \cong \BC(\cl{O}(-n)[1]) \times S$. Take an untilt $S^\sharp = \Spa(R^\sharp, R^{\sharp +})$ in characteristic $0$. 
%    For $n = 1$, we have an exact sequence
%    \[
%        0 \to \underline{F}_S  \to (\bb{A}^1_{S^\sharp})^\diamondsuit \to \BC(\cl{O}(-1)[1])\vert_S  \to 0. 
%    \]
%    \yuta{Transfer the argument in \cite[Proposition 24.2]{Sch17}.}
%\end{proof}

Let $T \subset \GL_2$ be the diagonal torus. Let $B \subset \GL_2$ be the standard Borel subgroup and let $\ov{B} \subset \GL_2$ be its opposite. By \cite[Theorem 3.1 (ii)]{Sch25}, we have
\[
    \cl{D}^\oc(\Bun_G^{b_0}, \Lambda) \cong \cl{D}(\GL_2(F), \Lambda), \quad
    \cl{D}^\oc(\Bun_G^{b_n}, \Lambda) \cong \cl{D}(T(F), \Lambda) ( n \neq 0). 
\]
Thus, $\cl{P}_{X, 0}$ is a complex of smooth $\GL_2(F)$-representations and $\cl{P}_{X, n}$ is a complex of smooth $T(F)$-representations for $n \neq 0$.
%As $\cl{D}^\oc(\Bun_G^{b_0}, \Lambda) \cong \cl{D}(\GL_2(F), \Lambda)$ and $\cl{D}^\oc(\Bun_G^{b_n}, \Lambda) \cong \cl{D}(T(F), \Lambda)$ for $n \neq 0$, $\cl{P}_{X, 0}$ is a complex of smooth $\GL_2(F)$-representations and $\cl{P}_{X, n}$ is a complex of smooth $T(F)$-representations for $n \neq 0$. 

\begin{prop} \label{prop:Heckperiod}
    For each integer $n$, we have
    \[
        \cl{P}_{X, n} = i^{b_n}_! \cInd^{T(F)}_{A(F)} \Lambda^{-1}_\norm [2n] (n < 0), 
        \cl{P}_{X, n} = i^{b_n}_! \cInd^{T(F)}_{A(F)} \Lambda_\norm [-2n] (n > 0), 
    \]
    \[
        \cl{P}_{X, 0} = i^1_! C_c^\infty(\GL_2(F)/A(F), \Lambda). 
    \]
\end{prop}
\begin{proof}
    For $n = 0$, $\pi_0$ is given by $[\ast/A(F)] \to [\ast/\GL_2(F)] \hookrightarrow \Bun_G$, so the claim follows from \Cref{lem:locprof!}. For $n < 0$, we have a diagram
    \begin{center}
        \begin{tikzcd}
            \Bun_{A,n} \ar[r, "f"] & \Bun_{T}^{b_n} \ar[r, "\mfr{i}", yshift = 0.5ex] & \Bun_B^{b_n} \ar[l, "\mfr{q}", yshift = -0.5ex] = \Bun_G^{b_n} \ar[r, "\mfr{p}", hook] & \Bun_G. 
        \end{tikzcd}
    \end{center}
    Here, $\pi_n = \mfr{p} \circ \mfr{i} \circ f$ and $\mfr{i}^*$ induces an equivalence $\cl{D}^\oc(\Bun_G^{b_n}, \Lambda) \cong \cl{D}(T(F), \Lambda)$. It is well-known in usual sheaf theories, and the same proof works in the $\cl{D}^\oc$-formalism once we have \Cref{lem:computBC}: $\mfr{i}$ is a torsor under a twisted form of $\BC(\cl{O}(-n))$, so $\mfr{i}^*$ is fully faithful by \Cref{lem:computBC} and essentially surjective by $\mfr{i}^*\circ \mfr{q}^* \cong \id$. To compute $\cl{P}_{X,n}$, it is enough to show $\mfr{i}^*\mfr{i}_!f_! \Lambda \cong \cInd^{T(F)}_{A(F)} \Lambda_\norm^{-1} [2n]$. 

    Since $\mfr{i}^*$ is an equivalence, we may write $f_!\Lambda \cong \mfr{i}^* K$ for some $K \in \cl{D}(\Bun_G^{b_n}, \Lambda)$. By the projection formula, we have
    \[
        \mfr{i}^* \mfr{i}_! f_!\Lambda \cong \mfr{i}^* (K \otimes \mfr{i}_! \Lambda) \cong f_! \Lambda \otimes \mfr{i}^* \mfr{i}_! \Lambda. 
    \]
    Thus, it is enough to show $f^*\mfr{i}^* \mfr{i}_! \Lambda \cong \Lambda_\norm^{-1}[2n]$. Let $\pi \colon \ast \to \Bun_T^{b_n}$. Since $\Bun_T^{b_n}\times_{\Bun_B^{b_n}, \mfr{i} \circ \pi} \ast \cong \BC(\cl{O}(-n))$, $\pi^* \mfr{i}^* \mfr{i}_! \Lambda \cong \Lambda[2n]$ by \Cref{lem:computBC}. Then, $f^*\mfr{i}^* \mfr{i}_! \Lambda$ is a shift of a smooth character of $A(F)$. To identify this character, we may reduce to the case $\Lambda = \bb{Z}_\ell$, and then to $\Lambda = \bb{Z}/\ell^m$ for $m \geq 0$. Then, the claim follows from \cite[Lemma 2.6 (1)]{HI24} since $A(F)$ acts on $\Bun_T^{b_n}\times_{\Bun_B^{b_n}, \mfr{i} \circ \pi} \ast \cong \BC(\cl{O}(-n))$ by the inverse of the usual multiplication by $F^\times$. % action of $F^\times$. 

    Thus, we get the claim for $n < 0$. For $n > 0$, the same argument works by using
    \begin{center}
        \begin{tikzcd}
            \Bun_{A,n} \ar[r, "f"] & \Bun_{T}^{b_n} \ar[r, "\mfr{\ov{i}}", yshift = 0.5ex] & \Bun_{\ov{B}}^{b_n} \ar[l, "\mfr{\ov{q}}", yshift = -0.5ex] = \Bun_G^{b_n} \ar[r, "\mfr{\ov{p}}", hook] & \Bun_G. 
        \end{tikzcd}
    \end{center} 
\end{proof}

To compare with the spectral side, we provide a description of $\cl{P}_X$ using the (normalized) geometric Eisenstein series. %Let $T \subset B \subset \GL_2$ be the standard Borel pair and let $\ov{B} \subset \GL_2$ be the opposite of $B$. 
Consider the diagram
\begin{center}
    \begin{tikzcd}
        & \Bun_B \ar[ld, "\mfr{p}"'] \ar[rd, "\mfr{q}"] & \\
        \Bun_G & & \Bun_T \\
        & \Bun_\ov{B}. \ar[lu, "\mfr{\ov{p}}"] \ar[ru, "\mfr{\ov{q}}"'] &
    \end{tikzcd}
\end{center}
Let $K_B \in \cl{D}^\oc(\Bun_T, \Lambda)$ (resp.\ $K_{\ov{B}} \in \cl{D}^\oc(\Bun_T, \Lambda)$) be the complex such that $\mfr{q}^*K_B$ (resp.\ $\mfr{\ov{q}}^*K_{\ov{B}}$) is the dualizing complex of $\Bun_B$ (resp.\ $\Bun_{\ov{B}}$). See \cite[Theorem 1.6]{HI24} for the concrete description in torsion coefficients. As previously, the same computation holds true in the $\cl{D}^\oc$-formalism: we can show that $K_B$ is a shift of a character and the computation of such a character can be reduced to $\Lambda = \bb{Z}_\ell$, and then to $\Lambda = \bb{Z}/\ell^m$ for $m \geq 0$.  

Now, let $\Lambda = \Qla$ and fix a square root $\sqrt{q} \in \Qla$. This choice determines the half-norm character $\Qla_\norm^{1/2}$ of $F^\times$ and canonical square roots $K_B^{1/2}$ and $K_{\ov{B}}^{1/2}$. 

\begin{defi}(\cite[Definition 2.1.7]{HHS24})
    The normalized geometric Eisenstein series functors associated to $B$ and $\ov{B}$ are defined as 
    \[
        \Eis_{B!} = \mfr{p}_! \mfr{q}^*((-)\otimes K_B^{1/2}) \colon \cl{D}^\oc(\Bun_T, \Qla) \to \cl{D}^\oc(\Bun_G, \Qla),
    \]
    \[
        \Eis_{\ov{B}!} = \mfr{\ov{p}}_! \mfr{\ov{q}}^*((-)\otimes K_{\ov{B}}^{1/2}) \colon \cl{D}^\oc(\Bun_T, \Qla) \to \cl{D}^\oc(\Bun_G, \Qla). 
    \]
\end{defi}

Let $U$ be the unipotent radical of $B$. Fix a nontrivial additive character $\psi \colon F \to \Qla$ and let $\cl{W}_\psi = i^1_! \cInd_{U(F)}^{G(F)} \psi$ be the associated Whittaker sheaf. 

\begin{prop} \label{prop:HeckepEis}
    For each nonzero integer $n$, we have
    \[
        \cl{P}_{X, n} = \Eis_{B!}(i^{b_n}_! \cInd_{A(F)}^{T(F)} \Qla_\norm^{-1/2}[n]) (n<0), 
        \cl{P}_{X, n} = \Eis_{\ov{B}!}(i^{b_n}_! \cInd_{A(F)}^{T(F)} \Qla_\norm^{1/2}[-n]) (n>0), 
    \]
    and for $n=0$, we have a short exact sequence
    \[
        \cl{W}_\psi \hookrightarrow \cl{P}_{X, 0} \to \Eis_{B!}(i^1_! \cInd_{A(F)}^{T(F)} \Qla_\norm^{-1/2}). 
    \]
\end{prop}
\begin{proof}
    For $n < 0$, $\nu_{b_n}$ is strictly codominant with respect to $B$ (cf.\ \cite[Lemma 2.9 (1)]{Tak24}), so $\Eis_{B!}\vert_{\Bun_T^{b_n}} = i^{b_n}_!(- \otimes (\Qla_{\norm} \boxtimes \Qla_{\norm}^{-1})^{1/2}[n])$ by the computation of $K_B$ in \cite[Theorem 1.6]{HI24}. Then, the claim follows from \Cref{prop:Heckperiod}. The same argument works for $n > 0$. 
    
    We will prove the claim for $n = 0$. In this case, $\Eis_{B!}\vert_{\Bun_T^1}$ is the normalized parabolic induction $\cInd_{B(F)}^{G(F)}(- \otimes (\Qla_{\norm} \boxtimes \Qla_{\norm}^{-1})^{1/2})$. Thus, we need to show that there is a short exact sequence
    \[
        \cInd_{U(F)}^{G(F)} \psi \hookrightarrow \cInd_{A(F)}^{G(F)} \Qla \twoheadrightarrow \cInd^{G(F)}_{B(F)}(\cInd_{A(F)}^{T(F)} \Qla_\norm^{-1}). 
    \]
    This can be proved via the unfolding (see \cite[Section 7.3.1]{FW25}). Consider the Fourier transform via the following diagram. 
    \begin{center}
        {\footnotesize 
        \begin{tikzcd}
            & & \lbrack \ast/\Mir_2(F) \rbrack \times_{\lbrack\ast/A(F)\rbrack} \lbrack F/A(F) \rbrack \ar[r, "\ev"] \ar[ld, "\pr_1"] \ar[rd, "\pr_2"] & \lbrack \ast/ \Qla \rbrack \\
            & \lbrack \ast/\Mir_2(F) \rbrack \ar[ld, "f"] \ar[rd, yshift = 0.5ex] & & \lbrack F/A(F) \rbrack \ar[ld]  \\
            \lbrack \ast/G(F) \rbrack & & \lbrack \ast/A(F) \rbrack \ar[lu, "z", yshift = - 0.5ex] & 
        \end{tikzcd}
        }
    \end{center}
    Here, $\Mir_2 = A \ltimes U \subset B$ and $F$ is identified with the linear dual of $U(F)$ in defining $\ev$. The usual Fourier transform on $C_c^\infty(F, \Qla)$ with respect to $\psi$ can be reinterpreted as 
    \[
        z_! \Qla_\triv \cong \pr_{1!} \ev^*\psi \otimes \Qla_\norm^{-1}. 
    \]
    Here, $z_! \Qla_\triv \cong C_c^\infty(U(F), \Qla)$ and $\pr_{1!} \ev^*\psi \cong C_c^\infty(F, \Qla)$. The action of $(a, u) \in A(F) \ltimes U(F)$ on each space is given by 
    \[
        ((a, u)\cdot f)(x) = f(ax+u), \quad ((a, u)\cdot \widehat{f})(y) = \psi(\langle u, a^{-1}y \rangle) \widehat{f}(a^{-1}y)
    \]
    for $f \in C_c^\infty(U(F), \Qla)$ and $\widehat{f} \in C_c^\infty(F, \Qla)$. Thus, the Fourier transform
    \[
        C_c^\infty(U(F), \Qla) \ni f(x) \mapsto \widehat{f}(y) = \int_{F} f(x) \psi(-\langle x, y \rangle) dx \in C_c^\infty(F, \Qla)
    \]
    provides $z_! \Qla_\triv \cong \pr_{1!} \ev^*\psi \otimes \Qla_\norm^{-1}$. Now, the desired short exact sequence is obtained from the excision along the closed-open decomposition
    \[
        [\ast/\Mir_2(F)] \hookrightarrow \lbrack \ast/\Mir_2(F) \rbrack \times_{\lbrack\ast/A(F)\rbrack} \lbrack F/A(F) \rbrack \hookleftarrow [\ast/U(F)]
    \]
    obtained from $F = \{0\} \cup F^\times$. 
\end{proof}

\section{Formulation of the $\cl{B}$-side} \label{sec:Bside}

In this section, we introduce the relative spectral stack $\Par_{\LG}^{\widehat{X}}$ and the unnormalized $L$-sheaf $\cl{L}_{\widehat{X}}$. Throughout this section, we work with rational coefficients to avoid essential technical difficulty in general coefficients. %In \Cref{ssec:IwTateL} and \Cref{ssec:HeckeL}, we compute the unnormalized $L$-sheaves for the Iwasawa-Tate and Hecke periods.  

Let $\Lambda$ be a classical $G$-ring over $\bb{Q}$. We refer to \Cref{app:DAG} for our notation and convention in derived algebraic geometry. 

\subsection{Recollections of $\Par_{\LG}$}

In this section, we review the stack of $L$-parameters introduced in \cite{DHKM20}, \cite{Zhu21} and \cite{FS24}. As we work in characteristic $0$, we explain the definition as the stack of Weil-Deligne $L$-parameters as in \cite[Section 3.1]{Zhu21}. 

Let $\widehat{G}$ be the dual group of $G$ over $\Lambda$. Take a finite quotient $Q$ of $W_F$ so that the action of $W_F$ on $\widehat{G}$ factors through $Q$. Let $\LG = \widehat{G} \rtimes Q$ be the $L$-group of $G$. When $G$ is split, we take $Q$ to be the trivial group so that $\LG \cong \widehat{G}$. 

Let $\widehat{\mfr{g}} = \Lie(\widehat{G})$ and let $\lvert \cdot \rvert \colon W_F \to \bb{R}_{>0}$ be the homomorphism such that $\lvert \sigma \rvert = q^{-\ord(\sigma)}$. In particular, $\lvert \Fr \rvert = q^{-1}$ for a geometric Frobenius element $\Fr \in W_F$. 

\begin{defi} \label{defi:WDpara}
    For a $\Lambda$-algebra $R$, a Weil-Deligne $L$-parameter over $R$ with values in $\LG$ is a pair $(\varphi^\sm, N)$ where $\varphi^\sm \colon W_F \to \LG(R)$ is a smooth homomorphism over $Q$ and $N \in \widehat{\mfr{g}} \otimes R$ is an element such that $\Ad(\varphi^\sm(\sigma))N = \lvert \sigma \rvert N$ for all $\sigma \in W_F$. 
\end{defi}

Here, $N$ is automatically nilpotent since $\Lambda$ is in characteristic $0$. When $\LG = \GL_n$ for some $n$, a Weil-Deligne $L$-parameter is also called a Weil-Deligne representation. 

%We will be using the algebraic version, \textit{Weil-Deligne group scheme}, as presented in \cite{Sch24}, and whose equivalence with \Cref{defi:WDpara} is explained in \cite{Imai_2024}. 

\begin{prop}\textup{(\cite[Proposition 3.1.6]{Zhu21})} \label{prop:ParLGproperty}
    Let $Z^1(\WD_F, \LG)$ be the functor sending a $\Lambda$-algebra $R$ to the set of Weil-Deligne $L$-parameters over $R$ with values in $\LG$. Then, $Z^1(\WD_F, \LG)$ is represented by a disjoint union of affine schemes, which is equidimensional of dimension $\dim \widehat{G}$ and is a local complete intersection over $\Lambda$. Moreover, the dualizing complex on $[Z^1(\WD_F, \LG) / \widehat{G}]$ is trivial. 
\end{prop}

The quotient stack $[Z^1(\WD_F, \LG) / \widehat{G}]$ is called the stack of $L$-parameters. This is independent of the choice of $Q$. When $\Lambda$ is a $\Ql$-algebra, $Z^1(\WD_F, \LG)$ is isomorphic to the moduli space $Z^1(W_F, \LG)$ of \textit{continuous} $L$-parameters (see \cite[Lemma 3.1.8]{Zhu21}). 

\begin{exa}(\cite[Section IX.6.4]{FS24}) \label{exa:ParGm}
    Suppose that $\Lambda = \Qla$ and $\LG = \bb{G}_m$. By the local class field theory, $Z^1(\WD_F, \bb{G}_m)$ sends a $\Lambda$-algebra $R$ to the set of smooth homomorphisms $F^\times \to R^\times$. Since $F^\times \cong O_F^\times \times \bb{Z}$ and $\Qla[O_F^\times/K]$ decomposes into a direct sum of characters for each open subgroup $K \subset O_F^\times$, we have
    \[
        Z^1(\WD_F, \bb{G}_m) = \bigsqcup_{\chi \colon O_F^\times \to \Qlax} \bb{G}_m. 
    \]
    Here, $\chi$ runs through all smooth characters $O_F^\times \to \Qlax$. The component $\bb{G}_m \hookrightarrow Z^1(\WD_F, \bb{G}_m)$ for each $\chi$ sends $T\in R^\times$ to a character $\chi_T\colon F^\times \to R^\times$ such that $\chi_T(\pi) = T$ and $\chi_T\vert_{O_F^\times} = \chi$. As a result, we have
    \[
        \Par_{\bb{G}_m} =  \bigsqcup_{\chi \colon O_F^\times \to \Qlax} \bb{G}_m \times [\ast/\bb{G}_m]. 
    \]
    The component for each $\chi$ is denoted by $\Par_{\bb{G}_m, \chi} \subset \Par_{\bb{G}_m}$. 
\end{exa}

\subsubsection{Derived enhancement} \label{sssec:ParLGmap}

In the literature, the stack of $L$-parameters is defined as a moduli over classical $\Lambda$-algebras. As observed in \cite[Remark 3.1.4]{Zhu21} and \cite[Proposition VIII.2.1]{FS24}, it is known that the definition can be extended to animated $\Lambda$-algebras. This extension is important for our purpose because relative spectral stacks to be introduced will have nontrivial derived structure. In this section, we will provide a definition of the stack of $L$-parameters as a derived stack and compare it with the classical one. 

\begin{defi}
    Let
    \[
        \ID_F = \bb{G}_a \times \und{I_F}, \quad \WD_F = \bb{G}_a \rtimes \und{W_F}
    \]
    be the inertia group scheme and the Weil-Deligne group scheme over $\bb{Q}$, respectively. Here, $\und{I_F}$ (resp.\ $\und{W_F}$) is the (resp.\ locally) profinite group scheme associated to $I_F$ (resp.\ $W_F$), and the conjugation action of $\sigma \in W_F$ on $\bb{G}_a$ is the multiplication by $\lvert \sigma \rvert$.  
\end{defi}

A choice of a geometric Frobenius $\Fr \in W_F$ provides $\WD_F \cong \ID_F \rtimes \bb{Z}$. For each animated $\bb{Q}$-algebra $R$, let $\WD_{F, R} = \WD_{F} \times \Spec(R)$ and $\ID_{F, R} = \ID_{F} \times \Spec(R)$. We will follow the notation in \Cref{app:WeilDeligne_cohomology}. 

\begin{defi}\label{defi:parst} %\textup{(\cite[Section 4.1]{Sch25})}
    %Let $[\ast/\WD_{F}]$ be the classifying stack of $\WD_F$ in the fpqc topology. 
    Let $\Par_{\LG} = \und{\Map}_{[\ast/Q]}([\ast/\WD_F], [\ast/\LG])$ be the mapping stack over $\CAlg_\Lambda$. Here, $\ast = \Spec(\Lambda)$. For each animated $\Lambda$-algebra $R$, $\Par_{\LG}(R)$ is the mapping anima $\Map_{[\ast/Q]}([\Spec(R)/\WD_{F}], [\ast/\LG])$.
\end{defi}

\begin{rmk}
    The \'{e}tale classifying stack $[\ast/\WD_{F}]$ is not equal to the fpqc classifying stack. This difference does not matter here since the \'{e}tale classifying stack $[\ast/\LG]$ satisfies the fpqc descent due to the smoothness of $\LG$.  
\end{rmk}

This definition is a derived enhancement of \cite[Definition 4.3]{Sch25}. We will verify that our definition coincides with the classical Artin stack $[Z^1(\WD_F, \LG)/\widehat{G}]$. %For this, one may assume that $\Lambda = \bb{Q}$ as the general case is obtained by the base change. 

First, we compute the cotangent complex of $\Par_{\LG}$. Let $S$ be a derived scheme over $\Lambda$ and let $x \in \Par_{\LG}(S)$. As in \cite[Section 10.3.2]{AG15}, the cotangent space at $x$ viewed as a functor
\[
    \QCoh(S)^{\leq 0} \to \Ani
\]
is given by the truncation of the group cohomology
\[
    \cl{M} \mapsto \tau^{\leq 0} R\Gamma( \WD_{F} , \cl{M} \otimes (\cl{P}_x\times^{\LG} \widehat{\mfr{g}})) [1]. 
\]
Here, $R\Gamma(\WD_{F}, -)$ denotes the pushforward along $[S/\WD_{F}] \to S$, $\LG$ acts on $\widehat{\mfr{g}}$ via the adjoint action, and $\cl{P}_x$ is the $\LG$-torsor over $[S/\WD_{F}]$ obtained from $x \colon [S/\WD_{F}] \to [\ast/\LG]$. 

Let us explain the computation of $R\Gamma(\WD_F, -)$. For each $\cl{M} \in \QCoh([S/\WD_{F}])$, let $\cl{M}_S \in \QCoh(S)$ be the underlying quasi-coherent sheaf on $S$. Then, $\cl{M}_S$ carries a smooth $W_F$-action, and a $\bb{G}_{a}$-action, which is given by a locally nilpotent endomorphism (see \Cref{lem:Ga_cohomology})
\[
    N \in \End(\cl{M}_S).
\]

\begin{prop} \textup{(\Cref{prop:WD_cohomology})}
    For each $\cl{M} \in \QCoh([S/\WD_{F}])$, we have 
    \[
        R\Gamma(\WD_{F}, \cl{M}) \cong \left[\cl{M}_S^{I_F} \xrightarrow{(\Fr - 1, N)} \cl{M}_S^{I_F} \oplus \cl{M}_S^{I_F} \xrightarrow{(N,  - q \Fr + 1)} \cl{M}_S^{I_F}\right]. 
    \]
    In particular, $R\Gamma(\WD_{F}, - )$ has cohomological dimension $2$, commutes with small colimits, and satisfies the projection formula and the base change along $S$. 
\end{prop}

\begin{defi}
    Let $\cl{O}(1)$ be a line bundle on $[\ast/\WD_F]$ corresponding to the norm character $\lvert \cdot \rvert$ and let $\cl{O}(n) = \cl{O}(1)^{\otimes n}$. For each $\cl{M} \in \QCoh([S/\WD_{F}])$ and $n \in \bb{Z}$, let $\cl{M}(n) = \cl{M} \otimes \cl{O}(n)$ be the $n$-th Tate twist.
\end{defi}

\begin{cor} \label{cor:Tatedual}
    For every perfect complex $\cl{M}$ on $[S / \WD_{F}]$, we have an isomorphism 
    \[
        R\Gamma(\WD_{F}, \cl{M})^\vee \cong R\Gamma(\WD_{F}, \cl{M}^\vee(1))[2]. 
    \]  
    %Here $(-)^\vee$ is the internal hom dual $(-)^\vee:= \underline{\Map}_{\QCoh([S/\WD_{F,S}])}(-, \cl{O})$ swhose underlying complex on $S$ is given by $\underline{\Map}_{\QCoh(S)}(-,\cl{O}_S)$. 
\end{cor}
\begin{proof}
    This follows from the explicit description of \Cref{prop:WD_cohomology}. %We note that both $R\Gamma(-)^\vee$ and $(-)^\vee(1)$ are exact and commmutes with shits so it suffices to prove the equivalence when $\cl{M}$ is a singe quasicoherent sheaf in degree $0$. 
\end{proof}

\begin{prop} \label{prop:ParLGcotangent}
    For a derived $\Lambda$-scheme $S$ and $x \in \Par_{\LG}(S)$, let $\widehat{\mfr{g}}_x = \cl{P}_x \times^{\LG} \widehat{\mfr{g}}$ be a vector bundle on $[S / \WD_{F}]$. For every $(S, x)$, the cotangent space of $\Par_{\LG}$ at $(S, x)$ is corepresentable by the perfect complex
    \[
        R\Gamma( \WD_{F}, \widehat{\mfr{g}}_x^\vee(1))[1]
    \]
    with tor-amplitude in $[-1, 1]$. In particular, $\Par_{\LG}$ admits a cotangent complex. % in the sense of \textup{\cite[Definition 1.3.13]{DAGXIV}}. 
\end{prop}
\begin{proof}
    The cotangent space at $(S, x)$ sends $\cl{M} \in \QCoh(S)^{\leq 0}$ to the truncation of
    \begin{align*}
         R\Gamma( \WD_{F} , \cl{M} \otimes \widehat{\mfr{g}}_x) [1] & \cong \cl{M} \otimes R\Gamma( \WD_{F}, \widehat{\mfr{g}}_x)[1] \\
        & \cong R\Hom(R\Gamma( \WD_{F}, \widehat{\mfr{g}}_x)^\vee[-1], \cl{M}). 
    \end{align*}
    Thus, the claim follows from \Cref{cor:Tatedual}. The second claim follows from \cite[Remark 1.3.14]{DAGXIV} since the Weil-Deligne cohomology satisfies the base change along $S$ by \Cref{prop:WD_cohomology}. 
\end{proof}

Next, we verify that the classical truncation of $\Par_{\LG}$ coincides with $[Z^1(\WD_F, \Lambda) / \LG]$. In the following, a group derived scheme refers to a group object in the sense of \cite[Definition 7.2.2.1]{HTT}. The category of group derived schemes over $Q$ is denoted by $\grp/Q$.

%Recall our foundations in \Cref{app:DAG}, where we let $\DStk_\Lambda$ the the topos of derived stacks over $\Lambda$.  Note that for any topos $\cl{C}$ we have an adjunction between the $\bb{E}_1$ objects and pointed objects, which induces an equivalence on the group like objects:
%\[ 
%\begin{tikzcd}
%    \bb{E}_1(\cl{C})\rar[shift left =1ex, "B"] &\cl{C}_* \ar[l, shift left =1ex, "\Omega"]  \\ 
%    \grp(\cl{C}) \ar[u, hook] \rar["\simeq"] & \cl{C}_{*}^{\ge 1 } \ar[u, hook] 
%\end{tikzcd}
%\]
%We will denote $\grp/Q:=\grp(\DStk_{/Q})$. 

\begin{prop} \label{prop:ParLGcl}
    We have $\Par_{\LG} \cong \und{\Map}_{\grp/Q}(\WD_F, \LG) / \widehat{G}$. Moreover, we have
    \[
        \clas \und{\Map}_{\grp/Q}(\WD_F, \LG)\cong Z^1(\WD_F, \Lambda) ,\quad 
        \clas \Par_{\LG} \cong [Z^1(\WD_F, \Lambda) / \widehat{G}]. 
    \]
\end{prop}
\begin{proof}
    Let $S$ be a derived scheme over $\Lambda$ and let $x \in \Par_{\LG}(S)$. From the diagram
    \begin{center}
        \begin{tikzcd}
            S \ar[r, dotted] \ar[d] & \lbrack \ast/\widehat{G} \rbrack \ar[d] \ar[r] & \ast \ar[d] \\
            \lbrack S / \WD_{F, S} \rbrack \ar[r, "x"] & \lbrack \ast/\LG \rbrack \ar[r] & \lbrack \ast/Q \rbrack, 
        \end{tikzcd}
    \end{center}
    we get a canonical $\widehat{G}$-torsor $\cl{\widehat{P}}_x$ over $S$. Let $\widehat{\cl{P}}^\univ \to \Par_{\LG}$ be the universal $\widehat{G}$-torsor. %Since a trivialization of $\cl{\widehat{P}}_x$ corresponds to a commutative diagram
    %\begin{center}
    %    \begin{tikzcd}
    %        S \ar[rr] \ar[d] & & \ast \ar[d] \\
    %        \lbrack S / \WD_{F, S} \rbrack \ar[rr, dotted] \ar[rd] & & \lbrack \ast/\LG \rbrack \ar[ld] \\
    %        & \lbrack \ast/Q \rbrack, & 
    %    \end{tikzcd}
    %\end{center}
    Then, $\widehat{\cl{P}}^\univ(S)$ is isomorphic to the pointed mapping anima $\Map_{S/(-)/[\ast/Q]}([S/\WD_{F, S}], [\ast/\LG])$. By taking the \v{C}ech nerve, we have 
    \[
        \Map_{S/(-)/[\ast/Q]}([S/\WD_{F}], [\ast/\LG]) = \Map_{Q^\bullet}(\WD_{F}^\bullet, (\LG)^\bullet) =\Map_{\grp/Q}(\WD_{F, S}, \LG).  
    \]
    Here, the second term refers to the mapping anima of simplicial sets. Thus, we get the first claim. When $S$ is classical, $\Map_{\grp/Q}(\WD_{F, S}, \LG)$ is static since $\WD_{F, S}$ and $\LG$ are both classical. Since a homomorphism $\bb{G}_{a, S} \to \LG$ uniquely corresponds to an element $N \in \widehat{\mfr{g}} \otimes \cl{O}_S$ that is nilpotent Zariski locally on $S$, it is easy to see from \Cref{defi:WDpara} that $\und{\Map}_{\grp/Q}(\WD_F, \LG)(S) = Z^1(\WD_F, \LG)(S)$. Thus, we get the second claim. 
\end{proof}
%\begin{rmk}
    
%Consider the map associated with basepoint $*\ra BH$ where we have the map 
%\[ 
%\begin{tikzcd}
%    \Map_{\grp{\cl{C}}}(G,H)\simeq  \underline{\Map}_{\cl{C}_*}(BG,BH) \rar \ar[d] &\Map_{\cl{C}} (BG, BH)\ar[d] \\
%    * \rar &\Map_{\cl{C}}(*,BH)
%\end{tikzcd}, \quad 
%\begin{tikzcd}
%    \Map_{\grp(\cl{C})}(G,H) \ar[d]  \rar &\Map_{\grp}(G,H)/H\ar[d] \\ 
%    *\rar & BH 
%\end{tikzcd}
%\]
% But by deelooping adjunction the fibers coincide so that under the canonical comparison $\Map_{\grp{\cl{C}}}(G,H)/H \ra \Map_{\cl{C}}(BG,BH)$ we have an equivalence. 
%\end{rmk}
\begin{prop} \label{prop:ParLGArtin}
    The derived stack $\Par_{\LG}$ is a derived $1$-stack locally of finite presentation. 
\end{prop}

See \Cref{app:DAG} for derived $1$-stacks and the Artin-Lurie representability. 

\begin{proof}
    We will apply the Artin-Lurie representability criterion \cite[Theorem 7.1.6]{DAG} to $\Par_{\LG}$. Note that it will only show that $\Par_{\LG}$ is almost of finite presentation. It implies the claim since the cotangent complex of $\Par_{\LG}$ is perfect (cf.\ \cite[Proposition 3.2.18]{DAG}). 
    
    First, the conditions on $\clas \Par_{\LG}$, (3) and (7) in loc. cit., follow from \Cref{prop:ParLGcl}. Moreover, $\Par_{\LG}$ satisfies the \'{e}tale descent, and admits a cotangent complex by \Cref{prop:ParLGcotangent}. It remains to see that $\Par_{\LG}$ commutes with filtered colimits on truncated objects, and is infinitesimally cohesive and convergent. 

    Since $[\ast/\WD_{F}] = \colim_{n \in \Delta^\op} \WD_{F}^n$, we have 
    \[
        \Par_{\LG}(S) = \lim_{n \in \Delta} \Map_{[\ast/Q]}(\WD_{F,S}^n, [\ast/\LG]). 
    \]
    Here, each term is isomorphic to $\Map(\WD_{F,S}^n, [\ast/\widehat{G}])$, so it commutes with filtered colimits on truncated objects, and is infinitesimally cohesive and convergent. Since the latter two conditions are preserved under limits, they hold for $\Par_{\LG}$. Moreover, if $S$ is $n$-truncated, $\Map(\WD_{F,S}^n, [\ast/\widehat{G}])$ is $(n+1)$-truncated by \cite[Corollary 5.3.8]{DAG}. Since the totalization of truncated spaces commutes with filtered colimits, $\Par_{\LG}$ commutes with filtered colimits on truncated objects. 
\end{proof}

As in \cite[Section 1.3.6]{GR17I}, there is a natural map 
\[
    [Z^1(\WD_F, \LG) / \widehat{G}] = \clas \Par_{\LG} \to \Par_{\LG}. 
\]
It can be seen to be an isomorphism by comparing the cotangent complexes. 

\begin{prop}
    The natural map $[Z^1(\WD_F, \LG) / \widehat{G}] \to \Par_{\LG}$ is an isomorphism. 
\end{prop}
\begin{proof}
    By \Cref{prop:ParLGproperty}, $[Z^1(\WD_F, \LG) / \widehat{G}]$ is quasi-smooth. It can be seen from the explicit description in the proof of \cite[Lemma 3.3.1]{Zhu21} and \cite[Corollary VIII.2.3]{FS24} that its cotangent complex is equal to the pullback of that of $\Par_{\LG}$ (see \Cref{prop:ParLGcotangent}). In particular, the relative cotangent complex of $[Z^1(\WD_F, \LG) / \widehat{G}] \to \Par_{\LG}$ is trivial. 

    Since $\Par_{\LG}$ is a derived $1$-stack over $\Lambda$, we may take a smooth cover $x \colon S \to \Par_{\LG}$ from a disjoint union of affine $\Lambda$-schemes. Let
    \[
        S' = [Z^1(\WD_F, \LG) / \widehat{G}] \times_{\Par_{\LG}, x} S. 
    \]
    It is enough to show that $S' \to S$ is an isomorphism. First, $S'$ is an Artin $\Lambda$-stack locally of finite presentation over $\Lambda$, and the relative cotangent complex $L_{S'/S}$ vanishes. Since the cotangent complex of $S$ is connective, that of $S'$ is connective. Moreover, $\clas S' \cong \clas S$, so by \cite[Theorem B.2.14]{AG15}, $S'$ is a derived scheme. Then, $S' \to S$ is flat by \cite[Corollary 3.4.10]{DAG}, so we get $S' \cong S$ from $\clas S'\cong \clas S$. 
\end{proof}

\subsubsection{Ind-coherent sheaves on $\Par_{\LG}$} \label{ssec:indcoh}

Recall from \Cref{prop:ParLGproperty} that $\Par_{\LG}$ is a disjoint union of the $\widehat{G}$-quotient of an affine $\widehat{G}$-scheme. Take the decomposition
\[
    \Par_{\LG} = \coprod_{x \in \pi_0(\Par_{\LG})} \Par_{\LG, x}
\]
into connected components. By \cite[Proposition 3.21]{BZFN10}, $\Par_{\LG, x}$ is perfect, so
\[
    \QCoh(\Par_{\LG}) = \prod_{x \in \pi_0(\Par_{\LG})} \QCoh(\Par_{\LG, x}) = \Ind\bigl(\bigoplus_{x} \Perf(\Par_{\LG, x})\bigr). 
\]
Similarly, since $\Par_{\LG, x}$ is a QCA stack, we have 
\[
    \IndCoh(\Par_{\LG}) = \prod_{x \in \pi_0(\Par_{\LG})} \IndCoh(\Par_{\LG, x}) = \Ind\bigl(\bigoplus_{x} \Coh(\Par_{\LG, x})\bigr). 
\]
%When one use the theory devloped in \cite[Section 9]{Zhu25}, one needs to regard $\Par_{\LG}$ as an ind-stack. 

\subsection{Relative spectral stacks $\Par_{\LG}^{\widehat{X}}$} \label{ssec:relspec}

In this section, we introduce the relative spectral stack $\Par_{\LG}^{\widehat{X}}$ and the associated unnormalized $L$-sheaf $\cl{L}_{\widehat{X}}$ for a suitable derived $\LG$-scheme $\widehat{X}$ over $\Lambda$. Here, our convention on $\LG$-schemes is the same as \Cref{conv:Gvar}.

\begin{defi} \label{defi:relspec}
    Let $\widehat{X}$ be a derived $\LG$-scheme locally of finite presentation (see \cite[p.32]{DAG}) over $\Lambda$ such that $\clas\widehat{X}$ is $\widehat{G}$-quasi-projective over $\Lambda$. Let $\Par_{\LG}^{\widehat{X}}$ be the mapping stack $\und{\Map}_{[\ast/Q]}([\ast/\WD_F], [\widehat{X}/\LG])$ over $\CAlg_\Lambda$. 
\end{defi}

\begin{rmk}
    Here, a $\widehat{G}$-scheme over $\Lambda$ is $\widehat{G}$-quasi-projective if it admit a $\widehat{G}$-equivariant ample line bundle. By \cite[Theorem 2.5]{Sum75}, this condition is satisfied when $\Lambda$ is normal and the $\widehat{G}$-scheme is normal and quasi-projective over $\Lambda$. 
\end{rmk}

\begin{exa}
    Let $\widehat{H} \subset \widehat{G}$ be a flat closed subgroup scheme stable under the action of $Q$. By abuse of notation, let $\LH = \widehat{H} \rtimes Q$ and suppose $\widehat{X} = \widehat{G}/\widehat{H}$. Then, $[\widehat{X}/\LG] \cong [\ast/\LH]$ and $\Par_{\LG}^{\widehat{X}} \cong \Par_{\LH}$. 
\end{exa}

In contrast to $\Par_{\LG}$, it is important to keep track of the derived structure of $\Par_{\LG}^{\widehat{X}}$ (see e.g. \Cref{prop:explicitrel}). We will show that $\Par_{\LG}^{\widehat{X}}$ is an Artin $\Lambda$-stack and $\Par_{\LG}^{\widehat{X}} \to \Par_{\LG}$ is schematic. 

First, we study the cotangent complex of $\Par_{\LG}^{\widehat{X}}$. Let $L_{[\widehat{X} / \widehat{G}]}$ be the relative cotangent complex of $[\widehat{X} / \LG] \to [\ast/Q]$. Since $\widehat{X}$ is locally of finite presentation, $L_{[\widehat{X} / \widehat{G}]}$ is perfect with tor-amplitude $(-\infty, 1]$ by \cite[Proposition 3.2.14]{DAG}. 

\begin{prop} \label{prop:ParLGXcotangent}
    For a derived $\Lambda$-scheme $S$ and $x \in \Par_{\LG}^{\widehat{X}}(S)$, the cotangent space of $\Par_{\LG}^{\widehat{X}}$ over $\Lambda$ at $(S, x)$ is corepresentable by the perfect complex
    \[
        R\Gamma(\WD_{F}, x^* L_{[\widehat{X} / \widehat{G}]} (1))[2]
    \]
    with tor-amplitude in $(-\infty, 1]$. In particular, $\Par_{\LG}^{\widehat{X}}$ admits a cotangent complex. 
\end{prop}
\begin{proof}
    As in \cite[Section 10.3.2]{AG15}, the cotangent space at $x$ viewed as a functor
    \[
        \QCoh(S)^{\leq 0} \to \Ani
    \]
    is given by the truncation of the group cohomology
    \[
        \cl{M} \mapsto \tau^{\leq 0} R\Gamma( \WD_{F} , \cl{M} \otimes x^* L_{[\widehat{X} / \widehat{G}]}^\vee). 
    \]
    By \Cref{prop:WD_cohomology} and \Cref{cor:Tatedual}, it is isomorphic to the truncation of
    \[
        \cl{M} \otimes R\Gamma( \WD_{F} , x^* L_{[\widehat{X} / \widehat{G}]}^\vee) \cong R\Hom(R\Gamma( \WD_{F} , x^* L_{\widehat{X}}(1))[2], \cl{M}).  
    \]
\end{proof}

\begin{rmk} \label{rmk:ParLGXrelcotangent}
    Let $L_{\widehat{X}}$ be the relative cotangent complex of $[\widehat{X} / \LG] \to [\ast/\LG]$, which is perfect and connective. By the same proof as above, the relative cotangent space of $\Par_{\LG}^{\widehat{X}} \to \Par_{\LG}$ at $(S, x)$ is corepresentable by the connective perfect complex
    \[
        R\Gamma(\WD_{F}, x^* L_{\widehat{X}} (1))[2]. 
    \]
\end{rmk}

Next, we study the classical truncation of $\Par_{\LG}^{\widehat{X}}$. 

\begin{prop} \label{prop:ParLGXclassical}
    Let $R$ be a $\Lambda$-algebra and let $x \in \Par_{\LG}(R)$. Let $\widehat{\cl{P}}_x$ be the $\widehat{G}$-torsor over $R$ corresponding to $x$. 
    \begin{enumerate}
        \item Let $\clas \widehat{X}_x = \cl{\widehat{P}}_x \times^{\widehat{G}} \clas \widehat{X}$. Then, $\clas \widehat{X}_x$ is represented by a scheme. 
        \item Let $\widehat{X}_R^{\WD_F} = \Par_{\LG}^{\widehat{X}} \times_{\Par_{\LG}, x} \Spec(R)$. Then, $\clas\widehat{X}_R^{\WD_F}$ is a closed subscheme of $\clas \widehat{X}_x$. 
    \end{enumerate}
\end{prop}
\begin{proof}
    Since $\clas \widehat{X}$ admits a $\widehat{G}$-equivariant ample line bundle, (1) follows from the descent proved in \cite[Theorem 7]{BLR90}. For (2), we may pass to a smooth cover $\widehat{\cl{P}}_x$ and assume that $\Spec(R) \to \Par_{\LG}$ factors through $Z^1(\WD_F, \LG)$. Let $\varphi_R\colon \WD_{F, R} \to \LG$ be the associated $L$-parameter over $R$. Then, $\WD_{F, R}$ acts on $\widehat{X}$ via $\varphi_R$. For each $R$-algebra $S$, 
    \begin{align*}
        \widehat{X}_R^{\WD_F}(S) & = \Sec([\widehat{X}_S / \WD_{F}] \to [S/\WD_{F}]) \\
        & \cong \lim_{[n] \in \Delta} \Sec(\widehat{X}_S \times_S \WD_{F, S}^n \to \WD_{F, S}^n).  
    \end{align*}
    Here, $\Sec(X \to Y)$ denotes the mapping anima $\Map_{Y}(Y, X)$. The second isomorphism follows from the descent along $\Spec(S) \to [\Spec(S)/\WD_{F}]$. Now, $\Sec(\widehat{X}_S \times_S \WD_{F, S}^n \to \WD_{F, S}^n)$ is static since $\WD_{F, S}^n$ is classical, so $\widehat{X}_R^{\WD_F}(S)$ is static. Then, we have
    \begin{equation} \label{eq:XWDF}
        \widehat{X}_R^{\WD_F}(S) \cong \Eq(\widehat{X}(S) \rightrightarrows \widehat{X}(\WD_{F, S})). 
    \end{equation} 
    We show that $\clas \widehat{X}_R^{\WD_F} \to \clas \widehat{X}$ is a closed immersion. Let $x \in \widehat{X}(S)$. Consider the diagram
    \begin{center}
        \begin{tikzcd}
            \WD_{F, S} \ar[d] \ar[r, "\varphi_R \times x"] & \LG \times \clas \widehat{X} \ar[d, "\rho"] \\
            \Spec(S) \ar[r, "x"] & \clas \widehat{X}. 
        \end{tikzcd}
    \end{center}
    Here, $\rho$ denotes the action of $\LG$ on $\clas \widehat{X}$. This diagram commutes if and only if $x \in \widehat{X}_R^{\WD_F}(S)$. Since $\clas \widehat{X}$ is separated, we can take the maximal closed subscheme $Z \subset \WD_{F, S}$ on which the above diagram commutes. 

    For each integer $n$, let $Z_n = Z \cap (\ID_{F, S} \rtimes \{\Fr^n\})$. There is a compact open subgroup $K \subset I_F$ such that $\varphi_R\vert_K$ is trivial. Let $K_n = \Ad(\Fr^n)(K)$. Then, we have a closed immersion $Z_n/K \subset (\bb{G}_{a, S} \times (I_F/K_n)) \rtimes \{\Fr^n\}$. Then, the category of $S$-algebras $A$ such that 
    \[
        (Z_n/K)_A = (\bb{G}_{a, A} \times (I_F/K_n)) \rtimes \{\Fr^n\}
    \]
    has an initial object, which is a quotient of $S$. By running over all $n$, it follows that $\clas \widehat{X}_R^{\WD_F} \to \clas \widehat{X}$ is a closed immersion. By taking the descent along $\widehat{\cl{P}}_x \to \Spec(R)$, we get a closed immersion $\clas \widehat{X}_R^{\WD_F} \hookrightarrow \clas \widehat{X}_x$. 
\end{proof}

\begin{cor} \label{cor:clParLGXArtin}
    The classical truncation $\clas \Par_{\LG}^{\widehat{X}}$ is a classical Artin stack locally of finite presentation over $\Lambda$. 
\end{cor}
\begin{proof}
    By \Cref{prop:ParLGXclassical}, $\clas \Par_{\LG}^{\widehat{X}} \times_{\Par_{\LG}} Z^1(\WD_F, \LG)$ is a scheme locally of finite presentation over $Z^1(\WD_F, \LG)$. Since $Z^1(\WD_F, \LG) \to \Par_{\LG}$ is a smooth cover, $\clas \Par_{\LG}^{\widehat{X}}$ is an Artin stack locally of finite presentation over $\Lambda$. 
\end{proof}

As in the proof of \Cref{prop:ParLGArtin}, the above properties are enough to prove the following. See \Cref{app:DAG} for derived $1$-stacks and the Artin-Lurie representability. 

\begin{prop} \label{prop:ParLGXArtin}
    The derived stack $\Par_{\LG}^{\widehat{X}}$ is a derived $1$-stack locally of finite presentation. 
\end{prop}
\begin{proof}
    We will apply the Artin-Lurie representability criterion \cite[Theorem 7.1.6]{DAG}. Note that it will only show that $\Par_{\LG}^{\widehat{X}}$ is almost of finite presentation. It implies the claim since the cotangent complex of $\Par^{\widehat{X}}_{\LG}$ is perfect. 
    
    The conditions on the classical truncation follow from \Cref{cor:clParLGXArtin}. By definition, $\Par_{\LG}^{\widehat{X}}$ satisfies the \'{e}tale descent, and it admits a cotangent complex by \Cref{prop:ParLGcotangent}. It remains to see that $\Par_{\LG}^{\widehat{X}}$ commutes with filtered colimits on truncated objects, and $\Par_{\LG}^{\widehat{X}}$ is infinitesimally cohesive and convergent. 

    Since $[\ast/\WD_{F}] = \colim_{n \in \Delta^\op} \WD_{F}^n$, we have 
    \[
        \Par_{\LG}^{\widehat{X}}(S) = \lim_{n \in \Delta} \Map_{[\ast/Q]}(\WD_{F,S}^n, [\widehat{X}/\LG]). 
    \]
    Here, each term is isomorphic to $\Map(\WD_{F,S}^n, [\widehat{X}/\widehat{G}])$, so it commutes with filtered colimits on truncated objects, and it is infinitesimally cohesive and convergent. Since the latter two conditions are preserved under limits, they hold for $\Par_{\LG}^{\widehat{X}}$. Moreover, if $S$ is $n$-truncated, $\Map(\WD_{F,S}^n, [\widehat{X}/\widehat{G}])$ is $(n+1)$-truncated by \cite[Corollary 5.3.8]{DAG}. Since the totalization of truncated spaces commutes with filtered colimits, $\Par_{\LG}^{\widehat{X}}$ commutes with filtered colimits on truncated objects. 
\end{proof}

\begin{prop} \label{prop:ParLGXschematic}
    The natural map $\pi_{\widehat{X}} \colon \Par_{\LG}^{\widehat{X}} \to \Par_{\LG}$ is schematic. In particular, $\Par_{\LG}^{\widehat{X}}$ has affine diagonal and is a $1$-Artin stack locally of finite presentation.  
\end{prop}
\begin{proof}
    Let $S$ be a derived $\Lambda$-scheme and let $x \in \Par_{\LG}(S)$. Let $S' = \Par_{\LG}^{\widehat{X}} \times_{\Par_{\LG}, x} S$. We will apply \cite[Theorem B.2.14]{AG15} (cf.\ \cite[Theorem 5.1.12]{DAG}) to show that $S'$ is a derived scheme. By \Cref{prop:ParLGXArtin}, $S'$ is a derived $1$-stack. Moreover, $\clas S'$ is a scheme by \Cref{prop:ParLGXclassical}. By the same proof as in \Cref{prop:ParLGXcotangent} (see \Cref{rmk:ParLGXrelcotangent}), the relative cotangent complex of $S' \to S$ is connective, so the cotangent complex of $S'$ is connective. Thus, we may apply the above criterion to $S'$ to see that $S'$ is a derived scheme. Thus, we get the first claim. 

    Then, $\widehat{X}^{\WD_F}_{Z^1} = \Par_{\LG}^{\widehat{X}} \times_{\Par_{\LG}} Z^1(\WD_F, \LG)$ is a derived scheme. Since $Z^1(\WD_F, \LG)$ is a disjoint union of affine schemes, $\widehat{X}^{\WD_F}_{Z^1}$ is separated by \Cref{prop:ParLGXclassical} (2). Since $\widehat{X}^{\WD_F}_{Z^1} \to \Par_{\LG}^{\widehat{X}}$ is a smooth affine cover, it follows that $\Par_{\LG}^{\widehat{X}}$ has an affine diagonal. Combining with \Cref{prop:ParLGXArtin}, we see that $\Par_{\LG}^{\widehat{X}}$ is $1$-Artin. 
\end{proof}

\begin{defi}
    Let $\omega_{\widehat{X}}$ be the dualizing complex of $\Par_{\LG}^{\widehat{X}}$. We define the unnormalized $L$-sheaf associated to $\widehat{X}$ as in \Cref{defi:relspec} by 
    \[
        \cl{L}_{\widehat{X}} = \pi_{\widehat{X} *} \omega_{\widehat{X}}. 
    \]
\end{defi}

Note that $\pi_{\widehat{X}}$ needs to be schematic to define $\pi_{\widehat{X} *}$ in the formalism of ind-coherent sheaves developed in \cite{GR17I} (see \cite[1.4 (i)]{GR17I}). 

Now, we consider the orbit decomposition of the underlying space $\lvert \Par_{\LG}^{\widehat{X}} \rvert$. 

\begin{lem} \label{lem:orbitdec}
    Let $Z \subset \widehat{X}$ be a $\LG$-stable closed subset and let $U = \widehat{X} - Z$. Then, we have the following. 
    \begin{enumerate}
        \item $\Par_{\LG}^U \to \Par_{\LG}^{\widehat{X}}$ is an open immersion. 
        \item $\lvert \Par_{\LG}^Z \rvert \to \lvert \Par_{\LG}^{\widehat{X}} \rvert$ is injective and the image is the complement of $\lvert \Par_{\LG}^U \rvert$. 
    \end{enumerate}
\end{lem}
\begin{proof}
    For each $R \in \CAlg_{\Lambda}$, $\Spec(R) \to [\Spec(R)/\WD_{F, R}]$ is homeomorphic, so (1) follows easily. For (2), it is enough to show that $\lvert \widehat{X}_R^{\WD_F} \rvert = \lvert Z_R^{\WD_F} \rvert \coprod \lvert U_R^{\WD_F} \rvert$ for every $\Spec(R) \to \Par_{\LG}$. First, $\lvert Z_R^{\WD_F} \rvert \to \lvert \widehat{X}_R^{\WD_F} \rvert$ is injective by \Cref{prop:ParLGXclassical} (2). Then, we can see from the description in \eqref{eq:XWDF} that every point of $\lvert  \widehat{X}_R^{\WD_F} \rvert$ outside $\lvert U_R^{\WD_F} \rvert$ lies in $\lvert Z_R^{\WD_F} \rvert$. 
\end{proof}

%\begin{cor} \label{cor:supprelX}
%    Suppose that $\Lambda= \Qla$, $Q$ is trivial and $\widehat{X}$ has finitely many $\widehat{G}$-orbits. Let $\widehat{X} = \coprod_{i\in I} \widehat{G}/\widehat{H}_i \cdot x_i$ be the orbit decomposition with $x_i \in \widehat{X}(\Qla)$ and $\widehat{H}_i = \Stab_{\widehat{G}}(x_i)$. Then, the image of $\pi_{\widehat{X}}$ is the union of the images of
%    \[
%        \Par_{\widehat{H}_i} \to \Par_{\widehat{G}}
%    \]
%    over all $i \in I$. 
%\end{cor}
%\begin{proof}
%    We prove by the induction on the number of $\widehat{G}$-orbits in $\widehat{X}$. Since $\widehat{X}$ has finitely many orbits, there is a closed $\widehat{G}$-orbit $Z \subset \widehat{X}$. Let $U = \widehat{X} - Z$. We may assume that $Z = \widehat{G}/\widehat{H}_i\cdot x_i$ for some $i \in I$. By \Cref{lem:orbitdec}, 
%    \[
%        \lvert \Par_{\widehat{G}}^{\widehat{X}} \rvert = \lvert \Par_{\widehat{G}}^{Z} \rvert \coprod \lvert \Par_{\widehat{G}}^{U} \rvert, 
%    \]
%    and $\Par_{\widehat{G}}^Z \cong \Par_{\widehat{H}_i}$ over $\Par_{\widehat{G}}$, so the claim follows from the induction hypothesis on $U$. 
%\end{proof}

\subsection{Description in the vectorial case} \label{ssec:vectcase}

In this section, we provide a description of $\Par_{\LG}^{\widehat{X}}$ as a derived vector bundle over $\Par_{\LG}$ when $\widehat{X}$ comes from an algebraic representation $V \in \Rep(\LG)$. 

\begin{prop} \label{prop:explicitrel}
    For each $V \in \Rep(\LG)$, $\Par_{\LG}^{V}$ is a derived vector bundle over $\Par_{\LG}$. For each $R \in \CAlg_\Lambda$ and $x \in \Par_\LG(R)$, $\Par_{\LG}^{V} \times_{\Par_\LG, x} \Spec(R)$ is a derived vector bundle representing $R\Gamma(\WD_{F}, x^* \cl{O}_V)$. 
\end{prop}

\begin{proof}
    It is enough to show the latter claim. Since $[V / \LG] \to [\ast / \LG]$ is a vector bundle associated to $\cl{O}_V$, we have 
    \[
        (\Par_{\LG}^{V} \times_{\Par_\LG, x} \Spec(R))(S) = \tau^{\leq 0} R\Gamma(\WD_{F}, x^* \cl{O}_V) = \tau^{\leq 0} (R\Gamma(\WD_{F}, x^* \cl{O}_V) \otimes S)
    \]
    for every animated $R$-algebra $S$ by \Cref{prop:WD_cohomology}. Here, $x$ is a map 
    \[ 
        x\colon [\Spec(R)/\WD_{F}] \to [ \ast /\LG], 
    \] 
    so $x^*\cl{O}_V$ is a perfect complex on $[\Spec(R)/\WD_{F}]$. 
\end{proof}

Note that the vectorial case does not cover many examples of relative Langlands duality (see \cite[Table 1.5.1]{BZSV}). However, we can see from the table in loc. cit. that vectorial extensions of homogeneous spaces do cover many examples. 

\begin{exa} \label{rmk:Vhomog}
    Let $\widehat{H} \subset \widehat{G}$ be a flat closed subgroup scheme stable under the action of $Q$ and let $\LH = \widehat{H} \rtimes Q$. When $\widehat{X} = \LG \times^{\LH} V$ for $V \in \Rep(\LH)$, $\pi_{\widehat{X}}$ is identified with
    \[
        \Par_{\LG}^{\widehat{X}} = \Par_{\LH}^{V} \xrightarrow{\pi_V} \Par_{\LH} \to \Par_{\LG}. 
    \]
\end{exa}

\section{Computations on the $\cl{B}$-side} \label{sec:Bcomputation}

In this section, we carry out the explicit computation of the unnormalized $L$-sheaf $\cl{L}_{\widehat{X}}$ for the dual pairs associated with the Iwasawa-Tate and Hecke periods.
%Building on the geometric properties of the relative spectral stack $\Par_{\LG}^{\widehat{X}}$ established in \Cref{sec:Bside}, we determine the structure of $\cl{L}_{\widehat{X}} = \pi_{\widehat{X}*} \omega_{\widehat{X}}$ as an object in $\IndCoh(\Par_{\LG})$.
For the Iwasawa-Tate period, we exploit the explicit description of $\Par_{\bb{G}_m}^{\bb{A}^1}$ as a derived vector bundle.
For the Hecke period, we utilize the spectral Eisenstein series and the weight decomposition of $\IndCoh(\Par_{\GL_2})$.
%These computations provide the spectral counterparts to the automorphic period sheaves computed in \Cref{sec:Aside}, necessary for the verification of the normalized period conjecture in \Cref{sec:comparison}.
\subsection{The Iwasawa-Tate $L$-sheaf} \label{ssec:IwTateL}

In this section, we compute the unnormalized $L$-sheaf $\cl{L}_{\bb{A}^1}$ over $\Lambda = \Qla$ when $\LG = \bb{G}_m$ and $\widehat{X} = \bb{A}^1$. Here, $\widehat{X} = \std$ as a $\bb{G}_m$-variety where $\std$ denotes the standard representation of $\bb{G}_m$. Since $\Par_{\bb{G}_m}$ is smooth, the natural embedding 
\[
    \QCoh(\Par_{\bb{G}_m}) \hookrightarrow \IndCoh(\Par_{\bb{G}_m})
\]
is an equivalence. We will compute in the category $\QCoh(\Par_{\bb{G}_m})$. 

Recall the decomposition $\Par_{\bb{G}_m} = \bigsqcup_{\chi \colon O_F^\times \to \Qlax} \Par_{\bb{G}_m, \chi}$ in \Cref{exa:ParGm}. Each $\Par_{\bb{G}_m, \chi}$ is isomorphic to $\bb{G}_m \times [\ast/\bb{G}_m]$. Let $\cl{L}_{\bb{A}^1, \chi}$ (resp.\ $\Par_{\bb{G}_m, \chi}^{\bb{A}^1}$) denote the restriction of $\cl{L}_{\bb{A}^1}$ (resp.\ $\Par_{\bb{G}_m}^{\bb{A}^1}$) to the component $\Par_{\bb{G}_m, \chi}$.

\begin{lem} \label{lem:IwTatent}
    If a smooth character $\chi \colon O_F^\times \to \Qlax$ is not trivial, $\Par_{\bb{G}_m, \chi}^{\bb{A}^1} \to \Par_{\bb{G}_m, \chi}$ is an isomorphism and we have $\cl{L}_{\bb{A}^1, \chi} \cong \cl{O}_{\Par_{\bb{G}_m, \chi}}$. 
\end{lem}

\begin{proof}
    The Weil-Deligne cohomology of the universal Weil-Deligne representation on $\Par_{\bb{G}_m, \chi}$ is trivial if $\chi$ is not trivial. Thus, the first claim follows from \Cref{prop:explicitrel}. The second claim follows from \Cref{prop:ParLGproperty}. 
\end{proof}

From now on, we concentrate on the component $\Par_{\bb{G}_m, \triv}$ associated to the trivial character of $O_F^\times$. Since $\Par_{\bb{G}_m, \triv} \cong \bb{G}_m \times [\ast/\bb{G}_m]$, $\QCoh(\Par_{\bb{G}_m, \triv}) \cong \prod_{\chi \in X^*(\bb{G}_m)} \QCoh(\bb{G}_m)$. Let $T$ denote the standard coordinate of $\bb{G}_m$. 

\begin{lem}
    The perfect complex on $\Par_{\bb{G}_m, \triv}$ associated to the derived vector bundle $\Par_{\bb{G}_m, \triv}^{\bb{A}^1} \to \Par_{\bb{G}_m, \triv}$ is given by the perfect complex on $\bb{G}_m$
    \[
        \left[ \Qla[T^{\pm}] \xrightarrow{T - 1} \Qla[T^{\pm}] \right] \oplus \left[ \Qla[T^{\pm}] \xrightarrow{qT - 1} \Qla[T^{\pm}] \right] [-1]
    \]
    which lies in $\QCoh(\Par_{\bb{G}_m, \triv})_{\std}$. 
\end{lem}
\begin{proof}
    The Weil-Deligne parameter on $\bb{G}_m \to \Par_{\bb{G}_m, \triv}$ is given by the unramified character $W_F \to \bb{Z} \to \Qla[T^\pm]$ sending $\Fr$ to $T$. Since $\widehat{X} = \std$, it follows from \Cref{prop:WD_cohomology} and \Cref{prop:explicitrel} that the desired perfect complex on $\bb{G}_m \times [\ast/\bb{G}_m]$ is given by
    \[
        \left[ \Qla[T^{\pm}] \boxtimes \cl{O}_{\std} \xrightarrow{(T - 1, 0)} (\Qla[T^{\pm}] \boxtimes \cl{O}_{\std})^{\oplus 2} \xrightarrow{(0, qT - 1)} \Qla[T^{\pm}] \boxtimes \cl{O}_{\std} \right]. 
    \]
\end{proof}

Let 
\[
    \cl{E}_\triv = \left[ \Qla[T^{\pm}] \xrightarrow{T - 1} \Qla[T^{\pm}] \right] ,\quad \cl{E}_{\cyc} = \left[ \Qla[T^{\pm}] \xrightarrow{qT - 1} \Qla[T^{\pm}] \right]
\]  
be perfect complexes over $\bb{G}_m$ regarded as objects in $\QCoh(\Par_{\bb{G}_m, \triv})_{\std}$.  Here, $\cl{E}_\triv$ is concentrated at the trivial character $\{ T = 1 \} \times [\ast/\bb{G}_m] \subset \Par_{\bb{G}_m, \triv}$ and $\cl{E}_{\cyc}$ is concentrated at the cyclotomic character $\{ T = q^{-1} \} \times [\ast/\bb{G}_m] \subset \Par_{\bb{G}_m, \triv}$. As a result, we get the following. 

\begin{lem} \label{lem:gluerelPar}
    Let $\bb{V}(\cl{E}_\triv)$ \textup{(resp.\ $\bb{V}(\cl{E}_{\cyc}[-1])$)} be the derived vector bundle on $\Par_{\bb{G}_m, \triv}$ corresponding to $\cl{E}_\triv$ \textup{(resp.\ $\cl{E}_{\cyc}[-1]$)}. Then, $\Par_{\bb{G}_m, \triv}^{\bb{A}^1}$ is isomorphic to $\bb{V}(\cl{E}_\triv)$ \textup{(resp.\ $\bb{V}(\cl{E}_{\cyc}[-1])$)} over the open subset $\{T \neq q^{-1}\}$ \textup{(resp.\ $\{T \neq 1\}$)} in $\Par_{\bb{G}_m, \triv}$. 
\end{lem}

Let 
\[
    U_1 = \{ T \neq 1 \}, \quad U_q = \{ T \neq q^{-1} \} \subset \bb{G}_m 
\]
be the above open subsets. 
%$U_1 = \{ T \neq 1 \} \subset \bb{G}_m$ (resp.\ $U_q = \{ T \neq q^{-1} \} \subset \bb{G}_m$). 
Our strategy for computing $\cl{L}_{\bb{A}^1, \triv}$ is to compute it over $U_1 \times [\ast/\bb{G}_m]$ and $U_q \times [\ast/\bb{G}_m]$ and glue the results together.

\subsubsection{Computation over $U_1 \times [\ast/\bb{G}_m]$}

In this section, we compute $\cl{L}_{\bb{A}^1, \triv}$ over $U_1 \times [\ast/\bb{G}_m]$. For this, we rely on the following lemma in \cite{FW25} (cf.\ \cite[Proposition G.1.5]{AG15}). 

\begin{lem} \textup{(\cite[Lemma 6.2.2]{FW25})} \label{lem:FWlem}
    Let $S$ be an eventually coconnective Artin stack locally of finite type over $\Qla$. Let $\cl{E}$ be a perfect complex on $S$ with tor-amplitude in $[0,\infty)$ and let $\pi \colon E = \Tot_{S}(\cl{E}) \to S$ be the projection. Let 
    \[
        \widehat{\pi}_* \colon \IndCoh(E) \to \IndCoh(S)
    \]
    denote the pushforward from the infinitesimal neighborhood of the zero section $S \hookrightarrow E$ along $\pi$ \textup{(see \cite[Section 4.4]{FW25})}. Let $\omega_{E/S} = \pi^! \cl{O}_S$ denote the relative dualizing complex of $E$ over $S$. Then, 
    \[
        \widehat{\pi}_*(\omega_{E/S}) \cong \Sym^\bullet(\cl{E}) \in \QCoh(S) \xhookrightarrow{\Xi_S} \IndCoh(S). 
    \]
\end{lem}

Let $E_1 = \Par_{\bb{G}_m, \triv}^{\bb{A}^1}\vert_{U_1 \times [\ast/\bb{G}_m]}$. By \Cref{lem:gluerelPar}, we have $E_1 = \bb{V}(\cl{E}_{\cyc}[-1])\vert_{U_1 \times [\ast/\bb{G}_m]}$. The projection $\pi_1\colon E_1 \to U_1 \times [\ast/\bb{G}_m]$ is bijective on the underlying spaces since $\cl{E}_{\cyc}[-1]$ has tor-amplitude in $[1,\infty)$, so we have $\pi_{1*} = \widehat{\pi}_{1*}$. Since the dualizing complex of $\Par_{\LG}$ is trivial by \Cref{prop:ParLGproperty}, we have
\[
    \cl{L}_{\bb{A}^1, \triv}\vert_{U_1 \times [\ast/\bb{G}_m]} \cong \Sym^\bullet(\cl{E}_{\cyc}[-1])\vert_{U_1\times [\ast/\bb{G}_m]}
\]
by \Cref{lem:FWlem}. We have $\QCoh(U_1 \times [\ast/\bb{G}_m]) \cong \prod_{\chi \in X^*(\bb{G}_m)} \QCoh(U_1 \times [\ast/\bb{G}_m])_\chi$. Then, $\Sym^n(\cl{E}_{\cyc}[-1])\vert_{U_1\times [\ast/\bb{G}_m]} \in \QCoh(U_1 \times [\ast/\bb{G}_m])_{\std^{n}}$ for each $n \geq 0$. As a result, we get the following. 

\begin{prop} \label{prop:computeU1}
    We have a decomposition
    \[
        \cl{L}_{\bb{A}^1, \triv}\vert_{U_1 \times [\ast/\bb{G}_m]} \cong \cl{O}_{U_1} \oplus \bigoplus_{n \geq 1} \left[ \cl{O}_{U_1} \xrightarrow{qT - 1} \cl{O}_{U_1}\right][- 2n + 1]
    \]
    in $\QCoh(U_1 \times [\ast/\bb{G}_m])$. Here, the term indexed by $n \geq 1$ (resp. the first term) on the right hand side is an object in $\QCoh(U_1 \times [\ast/\bb{G}_m])_{\std^{n}}$ \textup{(resp.\ $\QCoh(U_1 \times [\ast/\bb{G}_m])_{\triv}$)}. 
\end{prop}
\begin{proof}
    The claim follows from the above discussion and the computation of $\Sym^n(\cl{E}_{\cyc}[-1])$ in \Cref{lem:expSympow}. 
\end{proof}

\subsubsection{Computation over $U_q \times [\ast/\bb{G}_m]$}

In this section, we do the computation of $\cl{L}_{\bb{A}^1, \triv}$ over $U_q \times [\ast/\bb{G}_m]$. For this, we will provide an explicit description of $E_q = \Par_{\bb{G}_m, \triv}^{\bb{A}^1}\vert_{U_q \times [\ast/\bb{G}_m]}$. By \Cref{lem:gluerelPar}, we have $E_q = \bb{V}(\cl{E}_\triv)\vert_{U_q \times [\ast/\bb{G}_m]}$. 

\begin{lem} \label{lem:expEq}
    The natural map $\cl{E}_\triv \to \cl{O}_{\std}$ on $\Par_{\bb{G}_m, \triv}$ induces a closed immersion $\bb{V}(\cl{E}_\triv) \hookrightarrow \bb{V}(\cl{O}_\std)$. Its base change to $\bb{G}_m$ is a regular closed immersion 
    \[
        \bb{V}(\cl{E}_\triv) \times_{\Par_{\bb{G}_m, \triv}} \bb{G}_m \hookrightarrow \bb{A}^1 \times \bb{G}_m
    \]
    given by $(T-1)x = 0$, where $x$ denotes the coordinate on $\bb{A}^1$. 
\end{lem}
\begin{proof}
    It is enough to prove the second claim. By definition, $\bb{V}(\cl{E}_\triv) \times_{\Par_{\bb{G}_m, \triv}} \bb{G}_m$ is the relative spectrum $\underline{\Spec} (\Sym^\bullet \cl{E}_\triv^\vee)$. Here, $\cl{E}_\triv$ is regarded as a perfect complex on $\bb{G}_m$. By \Cref{lem:expSympow}, the natural map $\Sym^\bullet \cl{O}_{\bb{G}_m}^\vee \to \Sym^\bullet \cl{E}_\triv^\vee$ is surjective and its kernel is generated by $(T-1) x$. Thus, we get the claim. 
\end{proof}

\begin{lem} \label{lem:dualtriv}
    Let $\omega_{\triv}$ denote the relative dualizing complex of $\bb{V}(\cl{E}_\triv) \to \Par_{\bb{G}_m, \triv}$. Then, we have $\omega_\triv \cong \cl{O}_{\bb{V}(\cl{E}_\triv)}$. 
\end{lem}
\begin{proof}
    Let $f\colon \Spec(R) \to \Par_{\bb{G}_m, \triv}$ be a smooth morphism and let $q\colon \bb{V}(\cl{E}_\triv)_R \to \Spec(R)$ be the base change of $\bb{V}(\cl{E}_\triv) \to \Par_{\bb{G}_m, \triv}$ along $f$. It is enough to construct a functorial isomorphism $q^!\cl{O}_{\Spec(R)} \cong \cl{O}_{\bb{V}(\cl{E}_\triv)_R}$. 

    Let $p \colon \bb{V}(\cl{O}_\std)_R \to \Spec(R)$ be the base change of $\bb{V}(\cl{O}_\std) \to \Par_{\bb{G}_m, \triv}$. Then, $P = \Sym^\bullet \cl{O}_{\std, R}^\vee$ is the ring of global section on $\bb{V}(\cl{O}_{\std})_R$. Let $i \colon \bb{V}(\cl{E}_\triv)_R \hookrightarrow \bb{V}(\cl{O}_\std)_R$ be the regular closed immersion. Let $I = (T-1)\cdot \Sym^{\geq 1}\cl{O}_{\std, R}^\vee \subset P$ be an invertible ideal. Then, $S = P / I$ is the ring of global sections on $\bb{V}(\cl{E}_\triv)_R$. Since $R \to P$ is smooth of dimension $1$, $q^!\cl{O}_{\Spec(R)}$ is represented by an $S$-module
    \[
        R\Hom_P(S, \Omega_{P/R}[1])
    \]
    and this is canonically isomorphic to $\Omega_{P/R} \otimes_P I^{-1}/P$. Since $\Omega_{P/R} \cong \cl{O}_{\std, R}^\vee \otimes P$ and $I^{-1}/P \cong (T-1)^{-1} \cdot \cl{O}_{\std, R} \otimes S$, $q^!\cl{O}_{\Spec(R)}$ is canonically isomorphic to $\cl{O}_{\Spec(S)}$. 
\end{proof}

\begin{rmk} \label{rmk:compatisom}
    Let $\iota \colon \omega_\triv \cong \cl{O}_{\bb{V}(\cl{E}_\triv)}$ be the above isomorphism. Since $\bb{V}(\cl{E}_\triv) \to \Par_{\bb{G}_m, \triv}$ is an isomorphism over $U_1 \times [\ast/\bb{G}_m]$, there is an another isomorphism $\iota'\colon \omega_{\triv}\vert_{U_1\times [\ast/\bb{G}_m]} \cong \cl{O}_{\bb{V}(\cl{E}_\triv)\vert_{U_1\times [\ast/\bb{G}_m]}}$. Then, $\iota\vert_{\bb{V}(\cl{E}_\triv)\vert_{U_1\times [\ast/\bb{G}_m]}} \circ \iota'^{-1}$ is the multiplication by $T-1$. This can be checked after the pullback along $\bb{G}_m \to \Par_{\bb{G}_m, \triv}$ and stems from the difference of the choice of a defining equation of $\bb{V}(\cl{E}_\triv) \times_{\Par_{\bb{G}_m, \triv}} \bb{G}_m \hookrightarrow \bb{A}^1 \times \bb{G}_m$. 
\end{rmk}

Let $\pi_q \colon E_q \to U_q \times [\ast/\bb{G}_m]$ be the projection. By \Cref{lem:dualtriv}, we have
\[
    \cl{L}_{\bb{A}^1, \triv}\vert_{U_q \times [\ast/\bb{G}_m]} \cong \pi_{q*} \cl{O}_{E_q} \cong \Sym^\bullet(\cl{E}_\triv^\vee)\vert_{U_q \times [\ast/\bb{G}_m]}. 
\]
We have the weight decomposition $\QCoh(U_q \times [\ast/\bb{G}_m]) \cong \prod_{\chi \in X^*(\bb{G}_m)} \QCoh(U_q \times [\ast/\bb{G}_m])_\chi$. Then, $\Sym^n(\cl{E}_{\triv}^\vee)\vert_{U_q \times [\ast/\bb{G}_m]} \in \QCoh(U_q \times [\ast/\bb{G}_m])_{\std^{-n}}$ for each $n \geq 0$. As a result, we get the following. 

\begin{prop} \label{prop:computeUq}
    We have a decomposition
    \[
        \cl{L}_{\bb{A}^1, \triv}\vert_{U_q \times [\ast/\bb{G}_m]} \cong \cl{O}_{U_q} \oplus \bigoplus_{n \geq 1} \left[ \cl{O}_{U_q} \xrightarrow{T - 1} \cl{O}_{U_q}\right][1]
    \]
    in $\QCoh(U_q \times [\ast/\bb{G}_m])$. Here, the term indexed by $n \geq 1$ (resp. the first term) on the right hand side is an object in $\QCoh(U_q \times [\ast/\bb{G}_m])_{\std^{-n}}$ \textup{(resp.\ $\QCoh(U_q \times [\ast/\bb{G}_m])_{\triv}$)}. 
\end{prop}
\begin{proof}
    The claim follows from the above discussion and the computation of $\Sym^n(\cl{E}_\triv^\vee)$ in \Cref{lem:expSympow}. 
\end{proof}

\subsubsection{Gluing}

By combining \Cref{prop:computeU1} and \Cref{prop:computeUq}, we get the following. 

\begin{prop} \label{prop:IwTateLsheaf}
    We have a decomposition
    \[
        \cl{L}_{\bb{A}^1} \cong \cl{O}_{\Par_{\bb{G}_m}} \oplus \bigoplus_{n \geq 1} \cl{O}_{\bb{G}_m}/(T-1) \oplus \bigoplus_{n \geq 1} \cl{O}_{\bb{G}_m}/(qT-1)[- 2n]
    \]
    Here, the first term lies in $\QCoh(\Par_{\bb{G}_m})_\triv$, and the $n$-th term of the second factor (resp.\ the third factor) lies in $\QCoh(\Par_{\bb{G}_m, \triv})_{\std^{-n}}$ \textup{(resp.\ $\QCoh(\Par_{\bb{G}_m, \triv})_{\std^{n}}$)}. 
\end{prop}
\begin{proof}
    The contribution outside the trivial component follows from \Cref{lem:IwTatent}. For a nontrivial character $\chi \in X^*(\bb{G}_m)$, the $\chi$-component of $\cl{L}_{\bb{A}^1, \triv}$ appears at most once in the descriptions of \Cref{prop:computeU1} and \Cref{prop:computeUq}. Thus, the claim for the second and third factors follows from loc. cit., and it is enough to identify the factor in $\QCoh(\Par_{\bb{G}_m, \triv})_{\triv} \cong \QCoh(\bb{G}_m)$. By \Cref{rmk:compatisom}, it is the gluing of $\cl{O}_{U_1}$ and $\cl{O}_{U_q}$ along $\cl{O}_{U_1}\vert_{U_1 \cap U_q} \xrightarrow{T-1} \cl{O}_{U_q}\vert_{U_1 \cap U_q}$. Thus, it is isomorphic to $\cl{O}_{\bb{G}_m}$ by the isomorphisms $\cl{O}_{\bb{G}_m}\vert_{U_1} \xrightarrow{(T-1)^{-1}} \cl{O}_{U_1}$ and $\cl{O}_{\bb{G}_m}\vert_{U_q} \cong \cl{O}_{U_q}$. 
\end{proof}

We can summarize the computation on $\Par_{\bb{G}_m, \triv}$ in the following table: 
\begin{table}[tbhp!]
    \centering
    \begin{tabular}{c|c|c|c}
        \text{central weight} & \text{support} &  \text{local on $U_1$}& \text{local on $U_q$} \\ 
        \hline
        $\triv$ & $\Par_{\bb{G}_m, \triv}$  &$\cl{O}_{U_1}$ &$\cl{O}_{U_q}$   \\ 
        \hline 
        $\std^{-n}, n \ge 1$ &$\crbr{T=1}$ &$0$& \begin{tabular}{c}
        $\sqbr{\cl{O}_{U_q} \xra{T-1} \cl{O}_{U_q}}[1] $ \\
            $\cong j_{1*}\cl{O}$
        \end{tabular}  
        \\ 
        \hline 
        $\std^n, n \ge 1$ & 
    $\crbr{T=q^{-1}} $  &
    \begin{tabular}{c}
          $\sqbr{\cl{O}_{U_1}\xra{qT-1} \cl{O}_{U_1}}[-2n+1]$  \\
        $ \cong j_{q^{-1}*} \cl{O}[-2n] $ 
    \end{tabular}
    & $0$ 
    \end{tabular}
    \label{tab:L_sheaf_components_on_grading}
\end{table}

Here, $j_x \colon \crbr{T=x} \hra \bb{G}_m$ denotes the inclusion of the closed point for $x \in \crbr{1,q^{-1}}$.

\subsection{The Hecke $L$-sheaf} \label{ssec:HeckeL}

In this section, we compute the unnormalized $L$-sheaf $\cl{L}_{\std}$ over $\Lambda = \Qla$ when $\LG = \GL_2$ and $\widehat{X} = \bb{A}^2$. Here, $\widehat{X} = \std$ as a $\GL_2$-variety where $\std$ denotes the standard representation of $\GL_2$. 

Let $Z \subset \GL_2$ be the diagonal central torus. We have the weight decomposition
\[
    \IndCoh(\Par_{\GL_2}) = \prod_{\chi \in X^*(Z)} \IndCoh(\Par_{\GL_2})_\chi. 
\]
with respect to the $Z$-gerbe $\Par_{\GL_2} \to [Z^1(\WD_F, \GL_2) / (\GL_2 / Z)]$. Fix a standard isomorphism $Z \cong \bb{G}_m$. We will give the description of each component of $\cl{L}_\std$ using the spectral Eisenstein series. Let $T \subset B \subset \GL_2$ be the standard Borel pair and let $\ov{B}$ denote the opposite of $B$. Consider the following diagram. 
\begin{center}
    \begin{tikzcd}
        & \Par_{B} \ar[ld, "p^\spec"'] \ar[rd, "q^\spec"] &  \\
        \Par_{\GL_2} & &  \Par_T \\
        & \Par_{\ov{B}} \ar[lu, "\ov{p}^\spec"] \ar[ru, "\ov{q}^\spec"'] &
    \end{tikzcd}
\end{center}

\begin{defi}(\cite[Proposition 3.3.2]{Zhu21})
    The spectral Eisenstein series functors associated to $B$ and $\ov{B}$ are defined as
    \[
        \Eis_{B}^\spec = p^{\spec}_* q^{\spec !} \colon \IndCoh(\Par_T) \to \IndCoh(\Par_{\GL_2}), 
    \]
    \[
        \Eis_{\ov{B}}^\spec = \ov{p}^{\spec}_* \ov{q}^{\spec !} \colon \IndCoh(\Par_T) \to \IndCoh(\Par_{\GL_2}). 
    \]
\end{defi}

Let $i_\triv \colon \Spec(\Qla) \to \Par_{\bb{G}_m}$ be the map associated to the trivial $L$-parameter. Since it is schematic, we may consider $i_{\triv *} \Qla$. Let us denote its weight decomposition
\[
    i_{\triv *} \Qla = \bigoplus_{\chi \in X^*(\bb{G}_m)} (i_{\triv *} \Qla)_{\chi}. 
\]
Here, $(i_{\triv *} \Qla)_{\chi} \in \IndCoh(\Par_{\bb{G}_m})_{\chi} \cong \QCoh(Z^1(W_F, \bb{G}_m))$ is one-dimensional over $\Qla$. 

Let us denote the weight decomposition
\[
    \cl{L}_\std = \bigoplus_{n \in \bb{Z}} \cl{L}_{\std, n}
\]
with $\cl{L}_{\std, n} \in \IndCoh(\Par_{\GL_2})_{\std^n}$. By using $\Par_T \cong \Par_{\bb{G}_m} \times \Par_{\bb{G}_m}$, we have the following description of $\cl{L}_{\std, n}$. 

\begin{prop} \label{prop:HeckeL}
    For $n < 0$, we have
    \[
        \cl{L}_{\std, n} = \Eis_{\ov{B}}^\spec(\cl{O}_{\Par_{\bb{G}_m}} \boxtimes (i_{\triv*} \Qla)_{\std^n}),
    \]
    for $n = 0$, we have a fiber sequence
    \[
        \cl{O}_{\Par_{\GL_2}} \to \cl{L}_{\std, 0} \to \Eis_{\ov{B}}^\spec(\cl{O}_{\Par_{\bb{G}_m}} \boxtimes  (i_{\triv*} \Qla)_{\triv}), 
    \]
    and for $n > 0$, we have a fiber sequence
    \[
        R\Gamma(\WD_F^n \rtimes \Sigma_n, \cl{O}_{\std}^{\boxtimes n}) \to \cl{L}_{\std, n} \to \Eis_{\ov{B}}^\spec( \cl{O}_{\Par_{\bb{G}_m}} \boxtimes (i_{\triv*} \Qla)_{\std^n}). 
    \]
    Here, $\cl{O}_{\std}$ denotes the pullback of $\cl{O}_{\std}\vert_{[\ast/\GL_2]}$ along $\Par_{\GL_2} \to [\ast/\GL_2]$, which is equipped with the natural $\WD_F$-action via the universal $L$-parameter. 
\end{prop}
\begin{proof}
    The same argument as in \cite[Section 7.2]{FW25} works here and provides a fiber sequence of $\cl{L}_\std$. Then, we will argue that the fiber sequence is simplified as in the statement. % for $n \leq 0$. 

    First, let us recall the construction of a fiber sequence. We have a closed-open decomposition
    \begin{center}
        \begin{tikzcd}
            \Par_{\GL_2} \ar[r, "0", hook] \ar[rd, equal] & \Par_{\GL_2}^{\std} \ar[d, "\pi_\std"] & U \ar[l, "j"', hook'] \ar[ld]\\
            & \Par_{\GL_2}. & 
        \end{tikzcd}
    \end{center}
    Here, $0$ denotes the zero section of a derived vector bundle. Let $\cl{Z}$ denote the infinitesimal neighborhood of the zero section and let $\widehat{\pi}_\std \colon \cl{Z} \to \Par_{\GL_2}$ be the projection. Then, we have a fiber sequence
    \[
        \widehat{\pi}_{\std *} \omega_{\cl{Z}} \to \cl{L}_\std \to \pi_{\std *} j_* \omega_U. 
    \]
    %Since $\WD_F$ and its self-products are reduced, a closed point of $\Par_{\GL_2}^\std$ lies in the zero section if and only if the image of the associated map $[\ast / \WD_{F,\Qla}] \to [\bb{A}^2/\GL_2]$ equals $[\{0\}/\GL_2] \subset [\bb{A}^2/\GL_2]$. Then, the complement $U$ is identified with the mapping stack
    By \Cref{lem:orbitdec}, $U$ is identified with the mapping stack
    \[
        \und{\Map}([\ast/\WD_F], [(\bb{A}^2 - \{0\})/\GL_2]) \cong \und{\Map}([\ast/\WD_F], [\ast/\ov{\Mir}_2])
    \]
    Here, $\ov{\Mir}_2$ is the stabilizer of ${\scriptsize \begin{pmatrix} 0 \\ 1 \end{pmatrix}}$ and written as
    \[
        \ov{\Mir}_2 = \left\{ \begin{pmatrix} \ast & 0 \\ \ast & 1 \end{pmatrix} \right\} \subset \GL_2. 
    \]
    Then, $\pi_\std \circ j$ factors through $\Par_{\ov{B}}$ as in the following diagram.  
    \begin{center}
        \begin{tikzcd}
            \Par_{\GL_2} & \Par_{\ov{B}} \ar[l, "\ov{p}^\spec"] \ar[d, "\ov{q}^\spec"] & U \ar[l] \ar[d] \\
            & \Par_T & \Par_{\bb{G}_m} \ar[l, "\id \times i_\triv"]
        \end{tikzcd}
    \end{center}
    Since the above square is Cartesian, the base change formula provides
    \[
        j_* \pi_{\std*} \omega_U \cong \ov{p}^\spec_* \ov{q}^{\spec !} (\id \times i_
        \triv)_* \omega_{\Par_{\bb{G}_m}} \cong \Eis_{\ov{B}}^\spec(\cl{O}_{\Par_{\bb{G}_m}} \boxtimes i_{\triv *} \Qla). 
    \]
    It can be easily checked that its weight decomposition is given by
    \[
        j_* \pi_{\std*} \omega_U = \bigoplus_{n\in \bb{Z}} \Eis_{\ov{B}}^\spec(\cl{O}_{\Par_{\bb{G}_m}} \boxtimes (i_{\triv *} \Qla)_{\std^n}) 
    \]
    with $\Eis_{\ov{B}}^\spec(\cl{O}_{\Par_{\bb{G}_m}} \boxtimes (i_{\triv *} \Qla)_{\std^n}) \in \IndCoh(\Par_{\GL_2})_{\std^n}$. 

    Next, we study $\widehat{\pi}_{\std *} \omega_{\cl{Z}}$. Consider the universal Weil-Deligne representation of rank $2$ over $\Par_{\LG}$ and let 
    \[
        \cl{E} = \left[\cl{O}_\std^{I_F} \xrightarrow{(\Fr-1, N)} \cl{O}_\std^{I_F} \oplus \cl{O}_\std^{I_F} \xrightarrow{(N, -q \Fr + 1)} \cl{O}_\std^{I_F} \right]
    \]
    denote its Weil-Deligne cohomology as a perfect complex on $\Par_{\LG}$ (see \Cref{prop:WD_cohomology}). 

    Since the dualizing complex of $\Par_{\LG}$ is trivial (see \Cref{prop:ParLGproperty}), we have 
    \[
        \widehat{\pi}_{\cl{Z}*} \omega_{\cl{Z}} \cong \Sym^\bullet \cl{E} \in \QCoh(\Par_{\LG}) \xhookrightarrow{\Xi} \IndCoh(\Par_{\LG})
    \]
    by \Cref{lem:FWlem}. Here, $\Sym^n \cl{E} \in \QCoh(\Par_{\LG})_{\std^n}$ for each $n \geq 0$. Since we work in characteristic $0$ and $\Sym^n \cl{E} = (\cl{E}^{\otimes n})_{\Sigma_n}$, it is identified with $R\Gamma(\WD_F^n \rtimes \Sigma_n, \cl{O}_{\std}^{\boxtimes n})$. Thus, it follows that the fiber sequence
    \[
        \widehat{\pi}_{\std *} \omega_{\cl{Z}} \to \cl{L}_\std \to \pi_{\std *} j_* \omega_U. 
    \]
    specializes to the claim as in the statement. 
\end{proof}

\begin{rmk}
    For $n > 0$, the above fiber sequence is inefficient to give the whole description of $\cl{L}_{\std, n}$. In the comparison with the corresponding period sheaf, we will instead use the \textit{functional equation} developed in \cite[Section 11.10.2]{BZSV}. See \Cref{sssec:FEL} for more details. 
\end{rmk}

\section{Normalization} \label{sec:normalization}

In this section, we introduce the normalization of period sheaves and $L$-sheaves, and translate the normalized period conjecture into the setting of the categorical local Langlands correspondence. 
%needed to match period sheaves and $L$-sheaves for (conjectural) dual pairs $(G,X) \leftrightarrow (\LG, \widehat{X})$. Here, $X$ is a $G$-variety over $F$ and $\widehat{X}$ is a $\LG$-scheme over $\Lambda$ taken as in \Cref{defi:relspec}. 

\subsection{Normalization}

In this section, we introduce the normalization of period sheaves and $L$-sheaves.
%additional structures on dual pairs to make suitable normalization on periods and $L$-sheaves. %As a convention (see \cite[Section 2.6.1]{BZSV}), $\bb{G}_{gr}$ denotes a copy of $\bb{G}_m$, which is used for grading. 
From now on, we assume that $\Lambda$ is a $\bb{Q}_\ell[\sqrt{q}]$-algebra. Here, the choice of $\sqrt{q} \in \Lambda$ is an analogue of a square root of the canonical sheaf in the geometric setting \cite{BZSV}. 

\subsubsection{The $\cl{A}$-side}

First, we introduce the normalization of $\cl{P}_X$. From the choice of $\sqrt{q}$, we can take a square root $\Lambda_{\norm}^{1/2}$ of the norm character of $F^\times$ so that $\Lambda_\norm^{1/2}(x) = q^{-\ord(x)/2}$ for $x \in F^\times$. Then, we define the degree sheaf as follows. 

\begin{defi} \label{defi:degshf}
    Let $\underline{\deg} \in \cl{D}^\oc(\Bun_{\bb{G}_m}, \Lambda)$ be the invertible $\Lambda$-complex such that
    \[
        \underline{\deg}\vert_{\Bun_{\bb{G}_m, n}} = \Lambda^{- 1/2}_\norm[n]. 
    \]
\end{defi}

\begin{rmk}
    We can interpret the degree sheaf as the shearing of the half-volume factor, which is similar to the geometric setting (see \Cref{sec:classical_degree_normalization}, \cite[Remark 10.4.1]{BZSV}). 
\end{rmk}

Now, we will assume that $X$ is a smooth quasi-projective $G$-variety. The canonical bundle of $X$ is a $G$-equivariant line bundle, so it provides a map
\[
    \Omega_X^\tp \colon [X / G] \to [\ast / \bb{G}_m]. 
\]

\begin{defi}\label{defi:normP}
    Let $\omega_X \colon \Bun_G^X \to \Bun_{\bb{G}_m}$ be the composition map with $\Omega_X^\tp$. We define the normalized period sheaf associated to $X$ by
    \[
        \cl{P}_X^\norm = \pi_{X!} \omega_X^* \und{\deg}. 
    \]
\end{defi}

\begin{prop} \label{prop:normPetaX}
    Suppose that $X$ admits a nowhere vanishing eigen-volume form $\omega$. Let $\eta_X \colon G \to \bb{G}_m$ be the eigencharacter of $\omega$ such that $g^*\omega = \eta_X(g) \omega$. Then, we have 
    \[
        \cl{P}_X^\norm \cong \cl{P}_X \otimes \eta_X^* \und{\deg}. 
    \]
    In particular, if $X$ is unimodular, i.e. $X$ admits a $G$-equivariant volume form, $\cl{P}_X^\norm \cong \cl{P}_X$. 
\end{prop}
\begin{proof}
    The canonical bundle of $X$ is trivialized by $\omega$, and $\Omega_X^\tp$ factors as
    \[
        [X / G] \to [\ast / G] \xrightarrow{\eta_X} [\ast / \bb{G}_m]. 
    \]
    Then, $\omega_X^* \und{\deg} = \pi_X^* \eta_X^* \und{\deg}$, so the claim follows from the projection formula. 
\end{proof}
\begin{rmk}
    This description is similar to the normalization in the geometric setting (see \cite[Remark 10.4.1]{BZSV}). 
\end{rmk}

%\begin{rmk}
%    If $X$ admits a nowhere vanishing eigen-volume form $\omega$, the canonical bundle of $X$ is trivialized by $\omega$, and $\Omega_X^\tp$ factors as
%    \[
%        [X / G] \to [\ast / G] \xrightarrow{\eta_X} [\ast / \bb{G}_m]. 
%    \]
%    Here, $\eta_X \colon G \to \bb{G}_m$ is an eigencharacter of $\omega$ such that $g^*\omega = \eta_X(g) \omega$. By the projection formula, we have
%    \[
%        \cl{P}_X^\norm \cong \cl{P}_X \otimes \eta_X^* \und{\deg}. 
%    \]
%    In particular, if $X$ is unimodular, i.e. $X$ admits a $G$-equivariant volume form, $\cl{P}_X^\norm \cong \cl{P}_X$. 
%\end{rmk}
In general, $X$ does not admit such an eigen-volume form $\omega$ and $\omega_X^* \und{\deg}$ cannot be defined over $\Bun_G$, which is common especially when $X$ is proper. 

\begin{exa} \label{exa:flagvar}
    Let $P = MN \subset G$ be a parabolic subgroup, where $M$ is a Levi subgroup and $N$ is the unipotent radical. Let $\xi_P \colon M \to \bb{G}_m$ be the sum of roots in $N$. When $X = G / P$, the canonical bundle of $X$ is an anti-ample line bundle $\cl{O}(-\xi_P)$, so $\Omega_X^\tp$ factors as
    \[
        [X / G] = [\ast / P] \to [\ast / M] \xrightarrow{\xi_P^{-1}} [\ast / \bb{G}_m]. 
    \]
    It follows from the explicit description in \cite[Proposition 3.9]{HI24} that $\omega_X^* \und{\deg}$ equals the square root of the dualizing complex of $\Bun_P$, determined by the choice of $\sqrt{q} \in \Lambda$. 
\end{exa}

%\begin{rmk}
%    In contrast to \cite[Remark 10.4.2]{BZSV}, our normalization of $\cl{P}_X$ concerns only with the degree sheaf. This is due to the lack of dualizing complexes on the Fargues-Fontaine curve. We are not sure if our normalization is enough, but we expect that simpler normalization would do due to the simpler situation such as $\dim \Bun_G = 0$ in our case. 
%\end{rmk}

In the representation theory, a smooth $G(F)$-representation $\pi$ over $\Qla$ is \textit{$X$-distinguished} if $\Hom_{G(F)}(C_c^\infty(X(F), \Qla), \pi) \neq 0$. Here, one usually assumes that $X$ is quasi-affine and unimodular. Then, \Cref{cor:computetriv} suggests a geometrization of this definition. 

\begin{defi}
    We say that $A \in \cl{D}^\oc(\Bun_G, \Lambda)$ is $X$-distinguished if $\Map(\cl{P}_X^\norm, A) \neq 0$. 
\end{defi}

\begin{rmk}
    By \Cref{cor:computetriv} and \Cref{prop:normPetaX}, $\pi$ is $X$-distinguished if and only if $i^1_* \pi$ is $X$-distinguished when $X$ is quasi-affine and unimodular. 
\end{rmk}

%\begin{rmk}
%    In the representation theory, a smooth $G(F)$-representation $\pi$ over $\Qla$ is \textit{$X$-distinguished} if $\Hom_{G(F)}(C_c^\infty(X(F), \Qla), \pi) \neq 0$. Here, one usually assumes that $X$ is smooth, quasi-affine and unimodular, meaning that $X$ admits a $G$-equivariant volume-form. \Cref{cor:computetriv} suggests a geometrization of this definition: we say that $A \in \cl{D}^\oc(\Bun_G, \Lambda)$ is $X$-distinguished if $\Hom(\cl{P}_X^\norm, A) \neq 0$ (see \Cref{defi:normP} for the normalization). When $X$ is unimodular, $\cl{P}_X^\norm = \cl{P}_X$, so under the usual assumptions, $\pi$ is $X$-distinguished if and only if $i^1_* \pi$ is $X$-distinguished. 
%\end{rmk}

\subsubsection{The $\cl{B}$-side}

Next, we introduce the normalization of $\cl{L}_{\widehat{X}}$. Let $\bb{G}_{gr}$ denote a copy of $\bb{G}_m$ and we endow $\widehat{X}$ with a grading, a left $\bb{G}_{gr}$-action commuting with the $\LG$-action. Here, we normalize $\Par_{\LG}^{\widehat{X}}$ itself, similarly to the $\cl{A}$-side normalization in \cite[Section 10.2]{BZSV}. 

From the choice of $\sqrt{q}$, we have an unramified $L$-parameter $\sqrt{\cyc} \colon W_F \to \bb{G}_{gr}$ sending $\Fr$ to $q^{-1/2}$. It defines a point $\sqrt{\cyc} \in \Par_{\bb{G}_{gr}}(\Lambda)$. 

\begin{defi} \label{defi:normParLGX}
    We define the normalized relative stack $\Par_{\LG}^{\widehat{X}, \norm}$ by the Cartesian diagram
    \begin{center}
        \begin{tikzcd}[column sep = large]
            \Par_{\LG}^{\widehat{X}, \norm} \ar[r] \ar[d, "\pi_{\widehat{X}}^{\norm}"] & \Par_{\LG \times \bb{G}_{gr}}^{\widehat{X}} \ar[d, "\pi_{\widehat{X}}"] \\
            \Par_{\LG} \ar[r, "\id \times \sqrt{\cyc}"] & \Par_{\LG\times \bb{G}_{gr}}. 
        \end{tikzcd}
    \end{center}
\end{defi}

\begin{prop}
    The map $\pi_{\widehat{X}}^\norm$ is schematic. In particular, $\Par_{\LG}^{\widehat{X}, \norm}$ has an affine diagonal and is a $1$-Artin stack locally of finite presentation. 
\end{prop}

\begin{proof}
By \Cref{prop:ParLGXschematic}, 
the morphism
\[
  \pi_{\widehat{X}} :
    \Par_{\LG \times \bb{G}_{gr}}^{\widehat{X}} 
    \longrightarrow \Par_{\LG \times \bb{G}_{gr}}
\]
is schematic (and in particular has affine diagonal and is locally of finite presentation). Then, the claim follows since these properties are preserved under base change. 
\end{proof}

Let $\omega^\norm_{\widehat{X}}$ be the dualizing complex of $\Par_{\LG}^{\widehat{X}, \norm}$. Since the point $\sqrt{\cyc}$ factors as 
\[
    \sqrt{\cyc} \colon \ast \to [\ast / \bb{G}_{gr}] \hookrightarrow \Par_{\bb{G}_{gr}}, 
\]
the above diagram factors as
\begin{center}
    \begin{tikzcd}
        \Par_{\LG}^{\widehat{X}, \norm} \ar[r] \ar[d, "\pi_{\widehat{X}}^{\norm}"] & \lbrack \Par_{\LG}^{\widehat{X}, \norm} / \bb{G}_{gr} \rbrack \ar[d, "p_{\widehat{X}}"] &  \\
        \Par_{\LG} \ar[r, "q_{\widehat{X}}"] & \Par_{\LG} \times [\ast / \bb{G}_{gr}] \ar[r, hook] & \Par_{\LG} \times \Par_{\bb{G}_{gr}}. 
    \end{tikzcd}
\end{center}

%Here $q_{\widehat{X}}$ is the map $x \mapsto (x, \ast)$ induced by the trivial
%$\bb{G}_{gr}$-torsor on $\Par_{\LG}$, while $p_{\widehat{X}}$ is the morphism induced
%by the pair consisting of $\pi_{\widehat{X}}^\norm$ and the universal classifying map
%\[
%  \Par_{\LG}^{\widehat{X}, \norm} \longrightarrow [\ast / \bb{G}_{gr}]
%\]
%for its $\bb{G}_{gr}$-action. The last arrow is the product of $\id_{\Par_{\LG}}$ with
%the chosen lift
%\[
%  \sqrt{\cyc} : [\ast / \bb{G}_{gr}] \longrightarrow \Par_{\bb{G}_{gr}}
%\]
%of the square root of the cyclotomic character.

Since $\omega_{\widehat{X}}^\norm$ is isomorphic to the relative dualizing complex of $\pi_{\widehat{X}}^{\norm}$, we have 
\[
    \pi_{\widehat{X}*}^\norm \omega_{\widehat{X}}^\norm \cong q_{\widehat{X}}^* p_{\widehat{X}*} \omega_{p_{\widehat{X}}}
\]
by the smooth base change along $q_{\widehat{X}}$. Here, $\omega_{p_{\widehat{X}}}$ is the relative dualizing complex of $p_{\widehat{X}}$. In particular, $\pi_{\widehat{X}*}^\norm \omega_{\widehat{X}}^\norm$ is naturally equipped with a grading, so we can take the weight decomposition and the shearing
\[
    \pi_{\widehat{X}*}^\norm \omega_{\widehat{X}}^\norm = \bigoplus_{n \in \bb{Z}} (\pi_{\widehat{X}*}^\norm \omega_{\widehat{X}}^\norm)_{\std^n}, \quad
    (\pi_{\widehat{X}*}^\norm \omega_{\widehat{X}}^\norm)^\shear = \bigoplus_{n \in \bb{Z}} (\pi_{\widehat{X}*}^\norm \omega_{\widehat{X}}^\norm)_{\std^n}[n]. 
\]
\begin{defi} \label{defi:normalized_Lsheaf}
    We define the normalized $L$-sheaf associated to $\widehat{X}$ by 
    \[
        \cl{L}_{\widehat{X}}^\norm = (\pi_{\widehat{X}*}^\norm \omega_{\widehat{X}}^\norm)^\shear. 
    \]
\end{defi}

Now, we will explain the relation to the unnormalized $L$-sheaf $\cl{L}_{\widehat{X}}$ when the grading on $\widehat{X}$ comes from a central cocharacter. 

\begin{prop} \label{prop:Parnormz}
    Assume that the grading on $\widehat{X}$ is induced by a central cocharacter
    \[
        z_{\widehat{X}} \colon \bb{G}_{gr} \to \LG
    \]
    so that $a \cdot x = z_{\widehat{X}}(a)\,x$ for $a \in \bb{G}_{gr}$ and $x \in \widehat{X}$.  
    Let
    \[
        \tau_{\sqrt{\cyc}} \colon \Par_{\LG} \longrightarrow \Par_{\LG}, \qquad
        \varphi \longmapsto (z_{\widehat{X}} \circ \sqrt{\cyc}) \cdot \varphi
    \]
    be the translation by the unramified $L$-parameter $\sqrt{\cyc} \in \Par_{\bb{G}_{gr}}(\Lambda)$.  
    Then, there exists an isomorphism
    \[
        \iota \colon \Par_{\LG}^{\widehat{X}, \norm} \xrightarrow{\;\sim\;} \Par_{\LG}^{\widehat{X}}
    \]
    such that
    \[
        \pi_{\widehat{X}}^\norm \circ \iota^{-1} \;=\; \tau_{\sqrt{\cyc}}^{-1} \circ \pi_{\widehat{X}}.
    \]
\end{prop}

\begin{proof}
    We have an isomorphism
    \[
        [\widehat{X} / \LG \times \bb{G}_{gr}] \cong [\widehat{X} / \LG] \times [\ast / \bb{G}_{gr}] 
    \]
    given by an isomorphism $\LG \times \bb{G}_{gr} \cong \LG \times \bb{G}_{gr}$ sending $(g, a)$ to $(z_{\widehat{X}}(a)g, a)$. Applying the functor $\und{\Map}_{[\ast/Q]}([\ast/\WD_F], -)$, we obtain a
    Cartesian diagram
    \[
        \begin{tikzcd}
            \Par_{\LG \times \bb{G}_{gr}}^{\widehat{X}}
                \ar[r] \ar[d, "\pi_{\widehat{X},\LG \times \bb{G}_{gr}}"'] &
            \Par_{\LG}^{\widehat{X}}
                \ar[d, "\pi_{\widehat{X},\LG}"] \\
            \Par_{\LG \times \bb{G}_{gr}}
                \ar[r, "\id \times z_{\widehat{X}}"] &
            \Par_{\LG}.
        \end{tikzcd}
    %\tag{$\dagger$}
    \]

    %By \Cref{defi:normParLGX}, the normalized stack $\Par_{\LG}^{\widehat{X},\norm}$ is the base change of $\Par_{\LG \times \bb{G}_{gr}}^{\widehat{X}}$ along $\id \times \sqrt{\cyc} \colon \Par_{\LG} \longrightarrow \Par_{\LG \times \bb{G}_{gr}}.$ So 
    Pulling back the Cartesian square %$(\dagger)$ 
    along $\id \times \sqrt{\cyc}$, we obtain another
    Cartesian diagram
    \[
        \begin{tikzcd}
            \Par_{\LG}^{\widehat{X},\norm}
                \ar[r, "\iota"] \ar[d, "\pi_{\widehat{X}}^\norm"'] &
            \Par_{\LG}^{\widehat{X}}
                \ar[d, "\pi_{\widehat{X}}"] \\
            \Par_{\LG}
                \ar[r, "\tau_{\sqrt{\cyc}}"] &
            \Par_{\LG}. 
        \end{tikzcd}
    \]
    %where the lower horizontal arrow is the map
    %$\varphi \longmapsto (z_{\widehat{X}} \circ \sqrt{\cyc}) \cdot \varphi$, 
    %obtained by composing 
    %\[
    %    \Par_{\LG}
    %        \xrightarrow{\id \times \sqrt{\cyc}}
    %    \Par_{\LG \times \bb{G}_{gr}}
    %        \xrightarrow{\id \times z_{\widehat{X}}}
    %    \Par_{\LG}.
    %\]
   %which is an equivalence. Thus
   Since $\tau_{\sqrt{\cyc}}$ is an isomorphism, $\iota$ is a desired isomorphism. %and we get 
   % \[
   %     \pi_{\widehat{X}}^\norm \circ \iota ^{-1}= \tau_{\sqrt{\cyc}}^{-1} \circ \pi_{\widehat{X}},
   % \]
   % as required.
\end{proof}
Let $Z(\widehat{G})$ be the center of $\widehat{G}$ and let $Z = Z(\widehat{G})^Q$. Then, we have a weight decomposition
\[
    \IndCoh(\Par_{\LG}) = \prod_{\chi \in X^*(Z)} \IndCoh(\Par_{\LG})_{\chi}
\]
with respect to the $Z$-gerbe $\Par_{\LG} \to [Z^1(\WD_F, \LG) / (\widehat{G}/Z)]$. When the grading is given by a central cocharacter, $\cl{L}_{\widehat{X}}^\norm$ can be described in terms of this decomposition. 

\begin{cor} \label{cor:normLzX}
    In the setting of \Cref{prop:Parnormz}, $\cl{L}_{\widehat{X}}$ admits a grading and we have $\cl{L}_{\widehat{X}}^\norm \cong  \tau_{\sqrt{\cyc}}^* \cl{L}_{\widehat{X}}^\shear$. Moreover, the shearing can be described in terms of $Z$-weights as 
    \[
        \cl{L}_{\widehat{X}}^\shear = \bigoplus_{\chi \in X^*(Z)} \cl{L}_{\widehat{X}, \chi}[\langle \chi, z_{\widehat{X}} \rangle]. 
    \]
\end{cor}
\begin{proof}
    The first claim is immediate from \Cref{prop:Parnormz}. For the second claim, we will compare the weight decompositions of $\cl{L}_{\widehat{X}}$ by $Z$ and $\bb{G}_{gr}$. Since the smooth pullback is conservative, we may take the pullback along $[Z^1(\WD_F, \LG) / Z] \to \Par_{\LG}$. 

    First, the action of $\bb{G}_{gr}$ on $\Par_{\LG}^{\widehat{X}}$ induced by $\iota$ is the composition
    \[
        \Par_{\LG}^{\widehat{X}} \times \bb{G}_{gr} \to \und{\Map}_{[\ast/Q]}([\ast/\WD_{F}], [\widehat{X}/\LG] \times \bb{G}_{gr}) \to \Par_{\LG}^{\widehat{X}}. 
    \]
    Let $\widehat{X}_{Z^1}^{\WD_F} = \Par_{\LG}^{\widehat{X}} \times_{\Par_{\LG}} Z^1(\WD_F, \LG)$. Look at the map
    \[
        [\widehat{X}_{Z^1}^{\WD_F} / (Z \times \bb{G}_{gr})] \to Z^1(\WD_F, \LG) \times [\ast / Z] \times [\ast/ \bb{G}_{gr}]. 
    \]
    It is enough to show that the action of $Z \times \bb{G}_{gr}$ on $\widehat{X}_{Z^1}^{\WD_F}$ factors through $Z \times \bb{G}_{gr} \xrightarrow{\id \times z_{\widehat{X}}} Z$. 

    For each derived scheme $S$ and $\varphi \in Z^1(\WD_F, \LG)(S)$, we have
    \[
        \widehat{X}_{Z^1}^{\WD_F}(S) = \Sec([\widehat{X}/\WD_{F, S}] \to [S / \WD_{F, S}]). 
    \]
    For each $x \in \widehat{X}_{Z^1}^{\WD_F}(S)$, $g \in \LG(S)$ and $a \in \bb{G}_{gr}(S)$, the right action of $(g, a)$ is given by
    \[
        (\varphi, x)\cdot (g, a) = (\Ad(g^{-1})(\varphi), z_{\widehat{X}}(a)^{-1} g^{-1} x). 
    \]
    Then, the claim follows by restricting to the action of $Z \times \bb{G}_{gr}$. 
\end{proof}

\begin{rmk}
    Compare the normalization of $\cl{L}_{\widehat{X}}$ with the description of $\underline{\deg}$. They are compatible with the local class field theory: twisting representations by a character corresponds to translating $L$-parameters by a central parameter. 
\end{rmk}

%\begin{rmk}
%    It is a bit hard to see the similarity with \cite[Remark 11.5.2]{BZSV} aside the shearing: the half-cyclotomic translation $\tau_{\sqrt{\cyc}}^*$ seems to correspond to the half-epsilon factor $\varepsilon^{1/2}$, but the former is modelled on $z_{\widehat{X}}$, while the latter is modelled on $\eta_{\widehat{X}}$. We are not sure if our normalization is correct in general due to this incompatibility, but it works at least for the Iwasawa-Tate and Hecke periods. Another justification of this normalization is that we have functional equations in the vectorial case (see \Cref{sssec:FEL}). 
%\end{rmk}

\subsection{Functional equation} \label{sssec:FEL}

In the paradigm of \cite{BZSV}, the boundary theory is attached to the hyperspherical $G$-variety $T^*X$, not only to the spherical variety $X$. This philosophy is reminiscent of functional equations. In this section, let $\Lambda = \Qla$ for simplicity. 

%We say that a functional equation holds when periods or $L$-sheaves depend only on $T^*X$. 
\begin{defi}
    Let $X$ and $X'$ be smooth quasi-projective $G$-varieties, and let $\widehat{X}$ and $\widehat{X}'$ be derived $\LG$-schemes locally of finite presentation such that $\clas \widehat{X}$ and $\clas \widehat{X}'$ are $\widehat{G}$-quasi-projective. 
    \begin{enumerate}
        \item When $T^*X \cong T^* X'$ as symplectic $G$-varieties, we say that a functional equation holds for $X$ and $X'$ if $\cl{P}_X^\norm \cong \cl{P}_{X'}^\norm$. 
        \item When $T^*\widehat{X} \cong T^* \widehat{X}'$ as symplectic $\LG$-schemes, we say that a functional equation holds for $\widehat{X}$ and $\widehat{X}'$ if $\cl{L}_{\widehat{X}}^\norm \cong \cl{L}_{\widehat{X}'}^\norm$.
    \end{enumerate}
\end{defi}

Here, we verify functional equations for $L$-sheaves in the vectorial case. It essentially follows from the following functional equation proved in \cite{BZSV} and local Tate duality. 

\begin{lem}\textup{(\cite[Section 11.10.2, Lemma 11.10.3]{BZSV})} \label{lem:FEforL}
    Let $R$ be a local complete intersection $\Qla$-algebra. Consider perfect complexes over $R$
    \[
        \cl{E} = [Q \xrightarrow{d_0} Y \xrightarrow{d_1} P], \quad
        \cl{E}' = [P^\vee \xrightarrow{d_1^\vee} Y^\vee \xrightarrow{d_0^\vee} Q^\vee]
    \]
    with tor-amplitudes in $[0,2]$. Let $E = \Tot_R(\cl{E})$ and $E' = \Tot_R(\cl{E}')$. Let $\pi_E\colon E \to \Spec(R)$ and $\pi_{E'} \colon E' \to \Spec(R)$ be the natural projections. Then, we have a natural isomorphism
    \[
        (\pi_{E*} \omega_{E/R})^\shear \otimes \det(\cl{E}) \cong (\pi_{E'*} \omega_{E'/R})^\shear. 
    \]
    Here, the shearing is taken with respect to the scaling action on each side. This isomorphism changes the weight with respect to $\bb{G}_{gr}$ by $-1$. 
\end{lem}
\begin{proof}
    All the claims except for the last claim on weights are proved in loc. cit. The claim on weights follows from the explicit description in \cite[(11.51), (11.52)]{BZSV}. 
\end{proof}

%Let us briefly recall the shearing operation from \cite[Section 2.3.7]{BZSV}. Let $S$ be a derived scheme. For an object $M \in \IndCoh(S \times [\ast/\bb{G}_{gr}])$ with a weight decomposition $M = \bigoplus_{n \in \bb{Z}} M_{n}$, the shearing operation sends $M$ to
%\[
%    M^\shear = \bigoplus_{n \in \bb{Z}} M_{n}[n]. 
%\]
%As previously, we take a weight decomposition so that $M_n \in  \IndCoh(S \times [\ast/\bb{G}_{gr}])_{\std^n}$. 

\begin{cor} \label{cor:FEstd}
    Suppose that $\widehat{X} = V$ and $\widehat{X}' = V^\vee$ for some $V \in \Rep(\LG)$ and the gradings on them are scalar. Then, we have 
    \[
        \cl{L}_V^\norm \cong \cl{L}_{V^\vee}^\norm. 
    \]
    This isomorphism changes the weight with respect to $\bb{G}_{gr}$ by $-1$. 
\end{cor}
\begin{proof}
    Recall that each connected component of $Z^1(\WD_F, \LG)$ is affine and a local complete intersection over $\Qla$. Let $\Spec(R) \subset Z^1(\WD_F, \LG)$ be a connected component and let $(\varphi_R, N_R)$ be the associated Weil-Deligne $L$-parameter over $R$. Let $W = (V \otimes R)^{\varphi_R(I_F)}$ and let 
    \[
        \cl{E} = [W \xrightarrow{(q^{-1/2}\varphi_R(\Fr) - 1, N_R)} W \oplus W \xrightarrow{(N_R, q^{1/2}\varphi_R(\Fr) - 1)} W]
    \]
    %\[
    %    \cl{E}' = [W^\vee \xrightarrow{(q^{1/2}\rho_R(\Fr)^\vee - 1, N_R^\vee)} W^\vee \oplus W^\vee \xrightarrow{(N_R^\vee, q^{-1/2}\rho_R(\Fr)^\vee - 1)} W^\vee] 
    %\]
    be a perfect complex over $R$ with tor-amplitudes in $[0,2]$. Let $E = \Tot_R(\cl{E})$ and $E' = \Tot_R(\cl{E}^\vee[2])$. By \Cref{prop:explicitrel} and local Tate duality (see \Cref{cor:Tatedual}), we have
    \[
        E = \Par_{\LG}^{V, \norm} \times_{\Par_{\LG}} \Spec(R), \quad E' = \Par_{\LG}^{V^\vee, \norm} \times_{\Par_{\LG}} \Spec(R). 
    \]
    Since $\det(\cl{E})$ is trivial, it follows from \Cref{lem:FEforL} that
    \[
        (\pi_{E*} \omega_{E/R})^\shear \cong (\pi_{E'*} \omega_{E'/R})^\shear. 
    \]
    Then, we get an isomorphism of $\cl{L}^\norm_{V}$ and $\cl{L}^\norm_{V^\vee}$ over $Z^1(\WD_F, \LG)$. 
    By the naturality of \Cref{lem:FEforL}, it is $\widehat{G}$-equivariant, so %or similar isomorphisms can be constructed functorially on the \v{C}ech nerve of $Z^1(W_F, \GL_d) \to \Par_{\GL_d}$, '
    the isomorphism descends to $\Par_{\LG}$. 
\end{proof}

We expect that functional equations for periods also hold in the vectorial case. 

\begin{conj}
    For every $V \in \Rep(G)$, we have $\cl{P}_V^\norm \cong \cl{P}_{V^\vee}^\norm$. 
\end{conj}

For this, one needs to refine Fourier transforms on Banach-Colmez spaces developed in \cite{ALB25}. Note that it is enough to work with the universal case $G = \GL(V)$ by the projection formula. In \Cref{prop:FEIwTate}, we verify the one-dimensional case by hand.

\subsection{Normalized period conjecture} \label{ssec:NPC}

In this section, we state the normalized period conjecture in the categorical local Langlands correspondence (see \cite[Section 12.1]{BZSV} in the geometric setting) and its implication to distinction problems. Note that the statements of conjectures are conditional on the definition of dual pairs $(G, X) \leftrightarrow (\LG, \widehat{X})$. 

For simplicity, let us assume $\Lambda = \Qla$ and formulate with the $\cl{D}^\oc$-formalism instead of lisse-\'{e}tale sheaves. 

\begin{conj} \label{conj:normP=L}
    Let $G$ be a connected quasi-split reductive group over $F$ and fix a Whittaker datum of $G$. Suppose that there exists a categorical equivalence 
    \[
        \bb{L}_G \colon \cl{D}^\oc(\Bun_G, \Qla) \cong \IndCoh(\Par_{\LG})
    \]
    as formulated in \cite[Conjecture X.1.4]{FS24}. For every (conjecturally defined) dual pair $(G, X) \leftrightarrow (\LG, \widehat{X})$, we have 
    \[
        \bb{L}_G(\cl{P}_X^\norm) \cong \cl{L}_{\widehat{X}}^\norm. 
    \]
\end{conj}

%\begin{rmk}
    %There is not a good definition of dual pairs $(G, X) \leftrightarrow (\LG, \widehat{X})$ yet. One can find a partial list of dual pairs in \cite[Table 1.5.1]{BZSV}. One usual condition is that $X$ and $\widehat{X}$ are spherical. 
    %As in the geometric setting, we do not have a good definition of dual pairs $(G, X) \leftrightarrow (\LG, \widehat{X})$ yet. One can still find a partial list of dual pairs in \cite[Table 1.5.1]{BZSV}. One usual condition is that $X$ (resp.\ $\widehat{X}$) is spherical as a $G$-variety (resp.\ $\LG$-variety). 
%\end{rmk}

We explain a representation-theoretic consequence of the normalized period conjecture. First, we briefly recall the geometric properties of $L$-parameters necessary for our formulation, following \cite{Han24} and \cite{HL25}. Let $\Par_{\LG}^\sss$ be the coarse moduli space of $L$-parameters, which parametrizes $\widehat{G}$-conjugacy classes of semisimple $L$-parameters. The natural morphism
\[
    \Par_{\LG} \to \Par_{\LG}^\sss, \quad \varphi \mapsto \varphi^\sss
\]
sends an $L$-parameter to its semisimplification.

\begin{defi}(\cite[Definition 2.1.2]{Han24}) \label{def:generous}
    A semisimple $L$-parameter $\varphi \in \Par_{\LG}(\Qla)$ is \textit{of Langlands-Shahidi type} (or \textit{generous}) if we have an equivalence 
    \[
        [\ast / S_\varphi] \cong \Par_{\LG} \times_{\Par_{\LG}^\sss} \{ \varphi \}.
    \]
    Here, $S_\varphi \subset \widehat{G}$ is the centralizer of $\varphi$. 
\end{defi}
\begin{rmk}
    As explained in \cite[Remark 6.5]{HL25}, there is a cohomological criterion for being of Langlands-Shahidi type (see \cite[Definition 6.2]{HL25}). The terminology is also adopted from loc. cit. 
    %In \cite[Definition 6.2]{HL25}, these parameters are termed \textit{Langlands-Shahidi type}. They are characterized cohomologically: if $(M, \varphi_M)$ is the cuspidal support of $\varphi$ and $r$ denotes the adjoint action of $\LM$ on $\Lie(\widehat{N})$ (where $P=MN$ is a parabolic subgroup), genericity is equivalent to the vanishing of the Galois cohomology groups  $R\Gamma(W_F, r \circ \varphi_M)$ and $R\Gamma(W_F, r\circ \varphi^\vee_M)$. 
\end{rmk}

Let $\varphi \colon W_F \to \LG(\Qla)$ be a semisimple $L$-parameter of Langlands-Shahidi type. By \cite[Proposition 2.1.7]{Han24}, the map $[\ast / S_\varphi] \hookrightarrow \Par_{\LG}$ is a regular closed immersion, so 
\[
    i_\varphi \colon \Spec(\Qla) \to [\ast/S_\varphi] \hookrightarrow \Par_{\LG}
\]
is schematic and eventually coconnective. Now, $i_{\varphi*} \Qla$ should match a Hecke eigensheaf $\cl{F}_\varphi$ under $\bb{L}_G$ and it is expected to have the following form. For each $b \in B(G)$, let $i^{b, \ren}_!$ (resp.\ $i^{b, \ren}_*$) be the normalization of $i^b_!$ (resp.\ $i^b_*$) as in \cite[Definition 1.1.1]{Han24}.

\begin{conj}\textup{(\cite[Conjecture 2.1.8, 2.1.9]{Han24})}\label{conj:Heckeeigen}
    We have a decomposition 
    \[
        \cl{F}_\varphi = \bigoplus_{b \in B(G), \pi \in \Pi_\varphi(G_b)} i^{b, \ren}_{!} \pi^{\oplus m(\pi)}
    \]
    with finite nonzero multiplicity for each $\pi$, and each term satisfies $ i^{b, \ren}_{!} \pi \cong  i^{b, \ren}_{*} \pi$. 
\end{conj}

%Here, we write $\Pi_\varphi(G_b)$ for the (finite) fiber over $\varphi$ of the map \(\Pi(G_b) \to \Phi(G)\) in the \(B(G)\)-parametrized local Langlands correspondence; equivalently, \(\Pi_\varphi(G_b)\) is the finite set of isomorphism classes of irreducible smooth representations \(\pi\) of \(G_b(F)\) whose Fargues-Scholze parameter equals \(\varphi\). 
%\begin{exa}When \(G = \GL_n\) and \(b\) is basic, \(\Pi_\varphi(G_b)\) is a singleton \(\{\pi_\varphi\}\), where \(\pi_\varphi\) is the irreducible representation attached to \(\varphi\) by the classical local Langlands correspondence for \(\GL_n\). 
%\end{exa}
Here, $\Pi_\varphi(G_b)$ denotes the $L$-packet for the $B(G)$-parametrization of the local Langlands correspondence (see \cite[Theorem 3.8]{BMO25}). Under this expectation, the normalized period conjecture has the following consequence on distinction problems. 

\begin{prop} \label{conj:Xdistrep}
    Suppose that \Cref{conj:normP=L} holds for a dual pair $(G, X) \leftrightarrow (\LG, \widehat{X})$ such that $X$ is unimodular. If a smooth irreducible $G(F)$-representation $\pi$ is $X$-distinguished and its Fargues-Scholze parameter $\varphi_\pi^{\FS}$ is of Langlands-Shahidi type and satisfies \Cref{conj:Heckeeigen}, $\varphi_\pi^{\FS}$ lies in the image of $\pi_{\widehat{X}}^\norm$. 
\end{prop}

\begin{proof}
    Let $\varphi = \varphi_\pi^\FS$. Then, \Cref{conj:normP=L} implies 
    \begin{equation}
        \Hom(\cl{P}_X^\norm, \cl{F}_\varphi) \cong \Hom(\cl{L}_{\widehat{X}}^\norm, i_{\varphi*} \Qla). \label{eq:Homtoeigen}
    \end{equation} 
    Since $i^{1, \ren}_* = i^{1}_*$, it follows from \Cref{conj:Heckeeigen} that 
    \[
        \Hom(\cl{P}_X^\norm\vert_{\Bun_G^1}, \pi) \cong \Hom(\cl{P}_X^\norm, i^1_*\pi) \cong \Hom(\cl{P}_X^\norm, i^{1, \ren}_!\pi)
    \]
    is a direct summand of $\Hom(\cl{P}_X^\norm, \cl{F}_\varphi)$. Since $X$ is unimodular, we have $\cl{P}_X^\norm = \cl{P}_X$, so it follows from \Cref{cor:computetriv} that 
    \[
        \Hom_{G(F)}(C_c^\infty(X(F), \Qla), \pi)\hookrightarrow \Hom(\cl{P}_X^\norm, \cl{F}_\varphi)
    \]
    is a direct summand for each $\pi \in \Pi_{\varphi}(G)$. In particular, if there is an $X$-distinguished element $\pi \in \Pi_{\varphi}(G)$, $\Hom(\cl{P}_X^\norm, \cl{F}_\varphi)$ is nonzero. 

    On the other hand, let $\widehat{X}_\varphi$ be the fiber of $\pi_{\widehat{X}}^\norm$ at $\varphi$. 
    \begin{center}
        \begin{tikzcd}[column sep=large]
            \widehat{X}_\varphi \ar[r, "q_\varphi"] \ar[d, "p_\varphi"'] & \Par_{\LG}^{\widehat{X}, \norm} \ar[d, "\pi_{\widehat{X}}^\norm"] \\
            \Spec(\Qla) \ar[r, "i_\varphi"] & \Par_{\LG}. 
        \end{tikzcd}
    \end{center}
    Since $i_\varphi$ is eventually coconnective, we have
    \[
        \Hom(\cl{L}_{\widehat{X}}^\norm, i_{\varphi*} \Qla) \cong \Hom((i_{\varphi}^*\pi_{\widehat{X} *}^\norm \omega_{\widehat{X}}^\norm)^\shear, \Qla) \cong \Hom((p_{\varphi *} q_{\varphi}^* \omega_{\widehat{X}}^\norm)^{\shear}, \Qla). 
    \]
    by the base change in \cite[Proposition 3.2.2]{GR17I}. Thus, $\widehat{X}_{\varphi}$ is nonempty if $\Hom(\cl{L}_{\widehat{X}}^\norm, i_{\varphi*} \Qla)$ is nonzero.
    By \eqref{eq:Homtoeigen}, $\widehat{X}_{\varphi}$ is nonempty if there is an $X$-distinguished element $\pi \in \Pi_{\varphi}(G)$. 
    %Now, we have a condition for the non-vanishing of each side of \eqref{eq:Homtoeigen}. Combining with the description of the image of $\Par_{\LG}^{\widehat{X}} \to \Par_{\LG}$ in \Cref{cor:supprelX}, we may deduce the following representation-theoretic conjecture from the normalized period conjecture. 
\end{proof}

\begin{rmk}
    As suggested by Linus Hamann, the Langlands-Shahidi condition could be removed, and a similar analysis might work even for $A$-parameters by replacing $\cl{F}_\varphi$ with sheared eigensheaves developed in the unpublished work of Bertoloni-Meli and Koshikawa (see \cite{Kos24}). It would clarify the relation to the local conjecture of Sakellaridis-Venkatesh \cite[Conjecture 16.2.2]{SV17}. However, we need extra arguments even for discrete $L$-parameters due to the failure of \Cref{conj:Heckeeigen}. 
    %However, even for a discrete $L$-parameter $\varphi$, there is a technical subtlety that $i^1_* \pi$ may not be a direct summand of $\cl{F}_\varphi$, and $i^1_! C_c^\infty(X(F), \Qla)$ may not be a direct summand of $\cl{P}_X$. Even if $\pi$ is $X$-distinguished, we need extra arguments to deduce that $\cl{F}_{\varphi_\pi}$ is $X$-distinguished. 
\end{rmk}

Suppose that $X$ is a spherical $G$-variety. We expect that under suitable conditions, $X$ should admit an $L$-group $\LGX \subset \LG$, and $\widehat{X}$ should take the form 
\[
    \widehat{X} = \LG \times^{\LGX} V
\]
for some $V \in \Rep(\LGX)$ with a grading induced from $V$, as in \cite[Section 4.1.1]{BZSV}. 

Then, $\pi_{\widehat{X}}$ factors as
\[
    \Par_{\LG \times \bb{G}_{gr}}^{\widehat{X}} = \Par_{\LGX \times \bb{G}_{gr}}^V \to \Par_{\LGX \times \bb{G}_{gr}} \to \Par_{\LG \times \bb{G}_{gr}}. 
\]
Thus, we expect that $\varphi_\pi^\FS$ should factor through a conjugate of $\LGX$ if $\pi$ is $X$-distinguished and $\varphi_\pi^\FS$ is of Langlands-Shahidi type. 

\subsection{Example: Galois symmetric varieties} \label{sssec:Galoissym}

Let $E$ be a separable quadratic extension of $F$. Let $H$ be a connected quasi-split reductive group over $F$ and let $G = \Res_{E/F}(H_E)$ be the Weil restriction. %Take $Q = \Gal(E / F)$ so that $\LG = (\widehat{H} \times \widehat{H}) \rtimes Q$. Let $\Gal(E/F) = \{1, \sigma\}$.
%We have a short exact sequence 
%\[ 
%1 \ra W_E\ra W_F \ra \Gal(E/F) \ra 1 
%\]
A choice $\sigma \in W_F \backslash W_E$ provides an identification
\[
    \LG \cong (\widehat{H} \times \widehat{H}) \rtimes Q, \quad
    \tau \cdot (h_1,h_2) =  \left\{ 
    \begin{alignedat}{4}
        &(\tau h_1, \sigma^{-1} \tau \sigma \cdot h_2) &\quad& (\tau \in W_E) \\
        &(\sigma^2 h_2, h_1) &\quad& (\tau = \sigma). 
    \end{alignedat}
    \right. 
\]

We will consider the case $X = G / H$. First, we briefly recall its $L$-group given in \cite[2.2.11]{BP25}. Let $c \in \Aut(\widehat{H})$ be the \textit{duality involution}, an inner twist of the Chevalley involution sending a pinning to its opposite. 

\begin{exa}
    When $H=\GL_n$, $G(F) = \GL_n(E)$ and $X(F) = \GL_n(E)/\GL_n(F)$. The duality involution $c$ on $\widehat{H} = \GL_n$ is given by 
    \[ 
        g \mapsto w {}^tg^{-1}w^{-1}
    \]
    where $w$ is the longest Weyl element (anti-diagonal permutation matrix). 
\end{exa}

The involution $c$ commutes with the action of $Q$, and 
\[
    {}^L \iota_X \colon \LGX = \widehat{H} \rtimes Q \xhookrightarrow{(1, c \sigma^{-1}) \times \id_Q} (\widehat{H} \times \widehat{H}) \rtimes Q = \LG
\]
is the $L$-group of $X$. %Here, $\LGX$ is defined so that $\sigma$ acts on $\widehat{H}$ by $c$. 
As in the group case of \cite[Table 1.5.1]{BZSV}, 
\[
    (G, X = G / H) \leftrightarrow (\LG, \widehat{X} = \LG / \LGX)
\]
is expected to be a dual pair. 

\begin{rmk}
    Here, ${}^L \iota_X$ is a Galois equivariant extension of the morphism
    \[ 
        \iota_X\colon \widehat{G}_X  \to \widehat{G}
    \]
    attached to a spherical variety $X$ (see \cite{SV17} and \cite{KS17}). At present, a general construction of ${}^L \iota_X$ does not seem to be available (see \cite[2.2.10]{BP25}).
    %The $\SL_2$-factor is trivial in our setting, but it is nontrivial for general spherical varieties. 
    %The $\SL_2$ factor is necessary for general spherical varieties. The $\SL_2$ morphism, $\iota_X|_{\SL_{2,\bb{C}}}$ is the \textit{principal homomorphism} into the dual Levi $\widehat{L}(X)$ of the Levi subgroup $L(X)$ of the parabolic type $P(X)$ of the spherical variety.  $P(X)$ is the stabilizer of the dense open orbit of our fixed $B$. In which case, the homomorphism is trivial when $L(X)$ is a torus. A large collection of $X$ satisfies this condition, cf. \cite[Chapter 6]{SV17}. 
\end{rmk}

Let $\pi$ be a smooth irreducible $H(E)$-representation with a Fargues-Scholze parameter $\varphi \colon W_E \to \LH(\Qla)$, point in $\Par_{\LH_E}(\bar{\Qla})$. The corresponding point in $\Par_{\LG}(\bar{\Qla})$ is 
\[
    \Ind(\varphi) \colon W_F \to \LG(\Qla), \quad 
    \Ind(\varphi)(\tau) =  \left\{ 
    \begin{alignedat}{3}
        &(\varphi(\tau), \varphi(\sigma^{-1} \tau \sigma)) &\quad& (\tau \in W_E) \\
        &(\varphi(\tau \sigma), \varphi(\sigma^{-1} \tau)) &\quad& (\tau \in W_E \backslash W_F). 
    \end{alignedat}
    \right. 
\]
%the Shapiro induction of $\varphi$. 
Then, \Cref{conj:Xdistrep} predicts that if $\pi$ is $H(F)$-distinguished and $\varphi$ is of Langlands-Shahidi type, $\Ind(\varphi)$ factors through ${}^L \iota_X$, and we have a conjugacy relation
\[
    \sigma \circ \varphi \circ \Ad(\sigma^{-1}) \sim c \circ \varphi. 
\]
Indeed, this is part of Prasad's conjecture on $L$-parameters; \cite[Conjecture 2]{Pra15}. To study its numerical part, we assume the following. %Let $X^*(H)$ be the character group of $H$ defined over $F$. 
%Recall an $L$-parameter $\varphi:W_F \ra \LH$ is  \textit{discrete} if $C_{\widehat{H}}(\varphi)/Z(\widehat{H})^{\Gamma_F}$ is finite, cf. \cite[Chatper 5.4]{taibi}. We require
%\question{ why not just say $\varphi$ is a discrete and generous parameter?}
\begin{ass} \label{ass:Galoissym}
    Let $\varphi$ be a supercuspidal $L$-parameter and the center of $H$ is anisotropic.
\end{ass}

Note that supercuspidal $L$-parameters are called elliptic in \cite[Definition X.2.1]{FS24}. This condition is preserved under the identification $\Par_{\LG} \cong \Par_{{}^LH_E}$ in \cite[Proposition IX.6.3]{FS24}.

\begin{exa}
    The second condition is satisfied when $H = \mrm{SL}_n, \Sp_{2n}, \mrm{U}(n)$ for $n \geq 1$. Here, $\mrm{U}(n)$ is a unitary group with $n$-variables associated to $E / F$. 
    %The second condition of \Cref{ass:Galoissym} is satisfied when $H = \mrm{SL}_n, \Sp_{2n}, \text{ and } \mrm{U}(n)$ for $n \geq 1$. Here, $\mrm{U}(n)$ is a unitary group with $n$-variables associated to $E / F$, with $\overline{( -)}$ denote Galois conjugate. Indeed,   $Z(\SL_n) \cong \mu_n$, $Z(\Sp_{2n}) \cong \crbr{\pm I_{2n}}$ and $Z(U(n)) \cong R_{E/F}^{1}\bb{G}_m$, the norm one torus.  
\end{exa}
%\begin{lem}
%    The second assumption is satisfied when $H$ is simply connected or a special group in the sense of Serre. 
%\end{lem}
%\begin{proof}
%    When $H$ is simply conencted, $H^1(F, H) = 0$ by Kneser's theorem, so the claim follows. When $H$ is a special group, the $H$-bundle $G \to X$ is Zariski locally split, so $G(F) \to X(F)$ is surjective. 
%\end{proof}

%Then, $H(F)$ has a compact center, so every smooth $H(F)$-representation has a decomposition with respect to central characters. Moreover, $\pi$ is supercuspidal since $\varphi$ is supercuspidal. Thus, $\pi$ is injective and projective in the category of smooth $H(F)$-representations. 

%As notations \Cref{ssec:Notation}, here $\Hom$ denotes an (underived) homomorphism set, we now regard it enriched as a $\Qla$-vector space. 
%Let $\Gamma:=\Gal(\bar{F}/F)$ be the Galois group of $F$ and recall the Kottwitz isomorphism $H^1(F,H)\subset B(H)_\bas \cong \pi_1(H)_\Gamma$. 
% Under \Cref{ass:Galoissym}, $\pi_1(H)_\Gamma$ is a finite abelian group. 

By \Cref{cor:XFdecomposition}, we have
\begin{align*}
      \Hom_{G(F)}(C_c^\infty(X(F), \Qla), \pi) &\cong \bigoplus_{b\in \pi_X^{-1}(1)} \Hom_{G(F)}(\pi^\vee, \Ind_{H_b(F)}^{G(F)} \Qla)  \\ 
      & \cong \bigoplus_{b\in \pi_X^{-1}(1)} \Hom_{H_b(F)}(\pi^\vee\vert_{H_b(F)}, \Qla). 
\end{align*} 
Note that $\pi_X^{-1}(1) \subset H^1(F, H)$, so it is a finite set. Let $\Pi_\varphi(G) = \Pi_{\Ind(\varphi)}(G)$. For each $b \in \pi_X^{-1}(1)$, let
\[
    m_b(\Pi_\varphi) = \sum_{\pi \in \Pi_\varphi(G)} m(\pi) \dim \Hom_{H_b(F)}(\pi^\vee\vert_{H_b(F)}, \Qla). 
\]
Now, \Cref{conj:Heckeeigen} implies 
\[
    \cl{F}_{\Ind(\varphi)}\vert_{\Bun_G^{(1)}} = \bigoplus_{\pi \in \Pi_\varphi(G)} i^{1}_* \pi^{\oplus m(\pi)}
\]
since $\varphi$ is supercuspidal and $\cl{F}_{\Ind(\varphi)}$ is only supported on $B(G)_\bas$. Thus, we have
\begin{equation}
    \sum_{b \in \pi_X^{-1}(1)} m_b(\Pi_\varphi) = \dim \Hom(\cl{P}_X^\norm, \cl{F}_\varphi\vert_{\Bun_G^{(1)}}). \label{eq:multP}
\end{equation}
as $X = G / H$ is affine and unimodular. Now, we will compute the right hand side by passing to the $\cl{B}$-side. Let $Q_E \subset Q$ be the image of $W_E$.

\begin{lem} \label{lem:supercuspcoh}
    Let $\widehat{\mfr{h}} = \Lie(\widehat{H})$ and let $\widehat{\mfr{h}}_\varphi$ denote the adjoint representation on $\mfr{h}$ composed with $\varphi$. Let $\mfr{z}_E$ be the Lie algebra of $Z(\widehat{H})^{Q_E}$. When $\varphi$ is supercuspidal, we have 
    \[
        R\Gamma(\WD_{E}, \widehat{\mfr{h}}_\varphi) = \mfr{z}_E \oplus \mfr{z}_E[-1]. 
    \]
\end{lem}
\begin{proof}
    It is well-known to experts (see the footnote of \cite[Section X.2]{FS24}). For completeness, we provide a proof here. Since $\widehat{H} \to \widehat{H} / \widehat{H}^\der \times \widehat{H}^\ad$ induces an isomorphism on Lie algebras, we may assume that $\widehat{H}$ is commutative or $\widehat{H}$ is semisimple. The claim is immediate when $\widehat{H}$ is commutative. We will prove
    \[
        R\Gamma(\WD_{E}, \widehat{\mfr{h}}_\varphi) = 0
    \]
    when $\widehat{H}$ is semisimple. In the following, ${}^LH$, $\Fr$ and $q$ denote the counterparts for $E$. 
    
    Since $\varphi$ is supercuspidal, $S_\varphi^\circ = \{1\}$ as $\varphi$ is discrete and $\widehat{H}$ is semisimple (hence finite center), so 
    \[
        H^0(\WD_{E}, \widehat{\mfr{h}}_\varphi) = \widehat{\mfr{h}}^{\varphi(W_E) = \id} = \Lie(S_\varphi) = 0. 
    \]
    On the other hand, by local Tate duality (see \Cref{cor:Tatedual}), we have
    \begin{equation}
        H^2(\WD_{E}, \widehat{\mfr{h}}_\varphi) = H^0(\WD_{E}, \widehat{\mfr{h}}_\varphi^*(1)). \label{eq:H2Gal}
    \end{equation}
    Let $C_{\varphi(I_E)} \subset \LH$ be the centralizer of $\varphi(I_E)$. The Killing form induces an $\LH$-equivariant isomorphism $\widehat{\mfr{h}} \cong \widehat{\mfr{h}}^*$. Then, an element $\xi$ of \eqref{eq:H2Gal} is regarded as $\xi \in \Lie(C_{\varphi(I_F)})$ and satisfies $\Ad(\varphi(\Fr))(\xi) = q \xi$. Let $C_{\varphi} \subset \LH$ be the subgroup generated by $C_{\varphi(I_F)}$ and $\varphi(\Fr)$. By \cite[Theorem 2.1]{PY02}, $C_{\varphi(I_F)}^\circ = C_{\varphi}^\circ$ is reductive, so we may apply \cite[Lemma 3.5]{Tak24} to $C_\varphi$ by setting $s = \varphi(\Fr)$ and $u_1 = \xi$ to get a cocharacter $\mu \colon \bb{G}_m \to C_{\varphi(I_F)}^\circ$. Then, $\mu$ commutes with $\varphi(W_E)$, so it factors through $S_\varphi$. Then, $\mu$ is trivial since $S_\varphi^\circ = \{1\}$, so we get $\xi = 0$. 

    It remains to see $ H^1(\WD_{E}, \widehat{\mfr{h}}_\varphi) = 0$. It follows from the fact that the Euler characteristic of $ R\Gamma(\WD_{E}, \widehat{\mfr{h}}_\varphi)$ is zero (see \Cref{cor:Tatedual}). 
\end{proof}

\begin{prop} \label{prop:Xvarphifinet}
    Under \Cref{ass:Galoissym}, $\widehat{X}_\varphi$ is finite \'{e}tale. 
\end{prop}
\begin{proof}
    Since $\clas \widehat{X}_\varphi$ is of finite type by \Cref{prop:ParLGXclassical}, it is enough to show that $\widehat{X}_\varphi$ is \'{e}tale over $\Spec(\Qla)$. Let $L_{\widehat{X}_\varphi}$ be the cotangent complex. For each $\psi \in \widehat{X}_\varphi(\Qla)$, we will show that $L_{\widehat{X}_\varphi}$ vanishes at $\psi$. Here, $\psi$ corresponds to a lift $\WD_{F, \Qla} \to \LGX$ of $\varphi$. 

    Since $\Par_\LG^{\widehat{X}} = \Par_{\LGX}$, the cotangent complex of $\Par_{\LGX}$ at $\psi$ equals
    \[
        R\Gamma(\WD_{F}, \widehat{\mfr{h}}^*_{\psi}(1)) [1]
    \]
    as in the proof of \Cref{prop:ParLGcotangent}. Since $L_{\widehat{X}_\varphi}\vert_\psi$ is the relative cotangent complex of $\psi$ over $\Ind(\varphi)$, $L_{\widehat{X}_\varphi}\vert_\psi[-2]$ equals the fiber of 
    \[
        R\Gamma(\WD_{E}, \widehat{\mfr{h}}^*_{\varphi}(1)) = 
        R\Gamma(\WD_{F}, (\widehat{\mfr{h}}^* \oplus \widehat{\mfr{h}}^*)_{\Ind(\varphi)}(1)) \to
        R\Gamma(\WD_{F}, \widehat{\mfr{h}}^*_{\psi}(1)). 
    \]
    By local Tate duality (see \Cref{cor:Tatedual}), $L_{\widehat{X}_\varphi}^\vee\vert_\psi$ equals the fiber of
    \[
        R\Gamma(\WD_{F}, \widehat{\mfr{h}}_{\psi}) \to 
        R\Gamma(\WD_{E}, \widehat{\mfr{h}}_{\varphi}).  
    \]
    By \Cref{lem:supercuspcoh}, the target equals $\mfr{z}_E \oplus \mfr{z}_E [-1]$, and the domain equals
    \[
        R\Gamma(\WD_{F}, \widehat{\mfr{h}}_{\psi}) = R\Gamma(\bb{Z} / 2, R\Gamma(\WD_{E}, \widehat{\mfr{h}}_{\varphi})) = R\Gamma(\bb{Z} / 2, \mfr{z}_E \oplus \mfr{z}_E [-1]). 
    \]
    Let $X^*(H_E)$ denote the character group of $H_E$. Then, the action of $\sigma$ on $X^*(H_E)$ is the multiplication by $-1$ since $X^*(H) = 0$. Since the action of $c$ on $\mfr{z}_E$ is also the multiplication by $-1$, the action of $\sigma$ on $\mfr{z}_E \subset \widehat{\mfr{h}}_\psi$ is trivial. Thus, we get $L_{\widehat{X}_\varphi}^\vee\vert_\psi = 0$. 
\end{proof}

For each lift $\psi \colon W_F \to \LGX$ of $\varphi$, let $S_\psi \subset \widehat{H}$ be the centralizer of $\psi$. Then, $S_\psi \subset S_\varphi$ and let 
\[
    j_\psi \colon [\ast / S_\psi] \to [\ast / S_\varphi]. 
\]
Then, we have a diagram
\begin{center}
    \begin{tikzcd}
        \Par_{\LGX} \ar[r] & \Par_{\LG} \ar[r] & \Par_{\LG}^\sss \\
        \coprod\limits_{\psi / \varphi} \lbrack \ast / S_\psi \rbrack \ar[u, hook] \ar[r, "(j_\psi)_{\psi / \varphi}"] & \lbrack \ast / S_\varphi \rbrack \ar[u, "z_\varphi", hook] \ar[r] & \Spec(\Qla). \ar[u, "\varphi", hook]
    \end{tikzcd}
\end{center}
Here, the squares are Cartesian. By \Cref{prop:Xvarphifinet}, each $j_\psi$ is finite \'{e}tale and the number of lifts $\psi$ is finite. Since $\varphi$ is supercuspidal, 
\[
    j_\varphi \colon [\ast / Z(\widehat{H})^{Q_E}] \to [\ast / S_\varphi]
\]
is also finite \'{e}tale. Let $Z=Z(\widehat{H})^{Q_E}$. The trivial $Z$-weight part of $i_{\varphi * } \Qla$ equals $z_{\varphi *} j_{\varphi *} \cl{O}$. Here, $\cl{O}$ denotes the structure sheaf of $[*/Z]$.

Due to the Hecke compatibility, $z_{\varphi *} j_{\varphi *} \cl{O}$ should match $\cl{F}_\varphi\vert_{\Bun_G^{(1)}}$ under $\bb{L}_G$. Then, since the grading on $\widehat{X}$ is trivial, the right hand side of \eqref{eq:multP} is computed by the vector space
\[
    \Hom(\cl{L}_{\widehat{X}}^\norm, z_{\varphi *} j_{\varphi *} \cl{O}) \cong 
    \Hom(z_\varphi^* \cl{L}_{\widehat{X}}, j_{\varphi *} \cl{O}) \cong \bigoplus_{\psi / \varphi} \Hom(j_{\psi *}\cl{O}, j_{\varphi *} \cl{O}).
\]
The second isomorphism follows from the base change and $j_\psi^! \cl{O} \cong \cl{O}$ since $j_\psi$ is \'{e}tale. Let $Z^\circ = S_\varphi^\circ$. Recall the weight decomposition (see \Cref{ssec:weight_decomposition})
\[
    \QCoh([\ast / S_\varphi]) \cong \prod_{\chi \in X^*(Z^\circ)} \QCoh([\ast / S_\varphi])_\chi, \quad
    \QCoh([\ast / S_\varphi])_\triv \cong \QCoh([\ast / \pi_0(S_\varphi)])
\]
where $\pi_0(S_\varphi) = S_\varphi /S^\circ_\varphi$. Since $j_\psi$ is finite \'{e}tale, $S_\psi^\circ = Z^\circ$, so we have 
\[
    j_{\psi *}\cl{O} \cong \Qla[\pi_0(S_\varphi) / \pi_0(S_\psi)] ,\quad 
    j_{\varphi *}\cl{O} \cong \Qla[\pi_0(S_\varphi) / \pi_0(Z)]
\] 
via the identification $\QCoh([\ast / S_\varphi])_\triv \cong \QCoh([\ast / \pi_0(S_\varphi)])$. Thus, we get the following numerical equation. 

\begin{prop} \label{prop:multformula}
    Suppose that \Cref{conj:normP=L} holds for the dual pair 
    \[
        (G, X = G / H) \leftrightarrow (\LG, \widehat{X} = \LG / \LGX)
    \]
    and \Cref{conj:Heckeeigen} holds for $\varphi$. When \Cref{ass:Galoissym} holds, we have
    \[
        \sum_{b \in \pi_X^{-1}(1)} m_b(\Pi_\varphi) = \sum_{\psi / \varphi} \#(\pi_0(Z) \backslash \pi_0(S_\varphi) / \pi_0(S_\psi)). 
    \]
    Here, $\psi \colon W_F \to \LGX$ runs through all lifts of $\varphi$. 
\end{prop}

\begin{proof}
    By the above description, we have
    \begin{align*}
        \Hom(j_{\psi *}\cl{O}, j_{\varphi *} \cl{O}) &\cong \Hom_{\pi_0(S_\varphi)}(\Qla[\pi_0(S_\varphi) / \pi_0(S_\psi)], \Qla[\pi_0(S_\varphi) / \pi_0(Z)]) \\
        &\cong \Qla[\pi_0(S_\psi) \backslash \pi_0(S_\varphi) / \pi_0(Z)].     
    \end{align*}
    Thus, the claim follows from the combination of the above arguments. 
\end{proof}
\begin{rmk}
    Compare with the description of \cite[Conjecture 1]{BP18} when the character $\chi$ in loc. cit. is trivial. Then, $\pi_X^{-1}(1)$ equals the kernel of $H^1(F, H) \to H^1(F, G)$. It seems to be suggested there that $m_b(\Pi_\varphi)$ is independent of $b$. 
\end{rmk}

We can also predict the multiplicity of each $\pi \in \Pi_{\varphi}(G)$. %Let $\rho_\pi \in \Irr(\pi_0(S_\varphi) / \pi_0(Z))$ be the irreducible representation that should correspond to $i^1_!\pi$ under $\bb{L}_G$. 

\begin{prop} \label{prop:multpi}
    Suppose that \Cref{conj:normP=L} holds for the dual pair 
    \[
        (G, X = G / H) \leftrightarrow (\LG, \widehat{X} = \LG / \LGX)
    \]
    and $\rho_\pi \in \Irr(\pi_0(S_\varphi) / \pi_0(Z))$ satisfies $\bb{L}_G(i^1_! \pi) \cong z_{\varphi*} \cl{O}_{\rho_\pi}$. When \Cref{ass:Galoissym} holds, 
    \[
        \dim \Hom_{G(F)}(C_c^\infty(X(F), \Qla), \pi) = \sum_{\psi / \varphi} \dim \rho_\pi^{\pi_0(S_\psi) = \id}. 
    \]
    Moreover, for every $i > 0$, we have
    \[
        \Ext^i_{G(F)}(C_c^\infty(X(F), \Qla), \pi) = 0. 
    \]
\end{prop}
\begin{proof}
   Both claims follow from the same computation as \Cref{prop:multformula}. Note that the second claim follows from the vanishing of the higher cohomology groups of
    \[
        \Map(\cl{L}_{\widehat{X}}^\norm, z_{\varphi *} \cl{O}_{\rho_\pi}) \cong \bigoplus_{\psi / \varphi} \Map(j_{\psi *} \cl{O}, \cl{O}_{\rho_\pi})
    \]  
    since $j_{\psi *} \cl{O}$ and $\cl{O}_{\rho_\pi}$ lie in the heart $\Rep(\pi_0(S_\varphi)) \subset \QCoh([\ast / \pi_0(S_\varphi)])$, which is semisimple since $\Lambda$ is in characteristic $0$. 
\end{proof}

%\begin{rmk}
%    The index set $\pi_X^{-1}(1)$ is a singleton when $H$ is simply connected (Kneser's theorem) or a special group in the sense of Serre. By \Cref{cor:innerformBunGXexp}, similar multiplicity formulas are available for pure inner forms of $G$. 
%\end{rmk}

The vanishing of higher extension groups fails when the center of $H$ is not anisotropic. 

\begin{exa}
    Suppose that $H = \bb{G}_m$ and $E / F$ is an unramified quadratic extension. Then, $\Qla_\triv$ is a direct summand of $C_c^\infty(E^\times / F^\times, \Qla)$ and we have 
    \[
        \Ext^1_{E^\times}(\Qla_\triv, \Qla_\triv) \cong \Ext^1_{\Qla[T^{\pm}]}(\Qla[T^\pm]/(T-1), \Qla[T^\pm]/(T-1)) = \Qla. 
    \]
    In particular, $\Ext^1_{E^\times}(C_c^\infty(E^\times / F^\times, \Qla), \Qla_\triv) \neq 0$. 
\end{exa}

%\begin{rmk}
%    The local conjecture of Sakellaridis-Venkatesh \cite[Conjecture 16.2.2]{SV17} concerns the Plancherel decomposition of the $L^2$-spectrum $L^2(X(F))$, while $\cl{P}_X$ geometrizes $C_c^\infty(X(F), \Qla)$. It predicts that the $A$-parameter of a smooth irreducible representation appearing in the $L^2$-spectrum factors through the dual group of a spherical $G$-variety $X$. It would be interesting to search for its relation to \Cref{conj:normP=L}. 
%\end{rmk}

\subsection{Compatibility with Eisenstein series}

In this section, we introduce parabolic inductions of dual pairs and show that parabolic inductions of normalized period sheaves and $L$-sheaves are compatible under the Eisenstein compatibility. 

\subsubsection{The $\cl{A}$-side}

Let $P \subset G$ be a parabolic subgroup with a Levi subgroup $M$. 

\begin{defi}
    For a smooth quasi-projective $M$-variety $X$, let 
    \[
        \Ind_P^G(X) = G \times^P X
    \]
    be the parabolic induction of $X$. 
\end{defi}

\begin{lem}
    The parabolic induction $\Ind_P^G(X)$ is a smooth quasi-projective $G$-variety. 
\end{lem}
\begin{proof}
    By Sumihiro's theorem, $X$ admits an $M$-equivariant ample line bundle. Then, the natural map $\Ind_P^G(X) \to G/P$ is smooth and quasi-projective. Thus, we get the claim. 
\end{proof}

Consider a sequence
\begin{center}
    \begin{tikzcd}
        \Bun_G & \Bun_P \ar[l, "\mfr{p}"'] \ar[r, "\mfr{q}"] & \Bun_M. 
    \end{tikzcd}
\end{center}
By \Cref{exa:flagvar}, $\omega_{G / P}^* \und{\deg}$ is the square root of the dualizing complex of $\Bun_P$, so the normalized geometric Eisenstein series (see \cite[Definition 2.1.7]{HHS24}) is given by 
\[
    \Eis_{P!} = \mfr{p}_!( \mfr{q}^*(-) \otimes \omega_{G / P}^* \und{\deg}) \colon \cl{D}^\oc(\Bun_M, \Qla) \to \cl{D}^\oc(\Bun_G, \Qla). 
\]
%Let $\Eis_{P!} \colon \cl{D}^\oc(\Bun_M, \Qla) \to \cl{D}^\oc(\Bun_G, \Qla)$ be the normalized geometric Eisenstein series (see \cite[Definition 2.1.7]{HHS24}). 

\begin{prop} \label{prop:EisPnorm}
    For every smooth quasi-projective $M$-variety $X$, we have
    \[
        \cl{P}_{\Ind_P^G(X)}^\norm \cong \Eis_{P!}(\cl{P}_X^\norm). 
    \]
\end{prop}
\begin{proof}
    Since $[\Ind_P^G(X)/G] \cong [X / P]$, we have a commutative diagram
    \begin{center}
        \begin{tikzcd}
            & \Bun_G^{\Ind_P^G(X)} \ar[r, "\mfr{q}_X"] \ar[ld, "\pi_{\Ind_P^G(X)}"'] \ar[d, "\mfr{p}_X"] & \Bun_M^X \ar[d, "\pi_X"] \\
            \Bun_G & \Bun_P \ar[l, "\mfr{p}"'] \ar[r, "\mfr{q}"] & \Bun_M. 
        \end{tikzcd}
    \end{center}
    Here, the square is Cartesian. Then, we have
    \[
        \Eis_{P!}(\cl{P}_X^\norm) = \mfr{p}_! (\mfr{q}^*\pi_{X!} \omega_X^* \und{\deg} \otimes \omega_{G / P}^* \und{\deg}) \cong \pi_{\Ind_P^G(X)!} (\mfr{q}_X^* \omega_X^* \und{\deg} \otimes \mfr{p}_X^* \omega_{G/P}^* \und{\deg}). 
    \]
    It is enough to show that $\mfr{q}_X^* \omega_X^* \und{\deg} \otimes \mfr{p}_X^* \omega_{G/P}^* \und{\deg} \cong \omega_{\Ind_P^G(X)}^* \und{\deg}$. Let $m \colon \Bun_{\bb{G}_m} \times \Bun_{\bb{G}_m} \to \Bun_{\bb{G}_m}$ be the map sending a pair $(\cl{L}_1, \cl{L}_2)$ to $\cl{L}_1 \otimes \cl{L}_2$. By the description of $\und{\deg}$, we have
    \[
        \und{\deg} \boxtimes \und{\deg} \cong m^* \und{\deg}.  
    \]
    Thus, it is enough to show that 
    \[
        [\Ind_P^G(X) / G]  = [X / P] \to [X / M] \times [\ast / P] \xrightarrow{(\Omega_X^\tp, \Omega_{G / P}^\tp)} [\ast / \bb{G}_{m}] \times [\ast / \bb{G}_m] \to [\ast / \bb{G}_m]
    \]
    equals $\Omega_{\Ind_P^G(X)}^\tp$. Now, let $\Omega_{(-)}^\tp$ also denote the canonical bundle by abuse of notation. Then, $\Omega_{\Ind_P^G(X)}^\tp \otimes p_X^* \Omega_{G / P}^{\tp, \vee}$ is the relative canonical bundle of $\Ind_P^G(X) \to G / P$. Moreover, 
    \begin{center}
        \begin{tikzcd}
            \Ind_P^G(X) \ar[r, "q_X"] \ar[d] & \lbrack X / M \rbrack \ar[d] \\
            G / P \ar[r] & \lbrack \ast / M \rbrack
        \end{tikzcd}
    \end{center}
    is Cartesian, so we get $\Omega_{\Ind_P^G(X)}^\tp \cong  p_X^* \Omega_{G / P}^\tp \otimes q_X^* \Omega_X^\tp$. This proves the claim. 
\end{proof}

\subsubsection{The $\cl{B}$-side}

Let $\LP \subset \LG$ be a parabolic subgroup with a Levi subgroup $\LM$. 

\begin{defi}
    For a derived $\LM$-scheme $\widehat{X}$ locally of finite presentation such that $\clas\widehat{X}$ is $\widehat{M}$-quasi-projective with a grading, let
    \[
        \Ind_{\widehat{P}}^{\widehat{G}}(\widehat{X}) = \LG \times^{\LP} \widehat{X}
    \]
    be the parabolic induction of $\widehat{X}$ equipped with a grading induced from $\widehat{X}$, so that $a\cdot (g, x) = (g, ax)$ for $a \in \bb{G}_{gr}$, $g \in \LG$ and $x \in \widehat{X}$.   
\end{defi}
\begin{lem}
    The parabolic induction $\Ind_{\widehat{P}}^{\widehat{G}}(\widehat{X})$ is a derived $\LG$-scheme locally of finite presentation such that $\clas\Ind_{\widehat{P}}^{\widehat{G}}(\widehat{X})$ is $\widehat{G}$-quasi-projective. 
\end{lem}
\begin{proof}
    Since $\clas\widehat{X}$ is $\widehat{M}$-quasi-projective, 
    \[
        \clas\Ind_{\widehat{P}}^{\widehat{G}}(\widehat{X}) = \LG \times^{\LP} \clas \widehat{X}
    \]
    is a scheme and $\widehat{G}$-quasi-projective. Thus, $\Ind_{\widehat{P}}^{\widehat{G}}(\widehat{X}) \to \widehat{G} / \widehat{P}$ is schematic since it is smooth locally on target. Since $\widehat{X}$ and $\widehat{G} / \widehat{P}$ are locally of finite presentation, $\Ind_{\widehat{P}}^{\widehat{G}}(\widehat{X})$ is a derived scheme locally of finite presentation. 
\end{proof}

Consider a sequence
\begin{center}
    \begin{tikzcd}
        \Par_{\LG} & \Par_{\LP} \ar[l, "p^{\spec}"'] \ar[r, "q^{\spec}"] & \Par_{\LM}. 
    \end{tikzcd}
\end{center}
The spectral Eisenstein series (see \cite[Proposition 3.3.2]{Zhu21}) is given by 
\[
    \Eis_{P}^{\spec} = p^\spec_* q^{\spec!} \colon \IndCoh(\Par_{\LM}) \to \IndCoh(\Par_{\LG}). 
\]

\begin{prop} \label{prop:EisLnorm}
    For every derived $\LM$-scheme $\widehat{X}$ locally of finite presentation such that $\clas\widehat{X}$ is $\widehat{M}$-quasi-projective, we have
    \[
        \cl{L}_{\Ind_{\widehat{P}}^{\widehat{G}}(\widehat{X})}^\norm \cong \Eis_{P}^\spec(\cl{L}_{\widehat{X}}^\norm). 
    \]
\end{prop}
\begin{proof}
    Let $\widehat{Y} = \Ind_{\widehat{P}}^{\widehat{G}}(\widehat{X})$. Since $[\widehat{Y} / \LG \times \bb{G}_{gr}] \cong [\widehat{X} / \LP \times \bb{G}_{gr}]$, we have a $\bb{G}_{gr}$-equivariant commutative diagram
    \begin{center}
        \begin{tikzcd}
            & \Par_{\LG}^{\widehat{Y}, \norm} \ar[r, "q_{\widehat{X}}"] \ar[ld, "\pi^\norm_{\widehat{Y}}"'] \ar[d, "p_{\widehat{X}}"] & \Par_{\LG}^{\widehat{X}, \norm} \ar[d, "\pi_{\widehat{X}}^\norm"] \\
            \Par_{\LG} & \Par_\LP \ar[l, "p^\spec"'] \ar[r, "q^\spec"] & \Par_\LM. 
        \end{tikzcd}
    \end{center}
    Here, the square is Cartesian. Thus, we have
    \[
        \Eis_{P}^\spec(\cl{L}_{\widehat{X}}^\norm) \cong p^\spec_*(q^{\spec!}\pi_{\widehat{X}*}^\norm \omega_{\widehat{X}}^\norm)^\shear \cong p^{\spec}_*(p_{\widehat{X}*} \omega_{\widehat{Y}}^\norm)^\shear = \cl{L}_{\widehat{Y}}^\norm. 
    \]
\end{proof}

\subsubsection{Eisenstein compatibility}

Let $P \subset G$ be a parabolic subgroup with a Levi subgroup $M$. Let $\ov{P}$ be the opposite of $P$. 
%Let $\LPop \subset \LG$ be the $L$-group of the \textit{opposite} of $P$ with a Levi subgroup $\LM$. 
In \cite[Conjecture 1.4.7]{Han24}, the categorical local Langlands correspondence is expected to satisfy the Eisenstein compatibility, that is, the commutativity of the following diagram: 
\begin{center}
    \begin{tikzcd}
        \cl{D}^\oc(\Bun_M, \Qla) \ar[r, "\Eis_{P!}"] \ar[d, "\bb{L}_M"] & \cl{D}^\oc(\Bun_{G}, \Qla) \ar[d, "\bb{L}_{G}"] \\
        \IndCoh(\Par_{\LM}) \ar[r, "\Eis_{\ov{P}}^\spec"] & \IndCoh(\Par_{\LG}).
    \end{tikzcd}
\end{center}

\begin{defi}
    We define the parabolic induction of a dual pair $(M, X) \leftrightarrow (\LM, \widehat{X})$ along $P$ by
    \[
        (G, \Ind_P^G(X)) \leftrightarrow (\LG, \Ind_{\widehat{\ov{P}}}^{\widehat{G}}(\widehat{X})). 
    \]
\end{defi}

By \Cref{prop:EisPnorm} and \Cref{prop:EisLnorm}, the Eisenstein compatibility predicts that the normalized period conjecture is preserved under parabolic inductions. 

\begin{exa}
    Let $T \subset B \subset G$ be a Borel pair. Since $(T, T) \leftrightarrow (\LT, \ast)$ and $(T, \ast) \leftrightarrow (\LT, \widehat{T})$ are dual pairs, it is expected that 
    \[
        (G, G / U) \leftrightarrow (\LG, \widehat{G} / \widehat{\ov{B}}) ,\quad
        (G, G / B) \leftrightarrow (\LG, \widehat{G} / \widehat{\ov{U}})
    \]
    are dual pairs. Here, $U$ (resp. $\widehat{\ov{U}}$) is the unipotent radical of $B$ (resp.\ $\widehat{\ov{B}}$). 
\end{exa}

\section{Comparisons of period and $L$-sheaves}
\label{sec:comparison}

In this section, we verify the normalized period conjecture for the Iwasawa-Tate and Hecke periods, assuming the existence of the categorical local Langlands correspondence for $\GL_2$ with the Eisenstein compatibility.

\subsection{Comparison for the Iwasawa-Tate period} \label{ssec:compIwTate}

In this section, we compare the normalization of the period sheaf and the $L$-sheaf for the Iwasawa-Tate period. In this case, the dual pair and the normalization factors are as follows (see \cite[Section 6.1.3, 6.5.1]{FW25}). 
\begin{center}
    \begin{tabular}{cccc|cc}
        $G$ & $X$ & $\LG$ & $\widehat{X}$ & $\eta_X$ & $z_{\widehat{X}}$ \\ \hline
        $\bb{G}_m$ & $\std$ & $\bb{G}_m$ & $\std$ & $\std$ & $\std$
    \end{tabular}
\end{center}
Here, $\eta_X$ is an eigencharacter of $X$ and the grading on $\widehat{X}$ is given by a central cocharacter $z_{\widehat{X}}$. In particular, $\cl{P}_X^\norm$ and $\cl{L}_{\widehat{X}}^\norm$ can be computed as in \Cref{prop:normPetaX} and \Cref{cor:normLzX}. 

Let $\Lambda = \Qla$. For $\bb{G}_m$, the categorical local Langlands correspondence 
\[
    \bb{L}_{\bb{G}_m} \colon \cl{D}^\oc(\Bun_{\bb{G}_m}, \Qla) \cong \IndCoh(\Par_{\bb{G}_m})
\]
is proved by \cite{Zou24}. Let us recall the basic properties of $\bb{L}_{\bb{G}_m}$. 

\begin{enumerate}
    \item The functor $\bb{L}_{\bb{G}_m}$ restricts to the equivalence
    \[
        \bb{L}_{\bb{G}_m, n} \colon \cl{D}^\oc(\Bun_{\bb{G}_m, n}, \Qla) \cong \IndCoh(\Par_{\bb{G}_m})_{\std^n} \cong \QCoh(Z^1(W_F, \bb{G}_m)). 
    \]
    For $n = 0$, the Whittaker sheaf represented by $C_c^\infty(F^\times, \Qla)$ maps to $\cl{O}_{\Par_{\bb{G}_m}}$. 
    \item The category $\cl{D}^\oc(\Bun_{\bb{G}_m, n}, \Qla)$ is equivalent to 
    \[
        \colim_{K \subset F^\times} \cl{D}(\Qla[F^\times/K]). 
    \]
    For each compact open subgroup $K \subset F^\times$, we have a compact open subgroup $K_F \subset W_F$ such that $W_F / K_F \cong F^\times/K$ via the local class field theory. Then, 
    \[
        Z^1(W_F/K_F, \bb{G}_m) \subset Z^1(W_F, \bb{G}_m)
    \]
    is closed and open, and represented by $\Qla[F^\times/K]$. Then, we have
    \begin{center}
        \begin{tikzcd}
        \cl{D}^{\oc}(\Bun_{\bb{G}_{m,n}} )\rar["\sim"]  &\QCoh(Z^1(W_F, \bb{G}_m)) \\ 
            \cl{D}(\Qla[F^\times/K])\ar[r,"\sim"] \ar[u, hook] & \QCoh(Z^1(W_F/K_F, \bb{G}_m)).  \ar[u, hook]
        \end{tikzcd} 
    \end{center}
\end{enumerate}

By combining \Cref{prop:IwTateperiod} and \Cref{prop:IwTateLsheaf}, we get the following description. 

\begin{thm} \label{thm:IwTatecomp}
    For the Iwasawa-Tate period, the normalized period sheaf $\cl{P}_X^\norm$ is given by 
    \[
        \cl{P}^{\norm}_{X,n} = \Qla_\norm^{-1/2}[n] (n < 0), \cl{P}^\norm_{X,0} = C_c^\infty(F, \Qla) \otimes \Qla_\norm^{-1/2}, \cl{P}^{\norm}_{X,n} = \Qla^{1/2}_\norm [-n] (n > 0), 
    \]
    and the normalized $L$-sheaf $\cl{L}_{\widehat{X}}^\norm$ is given by
    \[
         \cl{L}_{\widehat{X}}^\norm \cong \cl{O}_{\Par_{\bb{G}_m}} \oplus \bigoplus_{n \geq 1} \cl{O}_{\bb{G}_m}/(T-q^{1/2})[-n] \oplus \bigoplus_{n \geq 1} \cl{O}_{\bb{G}_m}/(T-q^{-1/2})[- n]. 
    \]
    Here, 
    \begin{itemize}
        \item the $n$-th term of the second factor (resp.\ the third factor) lies in $\QCoh(\Par_{\bb{G}_m, \triv})_{\std^{-n}}$ \textup{(resp.\ $\QCoh(\Par_{\bb{G}_m, \triv})_{\std^{n}}$)}, and
        \item $T$ is the coordinate of $\bb{G}_m$ in the component $ \Par_{\bb{G}_m, \triv} \cong \bb{G}_m \times [*/\bb{G}_m]$. 
    \end{itemize} In particular, $\bb{L}_{\bb{G}_m}$ sends $\cl{P}_X^\norm$ to $\cl{L}_{\widehat{X}}^\norm$. 
\end{thm}
\begin{proof}
    Since $\eta_X = \std$, the description of $\cl{P}_X^\norm$ follows from \Cref{prop:IwTateperiod}. Moreover, since $z_{\widehat{X}} = \std$, the grading of $\bb{G}_{gr}$ and that of $\bb{G}_m$ on $\cl{L}_{V}$ are equal, so the description of $\cl{L}_{\widehat{X}}^\norm$ follows from \Cref{prop:IwTateLsheaf}: $\tau_{\sqrt{\cyc}}\vert_{\Par_{\bb{G}_m, \triv}}$ sends $T$ to $q^{-1/2}T$, so 
    \[ 
        \tau^{*}_{\sqrt{\cyc}} \smbr{\cl{O}_{\bb{G}_m}/(T-1)} \cong \cl{O}_{\bb{G}_m}/(T-q^{1/2}), 
    \]
    \[ 
        \tau^{*}_{\sqrt{\cyc}} \smbr{\cl{O}_{\bb{G}_m}/(qT-1)} \cong \cl{O}_{\bb{G}_m}/(T-q^{-1/2}). 
    \]
    Now, we will show that $\bb{L}_{\bb{G}_m}$ sends $\cl{P}_X^\norm$ to $\cl{L}_{\widehat{X}}^\norm$. 

    For $n < 0$, it is easy to see that $\bb{L}_{\bb{G}_m, n}$ sends $\Qla^{-1/2}_\norm[n]$ to $\cl{O}_{\bb{G}_m}/(T-q^{1/2})[n]$. For $n > 0$, it is easy to see that $\bb{L}_{\bb{G}_m, n}$ sends $\Qla^{1/2}_\norm[-n]$ to $\cl{O}_{\bb{G}_m}/(T-q^{-1/2})[-n]$. For $n = 0$, it is enough to show that $\bb{L}_{\bb{G}_m, 0}$ sends $C_c^\infty(F, \Qla)$ to $\cl{O}_{\Par_{\bb{G}_m}}$. Since $\cl{O}_{\Par_{\bb{G}_m}}$ corresponds to $C_c^\infty(F^\times, \Qla)$ under $\bb{L}_{\bb{G}_m, 0}$, it is enough to show $C_c^\infty(F^\times, \Qla) \cong C_c^\infty(F, \Qla)$.  

    Recall that we have a Bernstein decomposition
    \[
        \Rep^\sm(F^\times) = \prod_{\chi \colon O_F^\times \to \Qla} \Rep^\sm(F^\times)_\chi. 
    \]
    Here, each smooth $F^\times$-representation $V$ has a decomposition
    \[
        V = \bigoplus_\chi V_\chi ,\quad O_F^\times\vert_{V_\chi} = \chi \cdot \id_{V_\chi}
    \]
    and the above decomposition sends $V$ to $\bigoplus_{\chi} V_\chi$. We will show 
    \[
        C_c^\infty(F^\times, \Qla)_\chi \cong C_c^\infty(F, \Qla)_\chi
    \]
    for each $\chi$. 

    First, we have a short exact sequence
    \[
        C_c^\infty(F^\times, \Qla) \hookrightarrow C_c^\infty(F, \Qla) \twoheadrightarrow \Qla_\triv, 
    \]
    so $C_c^\infty(F^\times, \Qla)_\chi \hookrightarrow C_c^\infty(F, \Qla)_\chi$ is an isomorphism for a nontrivial character $\chi$. Moreover, $C_c^\infty(F, \Qla)_\triv$ is isomorphic to $C_c^\infty(F^\times, \Qla)_\triv$ by sending a function $f$ to a function $x \mapsto f(\pi x) - f(x)$. Thus, we get the claim. 
\end{proof}

We can verify the functional equation for the Iwasawa-Tate period by hand. 

%\subsubsection{Functional equation}
%We prove a functional equation for the Iwasawa-Tate period. Consider the opposite action on $\bb{A}^1$ of $\bb{G}_m$. In this case, the dual pair and the normalization factors are as follows (see \cite[Section 6.1.3, 6.5.2]{FW25}). 
%\begin{center}
%    \begin{tabular}{cccc|cc}
%        $G$ & $X'$ & $\LG$ & $\widehat{X}'$ & $\eta_{X'}$ & $z_{\widehat{X}'}$ \\ \hline
%        $\bb{G}_m$ & $\std^{-1}$ & $\bb{G}_m$ & $\std^{-1}$ & $\std^{-1}$ & $\std^{-1}$
%    \end{tabular}
%\end{center}
%By looking at the description of $\cl{P}_{X'}^\norm$, we prove the independence of the polarization for the Iwasawa-Tate period in the following sense. 

\begin{prop} \label{prop:FEIwTate}
    There is an isomorphism $\cl{P}_{\std}^\norm \cong \cl{P}_{\std^\vee}^\norm$. 
\end{prop}
\begin{proof}
    We have the following Cartesian diagram. 
    \begin{center}
        \begin{tikzcd}
            \Bun_{\bb{G}_m}^{\std^\vee} \ar[r] \ar[d] & \Bun_{\bb{G}_m}^{\std} \ar[d] \\
            \Bun_{\bb{G}_m} \ar[r, "\std^{-1}"] & \Bun_{\bb{G}_m}. 
        \end{tikzcd}
    \end{center}
    %Here, the bottom map is given by the isomorphism $\bb{G}_m \cong \bb{G}_m$ sending $x$ to $x^{-1}$. 
    It is easy to see that $\eta_{\std^\vee} = \std^{-1}$, so $\cl{P}_{\std^\vee}^\norm$ is the pullback of $\cl{P}_{\std}^\norm$ along $\std^{-1}$. 
    
    Now, the pullback along $\std^{-1}$ sends a character on $\Bun_{\bb{G}_m, n}$ to the inverse on $\Bun_{\bb{G}_m, -n}$. Thus, $\cl{P}_{\std, n}^\norm \cong \cl{P}_{\std^\vee, n}^\norm$ for $n \neq 0$ by the description in \Cref{thm:IwTatecomp}. For $n = 0$, $\cl{P}_{\std^\vee, 0}^\norm$ equals $C_c^\infty(F, \Qla) \otimes \Qla_{\norm}^{1/2}$ on which $F^\times$ acts by the \textit{right translation}. Fix a nontrivial additive character $\psi \colon F \to \Qla$. The Fourier transform
    \[
        C_c^\infty(F, \Qla) \ni f(x) \mapsto \widehat{f}(y) = \int_{F} f(x) \psi(-xy) dx \in C_c^\infty(F, \Qla)
    \]
    provides an isomorphism $C_c^\infty(F, \Qla) \cong C_c^\infty(F, \Qla) \otimes \Qla_{\norm}$, where $F^\times$ acts on the first (resp. second) $C_c^\infty(F, \Qla)$ by the left (resp. right) translation. Thus, this provides an isomorphism $\cl{P}_{\std, 0}^\norm \cong \cl{P}_{\std^\vee, 0}^\norm$. 
\end{proof}

\subsection{Comparison for the Hecke period} \label{ssec:compHecke}

In this section, we compare the normalization of the period sheaf and the $L$-sheaf for the Hecke period. In this case, the dual pair and the normalization factors are as follows (see \cite[Section 7.1.3, 7.1.6]{FW25}). 
\begin{center}
    \begin{tabular}{cccc|cc}
        $G$ & $X$ & $\LG$ & $\widehat{X}$ & $\eta_X$ & $z_{\widehat{X}}$ \\ \hline
        $\GL_2$ & $\GL_2/A$ & $\GL_2$ & $\std$ & $\triv$ & $\textrm{scaling}$
    \end{tabular}
\end{center}
Here, $\eta_X$ is an eigencharacter of $X$ and the grading on $\widehat{X}$ is given by a central cocharacter $z_{\widehat{X}}$ sending $x$ to ${\scriptsize \begin{pmatrix} x & \\ & x \end{pmatrix}}$. In particular, $\cl{P}_X^\norm$ and $\cl{L}_{\widehat{X}}^\norm$ can be computed as in \Cref{prop:normPetaX} and \Cref{cor:normLzX}. 

Let $\Lambda = \Qla$. For $\GL_2$, a proof of the categorical local Langlands correspondence does not appear in the literature yet. Nevertheless, it predicts a categorical equivalence
\[
    \bb{L}_{\GL_2} \colon \cl{D}^\oc(\Bun_{\GL_2}, \Qla) \cong \IndCoh(\Par_{\GL_2})
\]
satisfying the Whittaker and Eisenstein compatibility.
We get the following compatibility from the descriptions in \Cref{prop:HeckepEis} and \Cref{prop:HeckeL}. 

%Let us review these properties. 

%\begin{enumerate}
%    \item Let $U \subset B$ be the unipotent radical. Fix a nontrivial additive character $\psi \colon U(F) \cong F \to \Qla$. Then, $\bb{L}_{\GL_2}$ send the Whittaker sheaf $\cl{W}_\psi = i^1_! \cInd_{U(F)}^{G(F)} \psi$ to $\cl{O}_{\Par_{\GL_2}}$. 
%    \item The diagram
%    \begin{center}
%        \begin{tikzcd}
%            \cl{D}^\oc(\Bun_T, \Qla) \ar[r, "\Eis_{B!}"] \ar[d, "\bb{L}_T"] & \cl{D}^\oc(\Bun_{\GL_2}, \Qla) \ar[d, "\bb{L}_{\GL_2}"] \\
%            \IndCoh(\Par_T) \ar[r, "\Eis_{\ov{B}}^\spec"] & \IndCoh(\Par_{\GL_2}) 
%        \end{tikzcd}
%    \end{center}
%    commutes (see \cite[Conjecture 1.4.7]{Han24}). We also get a commutative diagram if we replace $\Eis_{B!}$ and $\Eis_{\ov{B}}^\spec$ with $\Eis_{\ov{B}!}$ and $\Eis_{B}^\spec$, respectively.
%\end{enumerate}
%Under these expected properties of $\bb{L}_{\GL_2}$, we can see the following compatibility of $\cl{P}^\norm_X$ and $\cl{L}^\norm_{\widehat{X}}$. Note that $\cl{P}_X^\norm = \cl{P}_X$ since $\eta_X$ is trivial. Recall the descriptions in \Cref{prop:HeckepEis} and \Cref{prop:HeckeL}. 

\begin{thm} \label{thm:Heckecomp}
    Suppose that there is a categorical equivalence $\bb{L}_{\GL_2}$ with the Whittaker and Eisenstein compatibility. For $n \neq 0$, we have
    \[
        \bb{L}_{\GL_2}(\cl{P}_{X,n}^\norm) \cong \cl{L}_{\widehat{X}, n}^\norm. 
    \]
    For $n = 0$, the first (resp.\ last) term of the fiber sequence
    \[
        \cl{W}_\psi \hookrightarrow \cl{P}^\norm_{X, 0} \to \Eis_{B!}(i^1_! \cInd_{A(F)}^{T(F)} \Qla_\norm^{-1/2})
    \]
    maps under $\bb{L}_{\GL_2}$ to the first (resp.\ last) term of the fiber sequence
    \[
        \cl{O}_{\Par_{\GL_2}} \to \cl{L}^\norm_{\widehat{X}, 0} \to \Eis_{\ov{B}}^\spec(\cl{O}_{\Par_{\bb{G}_m}} \boxtimes  (i_{\cyc^{-1/2}*} \Qla)_{\triv}). 
    \]
\end{thm}
\begin{proof}
    For $n < 0$, we have 
    \[
        \cl{L}_{\widehat{X}, n}^\norm = \Eis_{\ov{B}}^{\spec}(\cl{O}_{\Par_{\bb{G}_m}} \boxtimes (i_{\cyc^{-1/2} *}\Qla)_{\std^n}[n]). 
    \]
    Then, the claim follows because $\bb{L}_T = \bb{L}_{\bb{G}_m} \boxtimes \bb{L}_{\bb{G}_m}$ sends
    \[
        i^{b_n}_! \cInd_{A(F)}^{T(F)} \Qla_\norm^{-1/2}[n] = i^1_! C_c^\infty(F^\times, \Qla) \boxtimes i^{\pi^{-n}}_! \Qla_\norm^{-1/2}[n]
    \]
    %with $i^{\pi^{-n}} \colon \Bun_{\bb{G}_m, n} \hookrightarrow \Bun_{\bb{G}_m}$ 
    to 
    \[
        \cl{O}_{\Par_{\bb{G}_m}} \boxtimes (i_{\cyc^{-1/2} *}\Qla)_{\std^n}[n]. 
    \]
    For $n = 0$, the claim follows from the Whittaker compatibility and the previous argument. 
    
    For $n > 0$, we have $\cl{L}_{\std, n}^\norm \cong \cl{L}_{\std^\vee, -n}^\norm$ by \Cref{cor:FEstd}. For its computation, the same argument as in \Cref{prop:HeckeL} works by using the following closed-open decomposition. 
    \begin{center}
        \begin{tikzcd}
            \lbrack \{0\}/\GL_2 \rbrack \ar[r, hook] & \lbrack \std^\vee/\GL_2 \rbrack & \lbrack (\bb{A}^2-\{0\})/\GL_2 \rbrack \cong \lbrack \ast/\Mir_2 \rbrack. \ar[l, hook']
        \end{tikzcd}
    \end{center}
    Here, $\Mir_2$ is the stabilizer of ${\scriptsize \begin{pmatrix} 0 \\ 1 \end{pmatrix}}$ under $\std^\vee$ and written as
    \[
        \Mir_2 = \left\{ \begin{pmatrix} \ast & \ast \\ 0 & 1 \end{pmatrix} \right\} \subset \GL_2. 
    \]
    The same argument as in \Cref{prop:HeckeL} shows that
    \[
        \cl{L}_{\std^\vee, -n} \cong \Eis_{B}^\spec(\cl{O}_{\Par_{\bb{G}_m}} \boxtimes (i_{\triv*} \Qla)_{\std^{n}}). 
    \]
    Since the action of $\bb{G}_{gr}$ on $\std^\vee$ is scaling, we have $z_{\std^\vee} = z_{\std}^{-1}$ and we get 
    \[
        \cl{L}_{\std, n}^\norm \cong \cl{L}_{\std^\vee, -n}^\norm \cong \Eis_{B}^\spec(\cl{O}_{\Par_{\bb{G}_m}} \boxtimes (i_{\sqrt{\cyc}*} \Qla)_{\std^{n}}[-n]),
    \]
    The Eisenstein compatibility ensures that $\bb{L}_{\GL_2}(\cl{P}_{X, n}^\norm) \cong \cl{L}_{\std, n}^\norm$ by \Cref{prop:HeckepEis}. 
    %For $n > 0$, we defer the proof in \Cref{sssec:FEL}. 
\end{proof}

%\begin{rmk}
%    By the similar argument as in \cite[Section 7.3]{FW25}, one may construct a fiber sequence of $\cl{P}_{X,n}$ that matches \Cref{prop:HeckeL} under $\bb{L}_{\GL_2}$ even for $n > 0$. For this, one needs to formulate the Hecke action and the Fourier transform for Banach-Colmez spaces (as in \cite{ALB25}) in the $\cl{D}^\oc$-formalism. We do not pursue this here since this fiber sequence is inefficient to give the whole description of $\cl{P}_{X,n}$ (compare with \Cref{prop:HeckepEis}). 
%\end{rmk}

\appendix
\addtocontents{toc}{\protect\setcounter{tocdepth}{1}}

\section{Recollections of derived algebraic geometry} \label{app:DAG}

In this section, we review the theory of derived stacks and ind-coherent sheaves developed by \cite{GR17I}. Let $\Lambda$ be a classical $\bb{Q}$-algebra. We assume that $\Lambda$ is a $G$-ring for the use of the Artin-Lurie representability. Let $\CAlg_{\Lambda}$ be the $\infty$-category of animated $\Lambda$-algebras and let $\CAlg_{\Lambda}^\heartsuit \subset \CAlg_{\Lambda}$ be the subcategory of classical $\Lambda$-algebras. 

\subsection{Basic definitions and properties}

\subsubsection{Topology}

A morphism of animated algebras $A \ra B$ is flat (resp.\ \'{e}tale) if and only if $\pi_0(A) \to \pi_0(B)$ is flat (resp.\ \'{e}tale) and $\pi_i(A) \otimes_{\pi_0(A)} \pi_0(B) \cong \pi_i(B)$. Following the reason explained in \cite[Section 4.1]{GR17I}, $\CAlg_\Lambda^\op$ is equipped with the \'{e}tale topology unless otherwise stated. An \'{e}tale stack on $\CAlg_\Lambda^\op$ is called a derived stack. The category of derived stacks is denoted by $\DStk_{\Lambda}$. Its final object $\Spec(\Lambda)$ is denoted by $\ast$. 

\subsubsection{Classical stacks}

The classical truncation of a derived stack $S$ is 
\[
    \clas S = S\vert_{\CAlg_{\Lambda}^\heartsuit}, 
\]
the restriction to $\CAlg_{\Lambda}^\heartsuit$. Let $\clas \DStk_{\Lambda}$ denote the category of \'{e}tale stacks over $\CAlg_{\Lambda}^\heartsuit$. There is a fully faithful embedding
\[
    \LLKE\colon \clas \DStk \hookrightarrow \DStk
\]
given by the sheafification of the left Kan extension along $(\CAlg^\heartsuit_\Lambda)^\op \hookrightarrow \CAlg_{\Lambda}^\op$. A derived stack is called classical if it lies in the essential image of $\LLKE$. In other words, a derived stack $S$ is classical if and only if the natural map $\clas S \to S$ is an isomorphism. 

\subsubsection{Artin stacks} 

For each $n \geq 0$, the notion of $n$-Artin stacks is defined inductively as in \cite[Section 4.1]{GR17I}. Here, we just review $1$-Artin stacks, which are all we see in application. 

Following \cite[Definition 5.1.3]{DAG}, we say that a derived stack $S$ is a derived $1$-stack if $S$ admits a smooth cover from a derived scheme. Moreover, if $S$ has representable diagonal, we say that $S$ is a $1$-Artin stack.  

We say that a derived $1$-stack is locally (resp.\ almost) of finite presentation if it has a smooth cover from a derived scheme locally (resp.\ almost) of finite presentation. For the latter notion, see \cite[Proposition 3.2.14, Definition 5.3.1]{DAG}.

\subsubsection{Quasi-coherent and ind-coherent sheaves}

Let $\QCoh(S)$ (resp.\ $\IndCoh(S)$) denote the category of quasi-coherent (resp.\ ind-coherent) sheaves on a derived stack $S$. The dualizing complex of $S$ is 
\[
    \omega_S = \pi^! \cl{O}_{\Spec(\Lambda)} \in \IndCoh(S)
\]
with $\pi \colon S \to \Spec(\Lambda)$ the structure morphism. 

When an Artin stack $S$ admits a smooth cover from an $n$-connective derived scheme for some $n$, $S$ is said to be eventually coconnective (see \cite[Proposition 4.6.4]{GaiStack}). In this case, we have a fully faithful embedding \cite[11.7.3]{Gai13}
\[
    \Xi_S \colon \QCoh(S) \hookrightarrow \IndCoh(S). 
\]
When $S$ is an eventually coconnective derived scheme, $\Xi_S$ is obtained by the ind-extension of $\Perf(S) \hookrightarrow \Coh(S)$. When $\Lambda$ is a field and $S$ is smooth, $\Xi_S$ is an equivalence. %We will frequently use the embedding $\Xi_S$ without mentioning. 

When $S$ is QCA \cite[Definition 1.1.8]{DG13}, $\IndCoh(S)$ is compactly generated and we have $\IndCoh(S) = \Ind(\Coh(S))$ (see \cite[Theorem 3.3.5]{DG13}). A typical example of QCA stacks is the quotient of separated derived schemes almost of finite presentation by a smooth affine group scheme. 

A derived stack $S$ is said to be perfect when $S$ has an affine diagonal and $\QCoh(S) \cong \Ind(\Perf(S))$. Examples of perfect stacks are studied in \cite{BZFN10}. For example, see \cite[Proposition 3.21]{BZFN10}. 

\subsection{Deformation theory}

For $A \in \CAlg_\Lambda$ and a connective $A$-module $M\in\QCoh(A)^{\le 0}$, let
\[
    A \oplus M
\]
denote the trivial square-zero extension of $A$ by $M$ (see \cite[Chapter 25.3.1]{Lurie2018SAG}). When $A$ and $M$ are static, the multiplication law is written as
\[
    (a,m)\cdot(a',m')=(aa',am'+a'm) \quad (a \in A ,\quad m \in M). 
\]

\begin{defi}\label{defi:relative_cotangent_complex}
    Let $f\colon X \to Y$ be a morphism in $\DStk_\Lambda$ and let $A \in \CAlg_{\Lambda}$. For $x \in X(A)$ and $M\in \QCoh(A)^{\le 0}$, the \textit{space of relative derivations of $f$ at $x$ with values in $M$} is the homotopy fiber
    \[
        \Der(X/Y,x,M)
        =
        \fib\!\left(
            X(A \oplus M)
            \longrightarrow
            Y(A \oplus M)\times_{Y(A)} X(A)
        \right)_{(y,x)},
    \]
    where $y\in Y(A \oplus M)$ is the pullback of $f(x)\in Y(A)$ along the zero section $A\to A\oplus M$. A \textit{relative cotangent complex} of $f$ is a complex $L_{X/Y}\in \QCoh(X)$ such that for every $(A,x)$ as above, there is a natural equivalence
    \[
        \Der(X/Y,x,M) \cong \Map(x^*L_{X/Y}, M) \quad (M\in\QCoh(A)^{\le 0}).
    \]
    If it exists, $L_{X/Y}$ is unique up to a contractible choice.
\end{defi}

\begin{exa}\label{exa:cotangent_BG}
    Let $G$ be a smooth algebraic group over $\Lambda$, with Lie algebra
    $\mfr{g}=\Lie(G)$. The classifying stack $BG=[*/G]$ is a classical Artin
    $1$-stack, and one has
    \[
        L_{BG/\Lambda} \cong \mfr{g}^\vee[-1] \in \QCoh(BG). 
    \]
    Here, the $G$-action on $\mfr{g}^\vee$ is the coadjoint action. For a map $x\colon S\to BG$ corresponding to a $G$-torsor $\cl{P}$ on $S$, let
    \[
        \mfr{g}_{\cl{P}} = \cl{P}\times^G \frak g.
    \]
    Then, for any $M\in\QCoh(S)^{\le 0}$, we have
    \[
        \Der(BG,x,M) \cong \Map(\mfr{g}_{\cl{P}}^\vee[-1],M). 
    \]
\end{exa}

\begin{thm}\textup{(Artin-Lurie, \cite[Theorem 7.1.6]{DAG})} \label{thm:Artin_Lurie_representability}
    Let $F$ be a derived stack satisfying the following conditions. 
    \begin{enumerate}
        \item[\textup{(1)}] \textup{(Almost finite presentation)} For each $n\ge 0$, the restriction of $F$ to $n$-truncated connective $\Lambda$-algebras preserves filtered colimits.

        \item[\textup{(2)}] \textup{(Classical part)} The classical truncation $\clas F$ is a classical Artin stack locally of finite presentation over $\Lambda$. In particular, $\clas F$ takes values in groupoids. 

        \item[\textup{(3)}] \textup{(Cotangent complex)} The functor $F$ admits a cotangent complex. 

        \item[\textup{(4)}] \textup{(Nilcompleteness)} The functor $F$ is nilcomplete (or convergent), i.e.\ for every $R\in\CAlg_\Lambda$, the canonical map
        \[
            F(R) \to \varprojlim_n F(\tau_{\le n}R)
        \]
        is an equivalence.

        \item[\textup{(5)}] \textup{(Infinitesimal cohesiveness)} The functor $F$ is infinitesimally cohesive in the sense of \cite[Definition 3.4.1]{DAG}. 
    \end{enumerate}
    Then, $F$ is a derived Artin $1$-stack locally almost of finite presentation over $\Lambda$.
\end{thm}

In Lurie's original formulation, the condition on $\clas F$ is imposed in
deformation-theoretic terms; using the classical Artin representability
theorem for stacks in groupoids (see \cite[Theorem 7.1.1]{DAG}), one can replace the condition by \textup{(2)} as above. 

\subsection{Derived vector bundles}
For the definition of derived vector bundles, we adopt the convention in \cite[Section 6.1.1]{FYZ23}. Let $S$ be a derived stack. We define the totalization functor by
\[ 
    \Tot_{S}(-) \colon \QCoh(S)\ra \DStk_S, \quad
    \cl{E}\mapsto \smbr{T / S \mapsto \Map\smbr{\cl{O}_T , \cl{E}\vert_T}}. 
\]
When $\cl{E} \in \Perf(S)$, $\Tot_S(\cl{E})$ is referred to as the \textit{derived vector bundle} associated to $\cl{E}$. We sometimes denote this simply as $\bb{V}(\cl{E})$.

\begin{exa}\label{ex:affine_test_schemes}
    Suppose $T=\Spec(A)$ is affine. Then,
    \[ 
        \Tot_S(\cl{E})(A) = \Map_{A}\smbr{A, \cl{E}_A} \simeq \Omega^\infty (\tau^{\le 0}(\cl{E}_A))
    \]
    %The second isomorphism is via the Dold-Kan correspondence. 
    In the main body of the paper, we will omit $\Omega^\infty(-)$ by abuse of notation. 
\end{exa}

\begin{lem}
    If $\cl{E}$ is coconnective, $\Tot_S(\cl{E})$ is represented by the relative spectrum
    \[
        \underline{\Spec}_S(\Sym^\bullet \cl{E}^\vee). 
    \]
\end{lem}
\begin{proof}
    The claim follows since for every $T = \Spec(A)$, we have
    \[
        \Map_{\CAlg_A}(\Sym^\bullet \cl{E}_A^\vee, A) \cong \Map_A(\cl{E}_A^\vee, A) \cong \Tot_S(\cl{E})(A). 
    \]
\end{proof}

\begin{defi} \label{ex:derived_vector_bundle_G_representation}
    Let $G$ be a group scheme over $\Lambda$. For each $V \in \Rep(G)$, let
    \[
        \cl{O}_V \in \Perf([\ast / G])
    \]
    be the vector bundle such that $\Tot_{[\ast/G]}(\cl{O}_V) \cong [V / G]$ over $[\ast / G]$. 
\end{defi}

\subsection{The weight decomposition}
\label{ssec:weight_decomposition}
Let $\Delta$ be a diagonalizable group scheme over $\bb{Z}$ with a character group $X^*(\Delta)$. Let $S$ be a derived stack and let $X \to S$ be a $\Delta$-gerbe. It is known that we have a weight decomposition
\[
    \QCoh(X) \cong \prod_{\chi \in X^*(\Delta)} \QCoh(X)_{\chi}. 
\]
This is proved in \cite[Theorem A]{BS21} when $S$ is classical, and in \cite[Theorem 5.25]{BP23} when $S$ is derived and $\Delta = \bb{G}_m$. The pullback along $X \to S$ provides $\QCoh(X)_{\triv} \cong \QCoh(S)$. 

We take the index of the weight decomposition so that $\cl{O}_{\chi} \in \QCoh([\ast/\Delta])_{\chi}$ for each $\chi \in X^*(\Delta)$. When $X = S \times [\ast/\Delta]$, the tensor product with $\cl{O}_\chi^{-1}$ provides 
\[
    \QCoh(X)_{\chi} \cong \QCoh(X)_\triv \cong \QCoh(S). 
\]

Suppose that $X$ is a QCA stack and the weight decomposition for $\QCoh$ is available. Then, the weight decomposition restricts to an equivalence
\[
    \Coh(X) \cong \bigoplus_{\chi \in X^*(\Delta)} \Coh(X)_{\chi}. 
\]
The ind-extension provides the one for $\IndCoh$: 
\[
    \IndCoh(X) \cong \prod_{\chi \in X^*(\Delta)} \IndCoh(X)_{\chi}. 
\]

\subsection{Computation of symmetric powers}
\label{computation_symmetric_powers}

Let $S$ be an Artin stack. The symmetric power product $\Sym^\bullet_{\cl{O}_S}(-)$ is the left adjoint to the forgetful functor 
\[
    \CAlg(\QCoh(S)) \to \QCoh(S). 
\]
We omit the subscript when the context is clear. Concretely, for $\cl{E} \in \QCoh(S)$, we have
\[
    \Sym^\bullet(\cl{E}) = \bigoplus_{n\in \bb{Z}} \Sym^n(\cl{E}), \quad \Sym^n(\cl{E}) = (\cl{E}^{\otimes n})_{\Sigma_n}. 
\]
Here, $\Sigma_n$ denotes the symmetric group of degree $n$. In this section, we recall an explicit description of symmetric powers in characteristic $0$. 

Suppose that $\cl{E}$ is represented by a bounded complex $E = (E^\bullet, d^\bullet)$ of vector bundles on $S$ with amplitudes in $[a,b]$. For each $n \geq 0$, $E^{\otimes n}$ is represented by a complex of vector bundles given as follows: the $d$-th term is given by the direct sum
\[
    \bigoplus_{i_1 + i_2 + \cdots + i_n = d} E^{i_1} \otimes E^{i_2} \otimes \cdots \otimes E^{i_n}
\]
where the index runs over $n$-tuples $(i_1,\ldots,i_n)$ with $a \leq i_k \leq b$ for each $1\leq k \leq n$, and the differential 
\[
    E^{i_1} \otimes E^{i_2} \otimes \cdots \otimes E^{i_n} \to E^{j_1} \otimes E^{j_2} \otimes \cdots \otimes E^{j_n}
\]
is nonzero only if $i_k \neq j_k$ for exactly one index $k$ and $j_k = i_k + 1$ for such $k$, in which case the differential is given by $(-1)^{i_1 + \cdots + i_{k-1}} d^{i_k}$. The action of $\Sigma_n$ on the $d$-th term 
\[
    \bigoplus_{i_1 + i_2 + \cdots + i_n = d} E^{i_1} \otimes E^{i_2} \otimes \cdots \otimes E^{i_n}
\]
is given as follows: $\sigma \in \Sigma_n$ sends $a_1 \otimes a_2 \otimes \cdots \otimes a_n \in E^{i_1} \otimes E^{i_2} \otimes \cdots \otimes E^{i_n}$ to
\[
    \prod_{\substack{k < \ell \\ \sigma^{-1}(k) > \sigma^{-1}(\ell)}} (-1)^{i_k i_\ell} \cdot a_{\sigma^{-1}(1)} \otimes a_{\sigma^{-1}(2)} \otimes \cdots a_{\sigma^{-1}(n)}
\]
in $E^{i_{\sigma^{-1}(1)}} \otimes E^{i_{\sigma^{-1}(2)}} \otimes \cdots \otimes E^{i_{\sigma^{-1}(n)}}$. Here, the sign, known as the Koszul sign rule, is needed to make the action compatible with differentials. In the main body of this paper, we need the following computation. 

\begin{lem} \label{lem:expSympow}
    Suppose that $E$ is concentrated in two degrees $[a,a+1]$ and both $E^a$ and $E^{a+1}$ are of rank $1$. For each $n \geq 1$, $\Sym^n(E)$ is given as follows. 
    \begin{enumerate}
        \item When $a$ is even, $\Sym^n(E) = \lbrack (E^a)^{\otimes n} \xrightarrow{d^a \otimes \id} E^{a+1} \otimes (E^{a})^{\otimes (n-1)} \rbrack [-na]$. 
        \item When $a$ is odd, $\Sym^n(E) = \lbrack E^a \otimes (E^{a+1})^{\otimes (n-1)} \xrightarrow{d^a \otimes \id} (E^{a+1})^{\otimes n} \rbrack [-n(a+1) + 1]$. 
    \end{enumerate}
\end{lem}
\begin{proof}
    Since $\Lambda$ is in characteristic $0$, the $\Sigma_n$-coinvariant is canonically isomorphic to the $\Sigma_n$-invariant. We will compute the $\Sigma_n$-invariant of the $d$-th term 
    \[
        \bigoplus_{i_1 + i_2 + \cdots + i_n = d} E^{i_1} \otimes E^{i_2} \otimes \cdots \otimes E^{i_n}
    \]
    of $E^{\otimes n}$ for each $d$. Let $a = (a_{i_\bullet})$ be an $\Sigma_n$-invariant element of the $d$-th term. By assumption, $E^{i_1} \otimes E^{i_2} \otimes \cdots \otimes E^{i_n}$ is of rank $1$ for every $i_\bullet$. If there are two distinct indices $1 \leq k < \ell \leq n$ such that $i_k$ and $i_\ell$ are odd, we have $a_{i_\bullet} = 0$ by using $\sigma_{(k,l)} a = a$ for the transposition $\sigma_{(k,l)}$ between $k$ and $\ell$ since $\Lambda$ is in characteristic $0$ and $i_k = i_\ell$. 

    When $a$ is even (resp.\ odd), the above argument shows that $\Sym^n(E)$ is concentrated in two degrees $[na, na + 1]$ (resp.\ $[n(a+1) - 1, n(a+1)]$). The term for each degree can be computed explicitly as in the statement. 
\end{proof}

\section{Computation of the Weil-Deligne group cohomology}
\label{app:WeilDeligne_cohomology}

In this section, we explain a computation of the Weil-Deligne group cohomology. Recall
\[
    \WD_{F} = \bb{G}_a \rtimes \und{W_F} ,\quad 
    \ID_F = \bb{G}_a \times \und{I_F} \subset \WD_F. 
\]
Here, the conjugation action of $\sigma \in W_F$ on $\bb{G}_a$ is the multiplication by $\lvert \sigma \rvert$. For a derived $\Lambda$-scheme $S$, we write
\[
    \WD_{F,S} = \WD_F \times S,\qquad
    \ID_{F,S} = \ID_F \times S,\qquad
    I_{F,S} = \und{I_F} \times S
\]
and similarly for $\bb{G}_{a,S}$ and $\und{\bb{Z}}_S$. These group schemes are all flat over the base, but one needs to be careful about the fact that $\WD_F$ is not affine, of finite type, nor quasi-compact. 

For a flat group scheme $G$ over $\Lambda$ and a derived $\Lambda$-scheme $S$, we set
\[
    [S / G] = [\ast / G] \times S, 
\]
and for $\cl{M} \in \QCoh([S / G])$, let 
\[
    R\Gamma(G, \cl{M}) = R\Gamma([S/G], \cl{M}) \in \QCoh(S)
\]
be the derived pushforward of $\cl{M}$ along $[S / G] \to S$. 

Since $G$ is flat, $S \to [S / G]$ is a flat cover. Then, $\QCoh([S / G])$ is naturally equipped with a $t$-structure so that the pullback along $S \to [S / G]$ is $t$-exact. Let 
\[
    \QCoh([S / G])^+ \subset \QCoh([S / G])
\]
be the subcategory consisting of bounded below objects. 

We begin with a general formalism for extensions of group schemes.

\begin{lem}\label{prop:group_exact_sequence_WD}
    Consider a semiproduct of flat group schemes over $\Lambda$
    \[
        G = N \rtimes Q
    \]
    with $N$ quasi-compact and quasi-separated. Let $S$ be a derived $\Lambda$-scheme and let 
    \[
        \pi_S \colon [S / G] \longrightarrow [S / Q]. 
    \]
    For every $\cl{M} \in \QCoh([S / G])^+$, there is an isomorphism
    \[
        R\pi_{S*}\cl{M}\vert_{S} \cong R\Gamma(N, \cl{M})
    \]
    so that we have an isomorphism 
    \[
        R\Gamma(G, \cl{M}) \cong R\Gamma(Q, R\Gamma(N, \cl{M})). 
    \]
\end{lem}

\begin{proof}
    Since we have the Cartesian diagram
    \begin{center}
        \begin{tikzcd}
            \lbrack S / N \rbrack \ar[r] \ar[d] & \lbrack S / G \rbrack \ar[d] \\
            S \ar[r] & \lbrack S / Q \rbrack, 
        \end{tikzcd}
    \end{center}
    it is enough to prove the base change along this diagram for bounded below quasi-coherent sheaves. The key property here is that horizontal arrows are flat and $[S / G]$ admits a quasi-compact and quasi-separated flat cover
    \[
        [S / Q] \to [S / G]
    \]
    that is a section of $[S / G] \to [S / Q]$. Then, the same proof as in \cite[Proposition 2.3.2]{GR17I} works for the desired base change. 
\end{proof}

We now specialize to the cases $\ID_F = \bb{G}_a \times \und{I_F}$ and $\WD_F = \ID_F \rtimes \und{\bb{Z}}$. Then, we get
\begin{equation} \label{eq:WD_cohomology}
     R\Gamma(\WD_F, \cl{M}) \cong R\Gamma(\bb{Z}, R\Gamma(\bb{G}_a, R\Gamma(I_F, \cl{M}))) \quad (\cl{M} \in \QCoh([S / G_S])^+). 
\end{equation}
We first treat the cohomology of the inertia group. 

\begin{prop}\label{lem:IF_exact}
    Let $S$ be a derived $\Lambda$-scheme and let $\cl{M} \in \QCoh([S/\und{I_F}])$.
    \begin{enumerate}
        \item The functor $\cl{M} \mapsto R\Gamma(I_F, \cl{M})$ is $t$-exact and we have
            \[
                \bb{H}^iR\Gamma(I_F, \cl{M}) \cong \bb{H}^i(\cl{M}_S)^{I_F}
            \]
            in $\QCoh(S)^{\heart}$ for every $i$. 
        \item For each open subgroup $K \subset I_F$, let 
            \[
                \cl{M}_S^{K} = R\Gamma(K, \cl{M}) \in \QCoh(S). 
            \]
            Then, $\cl{M}_S \cong \colim_K \cl{M}_S^K$ and each $\cl{M}_S^K$ is a direct summand of $\cl{M}_S$. 
    \end{enumerate}
\end{prop}
\begin{proof}
    Since $\QCoh(S)$ and $\QCoh([S / \und{I_F}])$ are left $t$-complete and $R\Gamma(K, -)$ commutes with limits, all assertions can be checked after passing to the Postnikov limit. In particular, we may assume $\cl{M} \in \QCoh([S / \und{I_F}])^+$. 

    Since the pullback along $[S / \und{I_F}] \to S$ is $t$-exact, the functor $\cl{M} \mapsto R\Gamma(I_F, \cl{M})$ is left $t$-exact. Thus, for (1), it is enough to show 
    \[
        R\Gamma(I_F, \cl{M}) \cong \cl{M}_S^{I_F} \quad (\cl{M} \in \QCoh([S/\und{I_F}])^\heart). 
    \]
    We may assume that $S = \Spec(R)$ is affine and $M = \cl{M}_S$ is a static $\pi_0(R)$-module. Since $\Gamma(\und{I_F}, \cl{O}) = C^\infty(I_F, \Lambda)$, the $I_F$-action on $M$ is given by the comultiplication $M \to M \otimes C^\infty(I_F, \Lambda)$, and by taking the \v{C}ech nerve along $S \to [S / \und{I_F}]$, we get
    \[
        R\Gamma(I_F, \cl{M}) = [M \to M \otimes C^\infty(I_F, \Lambda) \to M \otimes C^\infty(I_F, \Lambda) \otimes C^\infty(I_F, \Lambda) \to \cdots]. 
    \]
    In particular, the $I_F$-action on $M$ is smooth and we have 
    \[
        R\Gamma(I_F, M) \cong \colim_{K \subset I_F} R\Gamma(I_F / K, M^K)
    \]
    since $C^\infty(I_F, \Lambda) = \colim_{K} C^\infty(I_F / K, \Lambda)$, where $K$ runs through normal open subgroups of $I_F$. Since $\Lambda$ is of characteristic $0$, $R\Gamma(I_F / K, M^K) \cong M^{I_F}$. Thus, we get (1). 

    For (2), we already see that the $I_F$-action on $\bb{H}^i(\cl{M}_S)$ is smooth and $\bb{H}^iR\Gamma(K, \cl{M}) \cong \bb{H}^i(\cl{M}_S)^{K}$. For each open subgroup $K \subset I_F$, let 
    \[
        \cl{M}_K = \cl{M}\vert_{[S / K]},
    \]
    and for a normal open subgroup $K' \subset K$, let 
    \[
        q_{K' \subset K} \colon [S / K'] \to [S / K]
    \]
    be a finite \'{e}tale cover. The unit $\cl{M}_K \to q_{K'\subset K*} \cl{M}_{K'}$ induces 
    \[
        \cl{M}_S^K = R\Gamma(K, \cl{M}_K) \to R\Gamma(K, q_{K'\subset K*} \cl{M}_{K'}) \cong \cl{M}_S^{K'}. 
    \]
    It induces $\bb{H}^i(\cl{M}_S)^{K} \subset \bb{H}^i(\cl{M}_S)^{K'}$ on the $i$-th cohomology. Thus,  
    \[
        \colim_K \cl{M}_S^K \to \cl{M}_S
    \]
    is a quasi-isomorphism by the smoothness of the $I_F$-action on $\bb{H}^i(\cl{M}_S)$. For the second claim of (2), it is enough to show that $\cl{M}_S^K \to \cl{M}_S^{K'}$ splits canonically. Since $\Lambda$ is in characteristic $0$ and $q_{K' \subset K}$ is finite \'{e}tale, the averaged trace provides a section $q_{K'\subset K*} \cl{M}_{K'} \to \cl{M}_K$ of the unit. By taking $R\Gamma(K, -)$, we get the claim. 
\end{proof}

Next, we treat the well-known description of the $\bb{G}_a$-cohomology in
characteristic $0$.

\begin{prop}\label{lem:Ga_cohomology}
    Let $S$ be a derived $\Lambda$-scheme and let $\cl{M} \in \QCoh([S/\bb{G}_{a}])$. Let 
    \[
        \nabla \colon \cl{M}_S \to \cl{M}_S \otimes \Lambda[T]
    \]
    be the comultiplication of $\cl{M}$ and let $N \in \End(\cl{M}_S)$ be the endomorphism such that
    \[
        \nabla \equiv 1 + NT \pmod{T^2}. 
    \]
    Then, there is a natural quasi-isomorphism
    \[
        R\Gamma(\bb{G}_{a},\cl{M}) \cong \bigl[\cl{M}_S \xrightarrow{N} \cl{M}_S\bigr]. 
    \]
    In particular, the functor $\cl{M} \mapsto R\Gamma(\bb{G}_a, \cl{M})$ has cohomological dimension $1$. 
\end{prop}

\begin{proof}
    Since $\QCoh(S)$ and $\QCoh([S / \bb{G}_a])$ are left $t$-complete and $R\Gamma(\bb{G}_a, -)$ commutes with limits, all assertions can be checked after passing to the Postnikov limit. In particular, we may assume $\cl{M} \in \QCoh([S / \bb{G}_a])^+$. 

    By taking the \v{C}ech nerve along $S \to [S / \bb{G}_a]$, we have 
    \begin{equation} \label{eq:Ga_cohomology}
        \begin{tikzcd}
             R\Gamma(\bb{G}_a,\cl{M}) \cong \Tot \bigl(\cl{M}_S \ar[r, yshift = -0.5ex, shorten = 8pt] \ar[r, yshift = 0.5ex, shorten = 8pt] \hspace{-8pt} & \hspace{-8pt} \cl{M}_S \otimes \Lambda[T] \ar[r, yshift = -0.8ex, shorten = 8pt] \ar[r, shorten = 8pt] \ar[r, yshift = 0.8ex, shorten = 8pt] \hspace{-8pt} & \hspace{-8pt} \cdots \bigr) 
        \end{tikzcd}
    \end{equation}
    %\begin{equation}
    %    R\Gamma(\bb{G}_{a},\cl{M}) \cong \Tot(\cl{M}_S \xrightarrow{\nabla - \id \otimes 1} \cl{M}_S \otimes \Lambda[T] \to \cl{M}_S \otimes \Lambda[T] \otimes \Lambda[T] \to \cdots). \label{eq:Ga_cohomology}
    %\end{equation}
    The stupid truncation and the projection to the coordinate of $T$ provide a map
    \[
        R\Gamma(\bb{G}_a, \cl{M}) \to \bigl[\cl{M}_S \xrightarrow{N} \cl{M}_S\bigr]. 
    \]
    We would like to show that it is a quasi-isomorphism. Both functors are left $t$-exact since the pullback along $[S / \bb{G}_a] \to S$ is $t$-exact. Thus, we may assume that $S = \Spec(R)$ is affine and $M = \cl{M}_S$ is a static $\pi_0(R)$-module. Since we work in characteristic $0$, $N \in \End(M)$ is locally nilpotent with $\nabla = \exp(NT)$, and the computation of \eqref{eq:Ga_cohomology} is well-known (e.g. \cite[Theorem 5.1]{Hoc61}). 
\end{proof}

\begin{cor}\label{cor:IDF_cohomology}
    Let $S$ be a derived $\Lambda$-scheme and let $\cl{M} \in \QCoh([S/\ID_{F}])$. Then, we have
    \[
        R\Gamma(\ID_{F},\cl{M}) \cong \bigl[\cl{M}_S^{I_F} \xrightarrow{N} \cl{M}_S^{I_F}\bigr],
    \]
    where $N$ is the locally nilpotent endomorphism induced by the $\bb{G}_{a}$-action.
\end{cor}

Finally, we treat the group cohomology of $\bb{Z}$. In the following, a generator of $\bb{Z} \subset W_F$ is denoted by $\Fr$ for consistency of notation. 

\begin{prop} \label{prop:Z_cohomology}
    Let $S$ be a derived $\Lambda$-scheme and let $\cl{M} \in \QCoh([S/\und{\bb{Z}}])$. Let $\Fr \in \Aut(\cl{M}_S)$ denote the action of $\Fr \in \bb{Z}$. Then, there is a natural quasi-isomorphism
    \[
        R\Gamma(\bb{Z},\cl{M}) \cong \bigl[\cl{M}_S \xrightarrow{\Fr - 1} \cl{M}_S\bigr]. 
    \]
    In particular, the functor $\cl{M} \mapsto R\Gamma(\bb{Z}, \cl{M})$ has cohomological dimension $1$. 
\end{prop}
\begin{proof}
    Since $\QCoh(S)$ and $\QCoh([S / \bb{Z}])$ are left $t$-complete and $R\Gamma(\bb{Z}, -)$ commutes with limits, all assertions can be checked after passing to the Postnikov limit. In particular, we may assume $\cl{M} \in \QCoh([S / \bb{Z}])^+$. 

    By taking the \v{C}ech nerve along $S \to [S / \bb{Z}]$, we have 
    \begin{equation} \label{eq:Z_cohomology}
        \begin{tikzcd}
             R\Gamma(\bb{Z},\cl{M}) \cong \Tot \bigl(\cl{M}_S \ar[r, yshift = -0.5ex, shorten = 8pt] \ar[r, yshift = 0.5ex, shorten = 8pt] \hspace{-8pt} & \hspace{-8pt}  \Map(\bb{Z}, \cl{M}_S) \ar[r, yshift = -0.8ex, shorten = 8pt] \ar[r, shorten = 8pt] \ar[r, yshift = 0.8ex, shorten = 8pt] \hspace{-8pt} & \hspace{-8pt} \cdots \bigr)
        \end{tikzcd}
    \end{equation}
    %\begin{equation}
    %    R\Gamma(\bb{Z},\cl{M}) \cong [ \cl{M}_S \to \Map(\bb{Z}, \cl{M}_S) \to \Map(\bb{Z} \times \bb{Z}, \cl{M}_S) \to \cdots ]. \label{eq:Z_cohomology}
    %\end{equation}
    %Here, the evaluation at $\Fr^n$
    %\[
    %    \cl{M}_S \to \Map(\bb{Z}, \cl{M}_S) \xrightarrow{\ev_{\Fr^n}} \cl{M}_S
    %\]
    %is given by $\Fr^n - 1$ for each $n$. 
    Now, $\Fr \in \Aut(\cl{M}_S)$ equips $\cl{M}_S$ with the $\Lambda[\bb{Z}]$-module structure. Then, \eqref{eq:Z_cohomology} implies
    \begin{equation*}
        \begin{tikzcd}
            R\Gamma(\bb{Z},\cl{M}) \cong R\Hom_{\Lambda[\bb{Z}]}\bigl(\Tot(\cdots \ar[r, yshift = -0.8ex, shorten = 8pt] \ar[r, shorten = 8pt] \ar[r, yshift = 0.8ex, shorten = 8pt] \hspace{-8pt} & \hspace{-8pt} \Lambda[\bb{Z} \oplus \bb{Z}] \ar[r, yshift = -0.5ex, shorten = 8pt] \ar[r, yshift = 0.5ex, shorten = 8pt] \hspace{-8pt} & \hspace{-8pt} \Lambda[\bb{Z}]), \cl{M}_S \bigr) \cong R\Hom_{\Lambda[\bb{Z}]}(\Lambda, \cl{M}_S). 
        \end{tikzcd}
    \end{equation*}
    By the resolution $\Lambda \cong [\Lambda[\bb{Z}] \xrightarrow{\Fr - 1} \Lambda[\bb{Z}]][1]$ as a $\Lambda[\bb{Z}]$-module, we get 
    \[
        R\Gamma(\bb{Z},\cl{M}) \cong R\Hom_{\Lambda[\bb{Z}]}(\Lambda, \cl{M}_S) \cong [\cl{M}_S \xrightarrow{\Fr - 1} \cl{M}_S]. 
    \]
\end{proof}

We can now assemble the previous lemmas to compute the cohomology of $\WD_{F}$. Let $\cl{M} \in \QCoh([S/\WD_{F}])$ and let $\cl{M}_S$ denote the underlying quasi-coherent sheaf on $S$. Then, $\cl{M}_S$ carries a smooth $W_F$-action, and a $\bb{G}_{a}$-action, which is given by a locally nilpotent endomorphism (see \Cref{lem:Ga_cohomology})
\[
    N \in \End(\cl{M}_S).
\]
\begin{prop}\label{prop:WD_cohomology}
    Let $S$ be a derived $\Lambda$-scheme and let $\cl{M} \in \QCoh([S/\WD_{F}])$. Then, there is a natural quasi-isomorphism
    \[
        R\Gamma(\WD_{F},\cl{M}) \cong
        \Bigl[
            \cl{M}_S^{I_F}
                \xrightarrow{(\Fr-1,\;N)}
            \cl{M}_S^{I_F} \oplus \cl{M}_S^{I_F}
                \xrightarrow{(N,\;1-q\Fr)}
            \cl{M}_S^{I_F}
        \Bigr],
    \]
    where
    \begin{itemize}
        \item $\Fr$ is a Frobenius element in $W_F$, acting on
        $\cl{M}_S^{I_F}$ via the $W_F$-action on $\cl{M}_S$, and
        \item $q$ is the cardinality of the residue field of $F$. 
    \end{itemize}
    In particular, $R\Gamma(\WD_{F},-)$ has cohomological dimension
    $2$, commutes with small colimits, and satisfies the base change and the projection formula along $S$.
\end{prop}

\begin{proof}
    Since $\QCoh(S)$ and $\QCoh([S / \WD_F])$ are left $t$-complete and $R\Gamma(\WD_F, -)$ commutes with limits, the first assertion can be checked after passing to the Postnikov limit. In particular, we may assume $\cl{M} \in \QCoh([S / \WD_F])^+$. By \eqref{eq:WD_cohomology} and \Cref{cor:IDF_cohomology}, 
    \[
        R\Gamma(\WD_F, \cl{M}) \cong R\Gamma(\bb{Z}, [\cl{M}_S^{I_F} \xrightarrow{N} \cl{M}_S^{I_F}]). 
    \]
    The action of $\Fr$ is given by 
    \begin{center}
        \begin{tikzcd}
            \cl{M}_S^{I_F} \ar[r, "N"] \ar[d, "\Fr"] & \cl{M}_S^{I_F} \ar[d, "q\Fr"] \\
            \cl{M}_S^{I_F} \ar[r, "N"] & \cl{M}_S^{I_F}. 
        \end{tikzcd}
    \end{center}
    Thus, the desired description follows from \Cref{prop:Z_cohomology}. Then, the cohomological dimension bound follows from this description. 

    By the description of $\bb{H}^iR\Gamma(I_F, -)$ in \Cref{lem:IF_exact} (1), the functor $\cl{M} \mapsto \cl{M}_S^{I_F}$ is exact and commutes with filtered colimits, so it commutes with small colimits. Thus, $R\Gamma(\WD_F, -)$ also commutes with small colimits. 

    For every $\cl{E} \in \QCoh(S)$, we have a natural map
    \[
        \cl{M}_S^{I_F} \otimes \cl{E} \to (\cl{M}_S \otimes \cl{E})^{I_F}. 
    \]  
    To show that it is an isomorphism, we may assume that $S$ is affine, so that $\QCoh(S) = \Ind(\Perf(S))$. Both sides commutes with small colimits in $\cl{E}$, so we may assume that $\cl{E} \in \Perf(S)$. Then, $\cl{E}$ is dualizable, so we have $(\cl{M}_S \otimes \cl{E})^{I_F} \cong \cl{M}_S^{I_F} \otimes \cl{E}$. Then, it implies the base change and the projection formula for the functor $\cl{M} \mapsto \cl{M}_S^{I_F}$, so the same holds for $R\Gamma(\WD_F, -)$.   
\end{proof}

\section{Classical degree normalization on algebraic curves}
\label{sec:classical_degree_normalization}

Let $\Sigma$ be a smooth projective curve over a finite field $\bb{F}_q$ with genus $g$ and fix an isomorphism $\Qla \cong \bb{C}$. In this section, we review the Iwasawa-Tate period and its normalization in the geometric setting of \cite{BZSV}. 

Let us first recast Tate's thesis in our context. Take $G = \Gm$ over a global field $F = K(\Sigma)$, and let $\chi \colon \bb{A}_F^\times/F^\times \to \bb{C}^\times$ be an id\`{e}le class character. Consider the Iwasawa-Tate dual pair
\[
    \bigl( X = \bb{A}^1 \circlearrowleft \Gm = G,\;\; \widehat{X} = \bb{A}^1 \circlearrowleft \Gm = \widehat{G} \bigr)
\]
with the standard scaling actions. On the $\cl{A}$-side, the unnormalized $X$-period of $\chi$ is the classical Iwasawa-Tate zeta integral
\begin{equation}\label{eq:period_integral_tate_case}
    P_X(\chi):=\anbr{P_X,\chi}_{[\bb{G}_m]}
    \;=\; \int_{\bb{A}_F^\times} \Phi(x)\,\chi(x)\,d^\times x,
\end{equation}
for a suitable Schwartz-Bruhat function $\Phi$ on $\bb{A}_F$ (for instance the characteristic function of the integral ad\`eles $\bb{O}_{\bb{A}_F}$) as in Tate's thesis \cite{Tate1967}. 

The normalized period sheaf $P_X^{\norm}$ of \cite{BZSV} (and of this paper) corresponds analytically to inserting the half-norm character in the integral. More precisely, after choosing a square root $\sqrt{q}$ and the resulting half-norm character $\lvert \cdot \rvert^{1/2}$, the normalized $X$-period is
\[
    P_X^{\norm}(\chi)
    \;=\; c_0 \int_{\bb{A}_F^\times} |x|^{1/2}\,\Phi(x)\,\chi(x)\,d^\times x,
\]
where $c_0$ is a nonzero constant depending only on the global normalization (the $\beta_X$-twist in \cite[Eq.\ (10.9)]{BZSV}). This extra factor $|x|^{1/2}$ reflects exactly the half-volume factor. 

Period sheaves are equipped with Frobenius actions, and via the function-sheaf dictionary, $P_X$ is regarded as a period function on $\Bun_G(\bb{F}_q)$: 
\[ 
    P_X \colon \Bun_G(\bb{F}_q) \ra \bb{C}, \quad
    L\mapsto q^{h^0(\Sigma, L \otimes K^{1/2})}, 
\]
where $h^i$ is the dimension of the $i$-th cohomology and $K^{1/2}$ is a square-root of the canonical bundle of $\Sigma$. Then, the normalized period function is
\[ 
    P_X^\norm \colon \Bun_G(\bb{F}_q)\ra \bb{C} ,\quad
    L\mapsto q^{h^0(\Sigma, L \otimes K^{1/2}) - \frac{1}{2}\deg (L\otimes K^{1/2})}
\]
(see \cite[10.6.2]{BZSV}). The Riemann-Roch theorem implies 
\[ 
    h^0(\Sigma, L\otimes K^{1/2}) - h^0(\Sigma, L^{-1}\otimes K^{1/2}) = \deg L. 
\]
Hence this shows $P_X^\norm(L) = P_X^\norm(L^{-1})$, which is what yields the classical functional equation once identifying terms with special $L$-values.  

\begin{exa}
    Everything can be made extremely explicit when $\Sigma=\bb{P}^1$. Any line bundle $L$ can be written as $L\cong \cl{O}(d)$ for some $d \in \bb{Z}$. Since $K^{1/2} \cong \cl{O}(-1)$, we have 
    \[ 
        P_X \colon \cl{O}(d) \mapsto q^{\max(0, d)}, \quad 
        P_X^\norm \colon \cl{O}(d) \mapsto q^{\frac{\lvert d \rvert + 1}{2}}. 
    \]
\end{exa}

Working with $\Sigma=\bb{P}^1$ is not a big loss, since it captures the generic behavior of these period functions. In general, the Riemann-Roch theorem and Serre duality give
\[
    h^0(\Sigma, L)- h^1(\Sigma,L) = \deg L - g + 1, 
\]
\[
    h^0(\Sigma, L)=0 \quad (\deg L< 0) ,\quad h^1(\Sigma, L)=0 \quad (\deg L \ge 2g -1). 
\]
Thus, generically (outside the region $1 - g \leq \deg L \leq g - 1$), we have
\[ 
    P_X^\norm \sim q^{\frac{\lvert \deg L \rvert +1-g}{2}}, 
\]
which is the asymptotic behavior discussed in \cite[Example 14.8.3]{BZSV}. 

\addtocontents{toc}{\protect\setcounter{tocdepth}{2}}

\renewcommand\bibfont{\footnotesize}
\printbibliography

\end{document}

%% file: packages.tex
\usepackage{amsmath,amsthm,amssymb,mathrsfs,stmaryrd,color,mathtools}

\usepackage[all]{xy}
\usepackage{url}
\usepackage{tikz-cd}

\usepackage[utf8]{inputenc}
\usepackage[T1]{fontenc}

\usepackage{relsize}
\usepackage[bbgreekl]{mathbbol}
\usepackage{amsfonts}

\DeclareSymbolFontAlphabet{\mathbb}{AMSb}
\DeclareSymbolFontAlphabet{\mathbbl}{bbold}

\usepackage{enumitem}

\usepackage[citestyle=alphabetic,bibstyle=alphabetic,backend=bibtex, maxalphanames=10, maxnames=10, url=true, doi=false]{biblatex}
\DeclareFieldFormat{postnote}{#1}
\DeclareFieldFormat{multipostnote}{#1}

%% file: commands.tex
\usepackage[margin=1in,footskip=.5in]{geometry}
\usepackage[colorlinks=true,hyperindex, linkcolor=magenta, pagebackref=false, citecolor=cyan,pdfpagelabels]{hyperref}
\usepackage[capitalize]{cleveref}

\newtheorem{thm}{Theorem}[section]
\newtheorem{thm2}{Theorem}
\newtheorem{conj}[thm]{Conjecture}
\newtheorem{conj2}[thm2]{Conjecture}
\newtheorem{conv}[thm]{Convention}

\newtheorem{prop}[thm]{Proposition}
\newtheorem{prop2}[thm2]{Proposition}

\newtheorem{cor}[thm]{Corollary}
\newtheorem{lem}[thm]{Lemma}

\theoremstyle{definition}
\newtheorem{defi}[thm]{Definition}
\newtheorem{defi2}[thm2]{Definition}
\newtheorem{rmk}[thm]{Remark}
\newtheorem{exa}[thm]{Example}

\newtheorem{const}[thm]{Construction}
\newtheorem{ass}[thm]{Assumption}

\numberwithin{equation}{section}

\newcommand{\shear}{{\mathbin{\mkern-6mu\fatslash}}}

\newcommand{\bb}[1]{\mathbb{#1}}
\newcommand{\cl}[1]{{\mathcal{#1}}}

\newcommand{\mfr}[1]{{\mathfrak{#1}}}
\newcommand{\mrm}[1]{{\mathrm{#1}}}

\newcommand{\ov}[1]{{\overline{#1}}}
\newcommand{\und}[1]{{\underline{#1}}}

\newcommand{\Ql}{{\mathbb{Q}_\ell}}

\newcommand{\Qla}{{\overline{\mathbb{Q}}_\ell}}
\newcommand{\Qlax}{{\overline{\mathbb{Q}}_\ell^\times}}

\newcommand{\LT}{{{}^LT}}
\newcommand{\LH}{{{}^LH}}
\newcommand{\LM}{{{}^LM}}
\newcommand{\LP}{{{}^LP}}

\newcommand{\LG}{{{}^LG}}
\newcommand{\LGX}{{{}^LG_X}}
\newcommand{\Gm}{{\mathbb{G}_m}}

\newcommand{\vStack}{{\operatorname{vStack}}}
\newcommand{\qcqsvShf}{{\operatorname{vShf}}^{\mathrm{qcqs}}}
\newcommand{\arcStack}{{\operatorname{arcStack}'}}
\newcommand{\CAlg}{{\operatorname{CAlg}}}

\newcommand{\qcqs}{{\operatorname{qcqs}}}
\newcommand{\Perf}{{\operatorname{Perf}}}
\newcommand{\Perft}{{\operatorname{Perf}'}}
\newcommand{\Perfa}{{\operatorname{Perf}}^{\mathrm{aff}}}
\newcommand{\Perfs}{{\operatorname{Perf}}^{\mathrm{std}}}
\newcommand{\lis}{{\operatorname{lis}}}
\newcommand{\mot}{{\operatorname{mot}}}
\newcommand{\oc}{{\operatorname{oc}}}
\newcommand{\arc}{{\operatorname{arc}}}

\newcommand{\alg}{{\operatorname{alg}}}
\newcommand{\Ind}{{\operatorname{Ind}}}
\newcommand{\cInd}{{\operatorname{cInd}}}

\newcommand{\Gal}{{\operatorname{Gal}}}
\newcommand{\Bun}{{\operatorname{Bun}}}

\newcommand{\Corr}{{\operatorname{Corr}}}
\newcommand{\Pres}{{\operatorname{Pr}}}
\newcommand{\GL}{{\operatorname{GL}}}
\newcommand{\fin}{{\operatorname{fin}}}
\newcommand{\std}{{\operatorname{std}}}
\newcommand{\colim}{{\operatorname{colim}}}

\newcommand{\sm}{{\operatorname{sm}}}

\newcommand{\Par}{{\operatorname{Par}}}
\newcommand{\Rep}{{\operatorname{Rep}}}
\newcommand{\univ}{{\operatorname{univ}}}
\newcommand{\triv}{{\operatorname{triv}}}

\newcommand{\norm}{{\operatorname{norm}}}
\newcommand{\ren}{{\operatorname{ren}}}
\newcommand{\tp}{{\operatorname{top}}}
\newcommand{\BC}{{\mathcal{BC}}}
\newcommand{\Sec}{{\operatorname{Sec}}}
\newcommand{\Spd}{{\operatorname{Spd}}}
\newcommand{\Spec}{{\operatorname{Spec}}}
\newcommand{\Spa}{{\operatorname{Spa}}}
\newcommand{\Spf}{{\operatorname{Spf}}}
\newcommand{\Div}{{\operatorname{Div}}}
\newcommand{\Cat}{{\operatorname{Cat}}}
\newcommand{\ev}{{\operatorname{ev}}}

\newcommand{\pr}{{\operatorname{pr}}}
\newcommand{\st}{{\operatorname{st}}}
\newcommand{\id}{{\operatorname{id}}}
\newcommand{\Aut}{{\operatorname{Aut}}}
\newcommand{\fib}{{\operatorname{fib}}}

\newcommand{\Mir}{{\operatorname{Mir}}}
\newcommand{\spec}{{\operatorname{spec}}}
\newcommand{\Fr}{{\operatorname{Fr}}}

\newcommand{\sss}{{\operatorname{ss}}}
\newcommand{\Lie}{{\operatorname{Lie}}}
\newcommand{\Irr}{{\operatorname{Irr}}}
\newcommand{\Eis}{{\operatorname{Eis}}}
\newcommand{\Tot}{{\operatorname{Tot}}}
\newcommand{\Sym}{{\operatorname{Sym}}}
\newcommand{\ord}{{\operatorname{ord}}}
\newcommand{\Sp}{{\operatorname{Sp}}}
\newcommand{\Ani}{{\operatorname{Ani}}}
\newcommand{\ID}{{\operatorname{ID}}}
\newcommand{\WD}{{\operatorname{WD}}}
\newcommand{\End}{{\operatorname{End}}}
\newcommand{\Map}{{\operatorname{Map}}}
\newcommand{\Hom}{{\operatorname{Hom}}}
\newcommand{\Ad}{{\operatorname{Ad}}}

\newcommand{\cyc}{{\operatorname{cyc}}}
\newcommand{\Shv}{\operatorname{Shv}}
\newcommand{\IndCoh}{{\operatorname{IndCoh}}}
\newcommand{\Coh}{{\operatorname{Coh}}}
\newcommand{\QCoh}{{\operatorname{QCoh}}}

\newcommand{\grp}{{\operatorname{grp}}}
\newcommand{\der}{{\operatorname{der}}}
\newcommand{\ad}{{\operatorname{ad}}}
\newcommand{\Res}{{\operatorname{Res}}}
\newcommand{\Eq}{{\operatorname{Eq}}}
\newcommand{\bas}{{\operatorname{bas}}}
\newcommand{\FS}{{\operatorname{FS}}}
\newcommand{\op}{{\operatorname{op}}}
\newcommand{\clas}{{}^{\operatorname{cl}}}
\newcommand{\DStk}{{\operatorname{DStk}}}

\newcommand{\LLKE}{{}^L{\mathrm{LKE}}}

\newcommand{\heart}{\heartsuit}

\newcommand{\Ext}{\mathrm{Ext}}

\newcommand{\xra}[1]{\xrightarrow{#1}}

\newcommand{\anbr}[1]{\left\langle #1 \right\rangle}
\newcommand{\crbr}[1]{\left \{ #1 \right \} }
\newcommand{\smbr}[1]{\left( #1 \right) }
\newcommand{\sqbr}[1]{\left [ #1 \right ] }

\newcommand{\hra}{\hookrightarrow}
\newcommand{\thra}{\twoheadrightarrow}
\newcommand{\ra}{\rightarrow}

\newcommand{\Der}{\mathrm{Der}}

%% file: refs.bib
@article {KS17,
    AUTHOR = {Knop, Friedrich and Schalke, Barbara},
     TITLE = {The dual group of a spherical variety},
   JOURNAL = {Trans. Moscow Math. Soc.},
  FJOURNAL = {Transactions of the Moscow Mathematical Society},
    VOLUME = {78},
      YEAR = {2017},
     PAGES = {187--216},
      ISSN = {0077-1554,1547-738X},
   MRCLASS = {17B22 (11F70 14L30 14M27)},
  MRNUMBER = {3738085},
MRREVIEWER = {Alexander\ Yu.\ Luzgarev},
       DOI = {10.1090/mosc/270},
}

@misc{BMO25,
      title={The B(G)-parametrization of the local Langlands correspondence}, 
      author={Alexander Bertoloni Meli and Masao Oi},
      year={2025},
      eprint={2211.13864},
      archivePrefix={arXiv},
      primaryClass={math.NT},
}

@article {PY02,
    AUTHOR = {Prasad, Gopal and Yu, Jiu-Kang},
     TITLE = {On finite group actions on reductive groups and buildings},
   JOURNAL = {Invent. Math.},
  FJOURNAL = {Inventiones Mathematicae},
    VOLUME = {147},
      YEAR = {2002},
    NUMBER = {3},
     PAGES = {545--560},
      ISSN = {0020-9910,1432-1297},
   MRCLASS = {20E42 (20G15 20G25)},
  MRNUMBER = {1893005},
MRREVIEWER = {Guy\ Rousseau},
       DOI = {10.1007/s002220100182},
}

@article {BP18,
    AUTHOR = {Beuzart-Plessis, Rapha\"{e}l},
     TITLE = {On distinguished square-integrable representations for
              {G}alois pairs and a conjecture of {P}rasad},
   JOURNAL = {Invent. Math.},
  FJOURNAL = {Inventiones Mathematicae},
    VOLUME = {214},
      YEAR = {2018},
    NUMBER = {1},
     PAGES = {437--521},
      ISSN = {0020-9910,1432-1297},
   MRCLASS = {22E50 (11F85)},
  MRNUMBER = {3858402},
MRREVIEWER = {Gergely\ Z\'abr\'adi},
       DOI = {10.1007/s00222-018-0807-z},
}

@misc{Pra15,
      title={A `relative' local Langlands correspondence}, 
      author={Dipendra Prasad},
      year={2015},
      eprint={1512.04347},
      archivePrefix={arXiv},
      primaryClass={math.NT},
}

@misc{BP25,
      title={Introduction to the relative Langlands program}, 
      author={Rapha\"{e}l Beuzart-Plessis},
      year={2025},
      eprint={2509.18062},
      archivePrefix={arXiv},
      primaryClass={math.NT},
}

@article{Kos24,
   JOURNAL = {Arithmetic Geometry. Oberwolfach Rep.},
  FJOURNAL = {Arithmetic Geometry. Oberwolfach Reports},
    VOLUME = {21},
      YEAR = {2024},
     PAGES = {1885--1887},
      title={$A$-parameters and eigensheaves}, 
      author={Teruhisa Koshikawa},
       DOI = {DOI 10.4171/OWR/2024/33},
}

@article {BZFN10,
    AUTHOR = {Ben-Zvi, David and Francis, John and Nadler, David},
     TITLE = {Integral transforms and {D}rinfeld centers in derived
              algebraic geometry},
   JOURNAL = {J. Amer. Math. Soc.},
  FJOURNAL = {Journal of the American Mathematical Society},
    VOLUME = {23},
      YEAR = {2010},
    NUMBER = {4},
     PAGES = {909--966},
      ISSN = {0894-0347,1088-6834},
   MRCLASS = {14D23 (14F05 18D10 18E30)},
  MRNUMBER = {2669705},
MRREVIEWER = {Andrei\ D.\ Halanay},
       DOI = {10.1090/S0894-0347-10-00669-7},
}

@article {DG13,
    AUTHOR = {Drinfeld, Vladimir and Gaitsgory, Dennis},
     TITLE = {On some finiteness questions for algebraic stacks},
   JOURNAL = {Geom. Funct. Anal.},
  FJOURNAL = {Geometric and Functional Analysis},
    VOLUME = {23},
      YEAR = {2013},
    NUMBER = {1},
     PAGES = {149--294},
      ISSN = {1016-443X,1420-8970},
   MRCLASS = {14A20 (14F05 14F10 18Dxx)},
  MRNUMBER = {3037900},
MRREVIEWER = {Pawel\ Sosna},
       DOI = {10.1007/s00039-012-0204-5},
}

@misc{BP23,
      title={GAGA problems for the Brauer group via derived geometry}, 
      author={Binda, Federico and Porta, Mauro},
      year={2023},
      eprint={2107.03914},
      archivePrefix={arXiv},
      primaryClass={math.AG},
}

@article {BS21,
    AUTHOR = {Bergh, Daniel and Schn\"urer, Olaf M.},
     TITLE = {Decompositions of derived categories of gerbes and of families
              of {B}rauer-{S}everi varieties},
   JOURNAL = {Doc. Math.},
  FJOURNAL = {Documenta Mathematica},
    VOLUME = {26},
      YEAR = {2021},
     PAGES = {1465--1500},
      ISSN = {1431-0635,1431-0643},
   MRCLASS = {14F08 (14A20)},
  MRNUMBER = {4334847},
MRREVIEWER = {Giulio\ Bresciani},
}

@book {BLR90,
    AUTHOR = {Bosch, Siegfried and L\"{u}tkebohmert, Werner and Raynaud,
              Michel},
     TITLE = {N\'{e}ron models},
    SERIES = {Ergebnisse der Mathematik und ihrer Grenzgebiete (3) [Results
              in Mathematics and Related Areas (3)]},
    VOLUME = {21},
 PUBLISHER = {Springer-Verlag, Berlin},
      YEAR = {1990},
     PAGES = {x+325},
      ISBN = {3-540-50587-3},
   MRCLASS = {14K15 (11G10 14L15)},
  MRNUMBER = {1045822},
MRREVIEWER = {James\ Milne},
       DOI = {10.1007/978-3-642-51438-8},
}

@misc{GaiStack,
      title={Notes on geometric Langlands: stacks}, 
      author={Dennis Gaitsgory},
      url={https://people.mpim-bonn.mpg.de/gaitsgde/GL/Stackstext.pdf}, 
      shorthand = {GL:Stacks}, 
}

@misc{DAG,
      title={Derived Algebraic Geometry}, 
      author={Jacob Lurie},
      url={https://people.math.harvard.edu/~lurie/papers/DAG.pdf}, 
      shorthand = {DAG}, 
}

@misc{DAGXIV,
      title={Derived Algebraic Geometry XIV: Representability Theorems}, 
      author={Jacob Lurie},
      url={https://www.math.ias.edu/~lurie/papers/DAG-XIV.pdf}, 
      shorthand = {DAG XIV}, 
}

@misc{Lurie2018SAG,
  author       = {Lurie, Jacob},
  title        = {Spectral Algebraic Geometry},
  year         = {2018},
  month        = {February},
  howpublished = {Unpublished manuscript},
  note         = {Under construction. Available at \url{https://www.math.ias.edu/~lurie/papers/SAG-rootfile.pdf}}
}

@article {AG15,
    AUTHOR = {Arinkin, Dima and Gaitsgory, Dennis},
     TITLE = {Singular support of coherent sheaves and the geometric
              {L}anglands conjecture},
   JOURNAL = {Selecta Math. (N.S.)},
  FJOURNAL = {Selecta Mathematica. New Series},
    VOLUME = {21},
      YEAR = {2015},
    NUMBER = {1},
     PAGES = {1--199},
      ISSN = {1022-1824,1420-9020},
   MRCLASS = {14D24 (14A20 14F05 22E57)},
  MRNUMBER = {3300415},
MRREVIEWER = {Richard\ P.\ Thomas},
       DOI = {10.1007/s00029-014-0167-5},
}

@article {SV17,
    AUTHOR = {Sakellaridis, Yiannis and Venkatesh, Akshay},
     TITLE = {Periods and harmonic analysis on spherical varieties},
   JOURNAL = {Ast\'erisque},
  FJOURNAL = {Ast\'erisque},
    NUMBER = {396},
      YEAR = {2017},
     PAGES = {viii+360},
      ISSN = {0303-1179,2492-5926},
      ISBN = {978-2-85629-871-8},
   MRCLASS = {22E50 (11F67)},
  MRNUMBER = {3764130},
MRREVIEWER = {Arnab\ Mitra},
}

@misc{HL25,
      title={Torsion Vanishing for Some Shimura Varieties}, 
      author={Linus Hamann and Si Ying Lee},
      year={2025},
      eprint={2309.08705v4},
      archivePrefix={arXiv},
      primaryClass={math.NT},
}

@article {Gai13,
    AUTHOR = {Gaitsgory, Dennis},
     TITLE = {ind-coherent sheaves},
   JOURNAL = {Mosc. Math. J.},
  FJOURNAL = {Moscow Mathematical Journal},
    VOLUME = {13},
      YEAR = {2013},
    NUMBER = {3},
     PAGES = {399--528, 553},
      ISSN = {1609-3321,1609-4514},
   MRCLASS = {14F05 (14D24)},
  MRNUMBER = {3136100},
MRREVIEWER = {Cristian\ V.\ Anghel},
       DOI = {10.17323/1609-4514-2013-13-3-399-528},
}

@article {ALB25,
    AUTHOR = {Ansch\"utz, Johannes and Le Bras, Arthur-C\'esar},
     TITLE = {A {F}ourier transform for {B}anach-{C}olmez spaces},
   JOURNAL = {J. Eur. Math. Soc. (JEMS)},
  FJOURNAL = {Journal of the European Mathematical Society (JEMS)},
    VOLUME = {27},
      YEAR = {2025},
    NUMBER = {9},
     PAGES = {3651--3712},
      ISSN = {1435-9855,1435-9863},
   MRCLASS = {11F85 (11S37 14G20 14G22 14G45)},
  MRNUMBER = {4939523},
       DOI = {10.4171/jems/1480},
}

@misc{haine2022nonabelianbasechangebasechangecoefficients,
      title={From nonabelian basechange to basechange with coefficients}, 
      author={Peter J. Haine},
      year={2022},
      eprint={2108.03545},
      archivePrefix={arXiv},
      primaryClass={math.CT},
}

@misc{Tak24,
      title={Second adjointness and cuspidal supports at the categorical level}, 
      author={Yuta Takaya},
      year={2024},
      eprint={2408.04582},
      archivePrefix={arXiv},
      primaryClass={math.RT},
}

@misc{Han24,
  url = {http://www.davidrenshawhansen.net/Beijing.pdf},
  author = {Hansen, David},
  title = {Beijing notes on the categorical local {L}anglands conjecture},
  year = {2024}
}

@misc{Zou24,
      title={The categorical form of Fargues' conjecture for tori}, 
      author={Konrad Zou},
      year={2024},
      eprint={2202.13238},
      archivePrefix={arXiv},
      primaryClass={math.RT},
}

@article {FW25,
    AUTHOR = {Feng, Tony and Wang, Jonathan},
     TITLE = {Geometric {L}anglands duality for periods},
   JOURNAL = {Geom. Funct. Anal.},
  FJOURNAL = {Geometric and Functional Analysis},
    VOLUME = {35},
      YEAR = {2025},
    NUMBER = {2},
     PAGES = {463--541},
      ISSN = {1016-443X,1420-8970},
   MRCLASS = {14D24 (11G)},
  MRNUMBER = {4880204},
       DOI = {10.1007/s00039-025-00702-4},
}

@article {Hoc61,
    AUTHOR = {Hochschild, Gerhard},
     TITLE = {Cohomology of algebraic linear groups},
   JOURNAL = {Illinois J. Math.},
  FJOURNAL = {Illinois Journal of Mathematics},
    VOLUME = {5},
      YEAR = {1961},
     PAGES = {492--519},
      ISSN = {0019-2082},
   MRCLASS = {14.50 (17.30)},
  MRNUMBER = {130901},
MRREVIEWER = {W.\ T.\ van Est},
}

@misc{FYZ23,
      title={Modularity of higher theta series I: cohomology of the generic fiber}, 
      author={Tony Feng and Zhiwei Yun and Wei Zhang},
      year={2023},
      eprint={2308.10979},
      archivePrefix={arXiv},
      primaryClass={math.NT},
}

@misc{HI24,
      title={Dualizing complexes on the moduli of parabolic bundles}, 
      author={Linus Hamann and Naoki Imai},
      year={2024},
      eprint={2401.06342},
      archivePrefix={arXiv},
      primaryClass={math.NT}
}

@misc{ALB21,
      title={Averaging functors in {F}argues' program for ${GL}_n$}, 
      author={Johannes Ansch\"{u}tz and Arthur-C\'{e}sar Le Bras},
      year={2021},
      eprint={2104.04701},
      archivePrefix={arXiv},
      primaryClass={math.NT}
}

@misc{BZSV,
      title={Relative Langlands Duality}, 
      author={David Ben-Zvi and Yiannis Sakellaridis and Akshay Venkatesh},
      year={2024},
      eprint={2409.04677},
      archivePrefix={arXiv},
      primaryClass={math.RT},
      shorthand = {BZSV}, 
}

@article {Vie24,
    AUTHOR = {Viehmann, Eva},
     TITLE = {On {N}ewton strata in the {$B_{\rm dR}^+$}-{G}rassmannian},
   JOURNAL = {Duke Math. J.},
  FJOURNAL = {Duke Mathematical Journal},
    VOLUME = {173},
      YEAR = {2024},
    NUMBER = {1},
     PAGES = {177--225},
      ISSN = {0012-7094,1547-7398},
   MRCLASS = {11G18 (14G20 14M15)},
  MRNUMBER = {4728690},
       DOI = {10.1215/00127094-2024-0005},
}

@misc{FS24,
      title={Geometrization of the local Langlands correspondence}, 
      author={Laurent Fargues and Peter Scholze},
      year={2024},
      eprint={2102.13459},
      archivePrefix={arXiv},
      primaryClass={math.RT},
}

@book{SW20,
    AUTHOR = {Scholze, Peter and Weinstein, Jared},
     TITLE = {Berkeley lectures on {$p$}-adic geometry},
    SERIES = {Annals of Mathematics Studies},
    VOLUME = {207},
 PUBLISHER = {Princeton University Press, Princeton, NJ},
      YEAR = {2020},
     PAGES = {x+250},
   MRCLASS = {14G45 (14A15 14G22 14G35 14M15)},
  MRNUMBER = {4446467},
}

@misc{Sch17,
  doi = {10.48550/ARXIV.1709.07343},
  
  author = {Scholze, Peter},
  
  title = {Etale cohomology of diamonds},
  
  publisher = {arXiv},
  
  year = {2017},
  
  eprint={1709.07343},
  archivePrefix={arXiv},
  
  copyright = {arXiv.org perpetual, non-exclusive license}
}

@misc{Zhu21,
      title={Coherent sheaves on the stack of {L}anglands parameters}, 
      author={Xinwen Zhu},
      year={2021},
      eprint={2008.02998},
      archivePrefix={arXiv},
      primaryClass={math.AG}
}

@misc{DHKM20,
      title={Moduli of {L}anglands {P}arameters}, 
      author={Jean-Fran\c{c}ois Dat and David Helm and Robert Kurinczuk and Gilbert Moss},
      year={2020},
      eprint={2009.06708},
      archivePrefix={arXiv},
      primaryClass={math.NT}
}

@misc{Man22,
      title={A $p$-Adic 6-Functor Formalism in Rigid-Analytic Geometry}, 
      author={Lucas Mann},
      year={2022},
      eprint={2206.02022},
      archivePrefix={arXiv},
      primaryClass={math.AG},
}

@book {GR17I,
    AUTHOR = {Gaitsgory, Dennis and Rozenblyum, Nick},
     TITLE = {A study in derived algebraic geometry. {V}ol. {I}.
              {C}orrespondences and duality},
    SERIES = {Mathematical Surveys and Monographs},
    VOLUME = {221},
 PUBLISHER = {American Mathematical Society, Providence, RI},
      YEAR = {2017},
     PAGES = {xl+533},
      ISBN = {978-1-4704-3569-1},
   MRCLASS = {14F05 (18D05 18G55)},
  MRNUMBER = {3701352},
MRREVIEWER = {Adrian\ Langer},
       DOI = {10.1090/surv/221.1},
}

@article {Sum75,
    AUTHOR = {Sumihiro, Hideyasu},
     TITLE = {Equivariant completion. {II}},
   JOURNAL = {J. Math. Kyoto Univ.},
  FJOURNAL = {Journal of Mathematics of Kyoto University},
    VOLUME = {15},
      YEAR = {1975},
    NUMBER = {3},
     PAGES = {573--605},
      ISSN = {0023-608X},
   MRCLASS = {14L99},
  MRNUMBER = {387294},
MRREVIEWER = {Vladimir\ L.\ Popov},
       DOI = {10.1215/kjm/1250523005},
}

@misc{Ham22,
      title={A Jacobian Criterion for Artin $v$-stacks}, 
      author={Linus Hamann},
      year={2022},
      eprint={2209.07495},
      archivePrefix={arXiv},
      primaryClass={math.AG},
}

@misc{Sch25,
      title={Geometrization of the local Langlands correspondence, motivically}, 
      author={Peter Scholze},
      year={2025},
      eprint={2501.07944},
      archivePrefix={arXiv},
      primaryClass={math.AG},
}

@misc{Sch24,
      title={Berkovich Motives}, 
      author={Peter Scholze},
      year={2024},
      eprint={2412.03382},
      archivePrefix={arXiv},
      primaryClass={math.AG},
}

@misc{HHS24,
  author = {Hamann, Linus and Hansen, David and Scholze, Peter},
  title = {Geometric Eisenstein series I: finiteness theorems},
  year={2024},
  eprint={2409.07363},
  archivePrefix={arXiv},
  primaryClass={math.NT}
}

@misc{HM24,
      title={6-Functor Formalisms and Smooth Representations}, 
      author={Claudius Heyer and Lucas Mann},
      year={2024},
      eprint={2410.13038},
      archivePrefix={arXiv},
      primaryClass={math.CT},
}

@misc{ALM24,
      title={A 6-functor formalism for solid quasi-coherent sheaves on the Fargues-Fontaine curve}, 
      author={Johannes Ansch\"{u}tz and Arthur-C\'{e}sar Le Bras and Lucas Mann},
      year={2024},
      eprint={2412.20968},
      archivePrefix={arXiv},
      primaryClass={math.AG},
}

@incollection{Tate1967,
  author    = {Tate, John T.},
  title     = {Fourier analysis in number fields, and {H}ecke's zeta-functions},
  booktitle = {Algebraic Number Theory},
  editor    = {Cassels, J. W. S. and Fr{\"o}hlich, A.},
  publisher = {Thompson Book Co.},
  address   = {Washington, D.C.},
  year      = {1967},
  pages     = {305--347}
}

@book {HTT,
    AUTHOR = {Lurie, Jacob},
     TITLE = {Higher topos theory},
    SERIES = {Annals of Mathematics Studies},
    VOLUME = {170},
 PUBLISHER = {Princeton University Press, Princeton, NJ},
      YEAR = {2009},
     PAGES = {xviii+925},
      ISBN = {978-0-691-14049-0; 0-691-14049-9},
   MRCLASS = {18-02 (18B25 18E35 18G30 18G55 55U40)},
  MRNUMBER = {2522659},
MRREVIEWER = {Mark\ Hovey},
       DOI = {10.1515/9781400830558},
}
